\newif\ifdraft
\draftfalse
\ifdraft
\documentclass[a4paper, draft]{amsart}
\else
\documentclass[a4paper]{amsart}
\fi
\usepackage{amsmath,amsthm,amssymb, times, tikz-cd, comment}
\usepackage[british]{babel}
\usepackage[utf8]{inputenc}
\usepackage[normalem]{ulem}
\usepackage{enumerate}
\usepackage{extarrows}
\usepackage{hyperref}
\ifdraft
\newcommand{\todo}[1]{\footnote{TODO #1}}
\else
\newcommand{\todo}[1]{}
\fi
\title{Gluing and deformations of asymptotically cylindrical special Lagrangians}
\author{Tim Talbot}
\address{DPMMS\\Centre for Mathematical Sciences\\Wilberforce Road\\Cambridge\\CB3 0WB}
\email{T.J.Talbot@maths.cam.ac.uk}
\thanks{This work was supported by the UK Engineering and Physical Sciences Research Council (EPSRC) grant EP/H023348/1 for the University of Cambridge Centre for Doctoral Training, the Cambridge Centre for Analysis.}
\date{}
\theoremstyle{definition}
\newtheorem{defin}{Definition}[section]
\newtheorem{convention}[defin]{Convention}

\newtheorem{hypothesis}[defin]{Hypothesis}
\newtheorem{condition}[defin]{Condition}
\theoremstyle{plain}
\newtheorem{lem}[defin]{Lemma}
\newtheorem{prop}[defin]{Proposition}
\newtheorem{thm}[defin]{Theorem}
\newtheorem{cor}[defin]{Corollary}

\newtheorem*{thma}{Theorem A}
\newtheorem*{thmaone}{Theorem A1}
\newtheorem*{thmb}{Theorem B}

\theoremstyle{remark}
\newtheorem*{rmk}{Remark}

\numberwithin{equation}{section}

\newcommand{\vol}{\mathrm{vol}}

\newcommand{\ddth}{\frac{\partial}{\partial \theta}}
\newcommand{\ddt}{\frac{\partial}{\partial t}}
\newcommand{\dds}{\frac{\partial}{\partial s}}
\newcommand{\ddr}{\frac{\partial}{\partial r}}
\newcommand{\C}{{\mathbb{C}}}
\newcommand{\R}{{\mathbb{R}}}
\newcommand{\xs}{{\mathsf{xs}}}
\newcommand{\tr}{\mathrm{tr}}

\newcommand{\la}{\langle}
\newcommand{\ra}{\rangle}

\newcommand{\rel}{\mathrm{rel}}

\newcommand{\F}{{\mathcal{F}}}
\newcommand{\cO}{{\mathcal{O}}}
\newcommand{\K}{{\mathcal{K}}}

\newcommand{\U}{{\mathcal{U}}}

\newcommand{\cpt}{{\mathrm{cpt}}} 

\newcommand{\SLing}{{\mathrm{SLing}}} 

\newcommand{\two}{\mathrm{I\!I}}

\newcommand{\Spin}{\mathrm{Spin}}
\newcommand{\Gr}{\mathrm{Gr}}
\DeclareMathOperator{\Vol}{Vol}
\DeclareMathOperator{\spn}{span}
\DeclareMathOperator{\harmpt}{harmpt} 
\DeclareMathOperator{\npt}{normpt}
\DeclareMathOperator{\hpt}{harmpt}
\newcommand{\id}{{\mathrm{id}}} 
\newcommand{\Sym}{{\mathrm{Sym}}}

\renewcommand{\Re}{\operatorname{Re}}
\renewcommand{\Im}{\operatorname{Im}}

\newcommand{\SUn}{{\mathcal{SU}_n}}
\newcommand{\ie}{i.\,e.\ }
\newcommand{\eg}{e.\,g.\ }
\ifdraft
\usepackage{xcolor}
\usepackage[colorlinks]{hyperref}

\usepackage{auxhook}
\newcounter{labelknownref}
\renewcommand*{\thelabelknownref}{\the\value{labelknownref}}
\makeatletter
\AddLineBeginAux{%
  \string\providecommand\string\LabelKnown[2]{}%
}
\newcommand*{\LabelKnown}[2]{%
  \expandafter\xdef\csname lkr@#2\endcsname{%
    \@ifundefined{r@#1}{0}{1}%
  }%
}

\usepackage{nameref}
\usepackage{letltxmacro}
\AtBeginDocument{%
  \LetLtxMacro\myorgref\ref
  \DeclareRobustCommand*{\ref}[1]{%
    \begingroup
      \stepcounter{labelknownref}%
      \if@filesw
        \protected@write\@auxout{}{%
          \string\LabelKnown{#1}{\thelabelknownref}%
        }%
      \fi 
      \if\csname lkr@\thelabelknownref\endcsname 1%
        \hypersetup{linkcolor=blue}%
        \myorgref{#1}\textsuperscript{\textcolor{blue}{(ok)}}%
      \else
        \if\csname lkr@\thelabelknownref\endcsname 0%
          \hypersetup{linkcolor=red}%
          \myorgref{#1}\textsuperscript{\textcolor{red}{(forward reference!)}}%
        \else
          \hypersetup{linkcolor=orange}%
          \myorgref{#1}\textsuperscript{\textcolor{orange}{(unknown)}}%
        \fi
      \fi  
    \endgroup
  }%
}   
\makeatother
\fi

\begin{document}
\begin{abstract}
We study gluings of asymptotically cylindrical special Lagrangian submanifolds in asymptotically cylindrical Calabi--Yau manifolds. We prove both that there is a well-defined gluing map, and, after reviewing the deformation theory for special Lagrangians, prove that this gluing map defines a local diffeomorphism from matching pairs of deformations of asymptotically cylindrical special Lagrangians to deformations of special Lagrangians. We also give some examples of asymptotically cylindrical special Lagrangian submanifolds to which these results apply.
\end{abstract}
\maketitle
\section{Introduction}
Calabi--Yau manifolds have various families of distinguished submanifolds. The most obvious family is the complex submanifolds, since Calabi--Yau manifolds are complex. There are other families that are less well understood. In this paper, we are concerned with special Lagrangian submanifolds, first introduced by Harvey and Lawson in \cite{harveylawson} in 1982. They used special Lagrangian submanifolds as an example of their notion of a calibrated submanifold. As calibrated submanifolds, special Lagrangians are minimal submanifolds; in fact, they are volume-minimising in their homology class. Rather than working with the calibration, we shall regard special Lagrangians as submanifolds $L$ on which the imaginary part of the holomorphic volume form $\Im \Omega$ and the K\"ahler form $\omega$ vanish. 

We shall prove two theorems concerning asymptotically cylindrical special Lagrangian submanifolds, one of which has a variant extending to a slightly more general setting. First we prove a gluing theorem
\begin{thma}[Theorem \ref{slperturbthm}]
Let $M_1$ and $M_2$ be asymptotically cylindrical Calabi--Yau manifolds, whose limit cylinders can be identified so that we can construct a ``connected sum" $M^T$, and that the limits of the Calabi--Yau structures agree under this identification. Suppose that $L_1$ and $L_2$ are asymptotically cylindrical special Lagrangian submanifolds and their limit cylinders agree under this identification. By cutting off, we construct a pair of closed forms $(\Omega^T, \omega^T)$ and a submanifold $L_0^T$. Suppose that $(\Omega^T, \omega^T)$ can be perturbed to give a Calabi--Yau structure, only changing the cohomology classes up to scaling. Then we may perturb $L_0^T$ to form a special Lagrangian $L^T$ in $M^T$ with this Calabi--Yau structure. 
\end{thma}
We also sketch an additional argument that extends Theorem A to
\begin{thmaone}[Theorem \ref{thm:finalslperturbthm}]
Let $M$ be a Calabi--Yau manifold and let $L$ be a closed submanifold. If $\Im \Omega|_L$ and $\omega|_L$ are sufficiently small depending on $M$, $L$ and the inclusion, and are exact, then $L$ can be perturbed to a special Lagrangian $L'$. 
\end{thmaone}

In the second part of the paper, we prove 
\begin{thmb}[Theorem \ref{speclaggluinglocaldiffeo}]
Let $M_1$, $M_2$, $L_1$, $L_2$, $M^T$ and $L^T$ be as in Theorem A. For a deformation of $(L_1, L_2)$ as a pair of special Lagrangians whose limits are identified, we may define a gluing map. Moreover, the space of all such pairs is a manifold around $(L_1, L_2)$, and this gluing map defines a smooth map from this manifold to the special Lagrangian deformations of $L^T$. This is a local diffeomorphism; in particular, it constructs an open subset of the special Lagrangians on $M^T$ around $L^T$.
\end{thmb}

We shall suppose throughout that our ambient Calabi--Yau manifolds are connected, but that the special Lagrangian submanifolds need not be, to allow for the possibility of gluing finitely many special Lagrangians with the same total collection of cross-sections. 
On physical grounds, it is conjectured that there is a duality between Calabi--Yau threefolds known as ``mirror symmetry" (see for example the review of Gross \cite{grosssyz}). One of the formulations of this given in \cite{grosssyz} is that there should be an isomorphism between the special Lagrangian submanifolds of a Calabi--Yau manifold equipped with a flat $U(1)$ bundle and the complex submanifolds of its mirror pair equipped with a holomorphic line bundle. Specifically, a conjecture of Strominger, Yau and Zaslow \cite{syzconjecture} says that a Calabi--Yau manifold and its mirror pair both admit special Lagrangian torus fibrations and the generic tori are dual (in the sense that one is an appropriately generalised first cohomology group of the other). This gives another reason to study special Lagrangian submanifolds. It is obvious that a torus would be obtained from Theorem A if the special Lagrangians we glue are topologically $\R \times T^{n-1}$, and then Theorem B would yield an interaction between deformations of these tori and deformations of the original $\R \times T^{n-1}$ special Lagrangians. 

The deformation theory of calibrated submanifolds in general and special Lagrangian submanifolds in particular was initiated by McLean \cite{mclean}. One of the key ideas in the special Lagrangian case is the use of the K\"ahler form to translate normal vector fields on $L$ into one-forms on $L$, though this of course works more generally for Lagrangian submanifolds of symplectic manifolds. 

If $M$ is an asymptotically cylindrical Calabi--Yau manifold, and $L$ is an asymptotically cylindrical special Lagrangian submanifold, then the induced metric on $L$ is itself asymptotically cylindrical. Hence, deformation results follow by combining the asymptotically cylindrical Laplace--Beltrami theory of Lockhart \cite{lockhart} with the McLean results. This was carried out by Salur and Todd \cite{salurtodd}. A slightly different deformation result for minimal Lagrangian submanifolds with boundary constrained to lie in an appropriate symplectic submanifold has also been obtained by Butscher \cite{butscherdeformations}. 

Joyce has written a series of papers \cite{joycecones1, joycecones2, joycecones3, joycecones4, joycecones5} on special Lagrangians with conical singularities in a fixed compact (generalised) Calabi--Yau manifold. For instance, he provides desingularisations of such singularities by gluing in a sufficiently small asymptotically conical submanifold of $\C^n$. This gluing argument uses the Lagrangian Neighbourhood Theorem to ensure that all submanifolds we deal with are Lagrangian. However, once we have a compact Lagrangian submanifold that is close to special Lagrangian he gives a general result on perturbing it to become special Lagrangian \cite[Theorem 5.3]{joycecones3}. Pacini extended this result to the case of asymptotically conical Lagrangian submanifolds in $\C^n$ given by patching together special Lagrangians with conical singularities and asymptotically conical ends (see \cite[Theorem 6.3]{pacini}). In the asymptotically cylindrical setting, where there is a change of Calabi--Yau structure arising from the Calabi--Yau gluing, it is much harder to remain Lagrangian and so the analysis is rather different. 

There are also other gluing constructions of submanifolds satisfying appropriate partial differential equations: for instance, Lotay has considered similar desingularisation problems for coassociative submanifolds of $G_2$ manifolds (see for example \cite{lotay}) and Butscher constructs contact stationary Legendrian submanifolds by a gluing method in \cite{butscherequivariantgluing}. 

We now outline the content of this paper. As a preliminary, in section \ref{sec:basicanalyticsetup}, we will first describe asymptotically cylindrical submanifolds and discuss the restriction map of forms in section \ref{sec:basicanalyticsetup}. We then discuss patching in the purely Riemannian case. We also briefly describe patching methods for asymptotically cylindrical submanifolds. 

We give our definitions and some examples and then review (and correct slightly, altering the dimension of the asymptotically cylindrical deformation space in Theorem \ref{acyldeformationtheorem}) the deformation theory of McLean and Salur--Todd in subsection \ref{sec:defsandefs}. We then turn to gluing. In Hypothesis \ref{hyp:ambientgluing}, we assume a result on the gluing of Calabi--Yau manifolds that generalises to higher $n$ the one obtained by the author in \cite{su3g2story}, and then prove in Theorem \ref{slperturbthm} (Theorem A) that asymptotically cylindrical special Lagrangian submanifolds can be glued. This follows by considering the same argument as in the deformation case. The main difficulty is to find a bound on the inverse of the linearisation, and this is done by showing that the inverse of the linearisation depends continuously on the Calabi--Yau structure and so by comparing with the original structures, we can assume we are just working with $d + d*$ as in the deformation case. Extending this slightly, by finding an auxiliary Calabi--Yau structure around any given nearly special Lagrangian submanifold, leads to Theorem \ref{thm:finalslperturbthm} (Theorem A1), that any nearly special Lagrangian submanifold can be perturbed to a special Lagrangian. 

In section \ref{sec:dsling}, we find the derivative of the map perturbing submanifolds to special Lagrangians that we have just constructed, subject to an analytic condition. In order to do this, we first give a careful discussion of how to identify normal vector fields on nearby pairs of submanifolds and how these give derivatives  for various natural maps on the ``manifold of submanifolds". We then briefly discuss how corresponding analysis applies to taking the harmonic parts of normal vector fields (that is, the part corresponding to a harmonic one-form), and hence find the derivative of the map making things special Lagrangian from the previous section. For instance, we show that given two curves of submanifolds $L_s$ and $L'_s$, the curve of normal vector fields such that $\exp_{v_s}(L_s) = L'_s$ is smooth, and give an expression for the tangent to $v_s$ at zero in terms of the tangents to the curves $L_s$ and $L'_s$ (for this expression, see Proposition \ref{combinedtransfer}). This material is inspired by Palais \cite{palaisfoundations} and Hamilton \cite{hamiltonnashmoser}, but we could not find it in the literature. 

In section \ref{sec:dpatching}, we proceed to the same kind of analysis of the patching of submanifolds. This is analytically more complicated, but because there is no need to worry about harmonic parts, it is conceptually simpler. 

Finally, we show in Theorem \ref{speclaggluinglocaldiffeo} (Theorem B) as the final part of section \ref{sec:opennessthm} that the gluing map of special Lagrangians so defined is a local diffeomorphism of moduli spaces. Intuitively, this is obvious, because the gluing map is just ``patch and become special Lagrangian" so its derivative should just be ``patch and become harmonic", and Nordstr\"om (in \cite[Theorem 3.1]{nordstromgluing}) says that this is an isomorphism, at least for long enough necks. 
 
\textbf{Acknowledgement.} The work in this paper is a slight extension of part of the author's PhD thesis \cite{mythesis}. He is indebted to his supervisor, Alexei Kovalev, for much useful advice. He would also like to thank the examiners, especially Dominic Joyce, for helpful comments. 
\section{Asymptotically cylindrical geometry}
\label{sec:basicanalyticsetup}
In this section, we make the basic Riemannian geometry and analytic definitions required for the rest of the paper, with particular emphasis on asymptotically cylindrical submanifolds. Throughout, all manifolds will be oriented. We make these definitions before defining special Lagrangians, as some of this material will be needed for our review of basic special Lagrangian theory in section \ref{sec:defsandefs}. The section falls into three parts. First, in subsection \ref{ssec:normsonsubmanifolds}, we define norms for objects on manifolds and submanifolds and briefly explain how these norms interact with the inclusion. Specifically, we state Theorem \ref{localrestrictionthm}, that (locally) if the second fundamental form of $L$ in $M$ is bounded in $C^{k-1}$, then the restriction maps are $C^k$ bounded. In subsection \ref{ssec:basicacylmanifolds}, we give definitions of asymptotically cylindrical manifolds and define some appropriate approximate gluing or patching maps. We briefly discuss the Laplace--Beltrami operator on the resulting manifolds. We then proceed in subsection \ref{ssec:basicacylsubmanifolds} to define asymptotically cylindrical submanifolds (Definition \ref{defin:acylsubmfd}) and give their basic properties. Finally, we describe an approximate gluing or patching map of asymptotically cylindrical submanifolds in Definition \ref{defin:newsubmanifoldapproxgluing}, and describe how this map interacts with the restriction maps of forms. 
\subsection{Norms}
To define norms, we shall use the notion of a jet bundle. We shall use the following basic results. 
\begin{prop}[\hspace{1sp}{\cite[ch. 2]{palaisfoundations}}, {\cite[Lemma 2.1]{jafarpourlewis}}]
\label{newjetbundles}
Let $M$ be a Riemannian manifold and let $E$ be a vector bundle over $M$. There exists a vector bundle $J^s(E)$, called the \emph{jet bundle}, such that at every point $p \in M$, the fibre $J^s(E)_p$ is given by the quotient of smooth sections of $E$ by the sections vanishing to order $s$ at $p$, and whenever $\sigma$ is a smooth section of $E$, we can write $j^k(\sigma)$ a smooth section of $J^s(E)$ given by the equivalence class of $\sigma$ at every point. We call the operation of getting $j^s(\sigma)$ from $\sigma$ \emph{jet prolongation}. 

Moreover, suppose we have a linear connection on $E$. Then using the Levi-Civita connection we have induced connections on tensor products. Hence, we may define a bundle map
\begin{equation}
\label{eq:jetbundlestructure}
\begin{tikzcd}[row sep=0pt]
J^s M \ar{r}& \oplus_{l=0}^s S^l (T^*M) \otimes E, \\
{[\sigma]} \ar[mapsto]{r}& (\sigma_p, \Sym (\nabla \sigma)_p, \ldots, \Sym (\nabla^k \sigma)_p).
\end{tikzcd}
\end{equation}
where $S^l$ is the symmetric product and $\Sym$ are the symmetrisation operators. \eqref{eq:jetbundlestructure} is a vector bundle isomorphism.
\end{prop}
A metric and connection on $J^sE$ follow immediately. We note also that if $F$ is a smooth fibre bundle, Palais \cite[ch. 15]{palaisfoundations} shows that we can also define a jet bundle $J^s(F)$. Moreover, this construction is functorial, so that if $F$ is a subbundle of a bundle associated to the tangent bundle, $J^s(F)$ is a smooth subbundle of $J^s(E)$. This means we have a metric and connection on $J^s(F)$ also. 

We may then define
\begin{defin}
\label{defin:holdernorms}
Let $(M, g)$ be a Riemannian manifold. Let $s$ be a non-negative integer, and $\mu \in [0, 1)$. Let $E$ be a bundle associated to the tangent bundle, so that the Riemannian metric and Levi-Civita connection define a metric and connection on the bundle $E$. Given a section $\sigma$ of $E$, define
\begin{equation}
\|\sigma\|_{C^{0, \mu}} = \sup_{ x \in M}  \left(g(\sigma_x, \sigma_x)\right)^{\frac12} + \sup_{\substack{x, y \in M \\ d(x, y) < \delta}} \frac{|\sigma_x - \sigma_y| } {d(x, y)^\mu},
\end{equation}
where $\delta$ is the injectivity radius and $|\sigma_x - \sigma_y|$ is given by parallel transporting $\sigma_x$ with $\nabla$ from the fibre $E_x$ to the fibre $E_y$ along the geodesic from $x$ to $y$ and then measuring the difference. Note that since parallel transport is an isometry, this is symmetric in $x$ and $y$. 

Since $J^s(E)$ is also a bundle associated to the tangent bundle, we then define
\begin{equation}
\|\sigma\|_{C^{s, \mu}} = \|j^s \sigma\|_{C^{0, \mu}},
\end{equation}
where $j^s \sigma$ is the jet prolongation from Proposition \ref{newjetbundles}.  

If $\mu = 0$, we just write $C^s$. 
\end{defin}
This definition is taken from Joyce \cite[p.5]{joycebook}. Note that if $M$ is not compact, smooth forms need not have finite $C^{s, \mu}$ norm for any $s$ and $\mu$. As we will mostly be working with forms on compact manifolds and asymptotically translation invariant forms on asymptotically cylindrical manifolds, this will not be a major issue. 

The following easy proposition explains why $C^s$ norms are convenient to work with. 
\begin{prop}
\label{bundlemapsckbounded}
Let $f: E \to F$ be a smooth bundle map between fibre bundles over a manifold $M$ that are subbundles of vector bundles associated to the tangent bundle. For each open subset $U$ in $M$ with compact closure, open subset $V$ of $E|_U$ with compact closure in $E$, constant $K$, and positive integer $k$,  we have a constant $C_{k,  K, V}$ such that whenever the sections $\sigma_1$ and $\sigma_2$ of $E|_U$ lie in $V$, and have derivatives up to order $s$ bounded by $K$, 
\begin{equation}
\|f(\sigma_1)|_U - f(\sigma_2)|_U\|_{C^s} \leq C_{s, K, V} \|\sigma_1|_U - \sigma_2|_U\|_{C^s}.
\end{equation}
\end{prop}

\label{ssec:normsonsubmanifolds}
We now pass to a submanifold $L$ of the Riemannian manifold $(M, g)$. $(L, g|_L)$ is itself Riemannian: write $\nabla^L$ for the induced Levi-Civita connection. It follows that we have natural $C^{k, \mu}$ norms on sections of bundles associated to the tangent bundle $TL$ just as in Definition \ref{defin:holdernorms}. We will also need norms on sections of the normal bundle $\nu_L$. The metric $g$ defines a metric on this bundle: in order to define a $C^k$ norm on its sections, we need a connection. 

The natural connection on $\nu_L$ is given by taking the normal part after applying the Levi-Civita connection on $M$. We again call this connection $\nabla^L$. Combining it with the Levi-Civita connection $\nabla^L$ on $T^*L$ yields connections on tensor products by the Leibniz rule, so we may immediately extend the $C^s$ norms from Definition \ref{defin:holdernorms} to normal vector fields. Similarly, we may use $\nabla^L$ to define a $C^s$ norm on the second fundamental form $\two$ of $L$ in $M$, as this is a section of $T^*L \otimes T^*L \otimes \nu_L$. 

We now pass to the relation between norms on $L$ and $M$. Suppose $\alpha$ is a differential form on $M$ with $\|\alpha\|_{C^{k}}$ small. We would like to know that $\|\alpha|_L\|_{C^{k}}$ is also small. It can be shown that this holds locally given a local bound on the second fundamental form. More precisely, we have
\begin{thm}
\label{localrestrictionthm}
Let $L$ be an immersed submanifold of $(\R^n, g)$, where the metric $g$ need not be Euclidean. If the second fundamental form $\two$ has finite $C^{k-1}(L)$ norm , then for each $p$ there exists $C_p$ such that
\begin{equation}
\|\alpha|_L\|_{C^k(L)} \leq C_p \|\alpha\|_{C^k(M)} 
\end{equation}
for every $p$-form $\alpha$ on $M$; moreover, $C_p$ depends only on $p$ and the $C^{k-1}(L)$ bound on $\two$. 
\end{thm}
This result is not immediately apparent in the literature, and proving it in full generality is somewhat complicated. The idea, however, is easy. Since everything is multilinear, it suffices to work with one-forms, and since we have a Riemannian metric, we can regard $T^*L$ as a subbundle of $T^*M|_L$. The restriction map corresponds to orthogonal projection to this subbundle. The case $k=0$ follows immediately. For higher $k$, we know essentially by the Gauss-Weingarten formulae that derivatives on $L$ are just the $T^*L$-components of derivatives on $M$. Hence, for $k=1$ for instance, we obtain that the first derivative of $\alpha|_L$ is the $T^*L$-component of the first derivative of the $T^*L$-component of $\alpha$. But the difference between this and the $T^*L$-component of the first derivative $\alpha$ is the tangent part of the derivative of a normal part, and so can be controlled essentially by the second fundamental form. For $k>1$, there are simply more terms to deal with. 

Moreover, this also shows that the norms defined on forms on $L$ by using the connection on $M$ rather than $\nabla^L$ in the isomorphism of Proposition \ref{newjetbundles} are Lipschitz equivalent to the standard $C^k$ norms. The Lipschitz constant depends only on the $C^{k-1}$ norm of the second fundamental form of $L$ in $M$.
\subsection{Asymptotically cylindrical manifolds and their patching}
\label{ssec:basicacylmanifolds}
In this subsection, we briefly review the definition of asymptotically cylindrical manifolds and asymptotically translation invariant objects, appropriate norms, and patching maps for such manifolds and closed forms on them. This is well-known; for instance, the ideas are present in Kovalev \cite{kovalevtwistedconnected} and appear more definitely in Nordstr\"om \cite{nordstromgluing}. We also discuss the Laplace--Beltrami operator on the result of a patching in Theorem \ref{laplacelowerbound}, and explain that it can, up to harmonic forms, be bounded below in terms of the length of the neck created by the patching. Finally, we extend Nordstr\"om's work in \cite{nordstromgluing} to give a slightly more concrete statement on the gluing of harmonic forms in Proposition \ref{harmonicgluinglowerbound}. 

We begin by defining an asymptotically cylindrical manifold. 
\begin{defin}
\label{defin:acylmanifold}
A Riemannian manifold $(M, g)$ is said to be \emph{asymptotically cylindrical with rate $\delta>0$} if there exists a compact manifold with boundary $M^\cpt$, whose boundary is the compact manifold  $N$, a metric $g_N$ on $N$, and constants $C_r$ satisfying the following. Firstly, $M$ is the union of $M^\cpt$ and $N \times [0, \infty)$, where these two parts are identified by the obvious identification between $N = \partial M^\cpt$ and $N \times \{0\}$. Secondly, we have the estimate 
\begin{equation}
\label{eq:exponentialdecayestimate}
|\nabla^r (g|_{N \times [0, \infty)} - g_N - dt^2)| < C_re^{-\delta t}
\end{equation}
for every $r = 0, 1, \ldots$, where $t$ is the coordinate on $[0, \infty)$ extended to a global smooth function on $M$, $\nabla$ is the Levi-Civita connection induced by $g$, and $|\cdot|$ is the metric induced by $g$ on the appropriate space of tensors. $(M,g)$ is said to be \emph{asymptotically cylindrical} if it is asymptotically cylindrical with rate $\delta$ for some $\delta >0$. 

We refer to the cylindrical part $N \times [0, \infty)$ as the \emph{end(s)} of $M$, and we write $\tilde g$ for the cylindrical metric $g_N + dt^2$; $(M, g)$ is said to be \emph{(eventually) cylindrical} if $g = \tilde g$ for $t$ large enough. 
\end{defin}
The analogous sections of vector bundles are said to be asymptotically translation invariant
\begin{defin}
Suppose that $M$ is an asymptotically cylindrical manifold. Given a bundle $E$ associated to the tangent bundle over $M$, a section $\tilde \alpha$ of $E|_N$ extends to a section of $E|_{N \times [0, \infty)}$ by extending parallel in $t$. A section $\alpha$ of $E$ is then said to be \emph{asymptotically translation-invariant (with rate $0<\delta'<\delta$)} if there is a section $\tilde \alpha$ of $E|_N$ and \eqref{eq:exponentialdecayestimate} holds (with $\delta'$) for $|\nabla^r (\alpha -\tilde\alpha)|$ for $t>T$. In general, given an asymptotically translation invariant section $\alpha$, we will write $\tilde \alpha$ for its limit in this sense. 
\end{defin}
If $\alpha$ is asymptotically translation invariant with $\tilde \alpha = 0$, we say $\alpha$ is exponentially decaying; this just means that \eqref{eq:exponentialdecayestimate} holds for $|\nabla^r \alpha|$ itself. We will need norms adapted to asymptotically translation invariant forms. Specifically, we define a $C^s_{\delta}$ norm on a subset of asymptotically translation invariant sections $\alpha$ of a bundle $E$ associated to the tangent bundle by setting
\begin{defin}
\label{defin:holderextendedweighted}
\begin{equation}
\label{eq:ckdnorm}
\|\alpha\|_{C^s_\delta} = \|(1-\psi)\alpha + \psi e^{\delta t}(\alpha -\tilde \alpha)\|_{C^s} + \|\tilde \alpha\|_{C^s},
\end{equation}
where the cutoff function $\psi(t)$ has $\psi(t) = 1$ for $t>2$, and $\psi(t) = 0$ for $t<1$. 
\end{defin}
The topology induced by \eqref{eq:ckdnorm} is called the extended weighted topology with weight $\delta$. If we restrict to exponentially decaying forms (that is, the closed subset with $\tilde \alpha = 0$), we obtain an ordinary weighted topology. In the same way we have a $C^{s, \mu}_\delta$ topology, and by taking the inverse limit we also have a $C^\infty_\delta$ topology. See also the discussion of Sobolev extended weighted topologies in, for example, \cite[section 3]{kovalevdeformations} and \cite[p.58ff.]{aps1}.

We now consider patching. We make the following definitions.
\begin{defin}
\label{defin:newgluingmanifolds}
Suppose that $M_1$ and $M_2$ are asymptotically cylindrical manifolds with ends, with corresponding cross-sections $N_1$ and $N_2$ and metrics $g_1$ and $g_2$ with limits $\tilde g_1$ and $\tilde g_2$. Suppose given an orientation-reversing diffeomorphism $F: N_1 \to N_2$. Combining $F$ with $t \mapsto 1-t$ induces an orientation-preserving map $F: N_1 \times  (0, 1) \to N_2 \times (0, 1)$. 

$(M_1, g_1)$ and $(M_2, g_2)$ \emph{match} if there exists $F$ as above such that $F^* g_2 = g_1$. We may then fix $T>1$, a ``gluing parameter", and let
\begin{equation}
M_i \supset M_i^{\tr (T+1)} := M_i^{\cpt} \cup N_i \times (0, T+1),
\end{equation}
Moreover, we may define additional metrics on $M_i$ by $\hat g_i = g_i + \psi_T(\tilde g_i - g_i)$, and restrict these to $M_i^{\tr (T+1)}$, where $\psi_T$ is a cutoff function with $\psi_T(t) = 1$ for $t\geq T$ and $\psi_T(t) = 0$ for $t< T-1$. 

Now $F$ defines an orientation-preserving diffeomorphism between $N_1 \times (T, T+1)$ and $N_2 \times (T, T+1)$, so we may use it to join together $M_1^{\tr (T+1)}$ with $M_2^{\tr (T+1)}$ and obtain the closed manifold $M^T$. Note that a subset of $M^T$ is parametrised by $(-T-\frac12, T+\frac12) \times N$, with $(-\frac12, \frac12) \times N$ corresponding to the identification region; we shall call this subset the \emph{neck} of $M^T$. On this identification region, $\hat g_1$ and $\hat g_2$ are identified, and so we can combine them to form the metric $g$ on $M^T$.
\end{defin}
We shall use the subsets $M_i^{\tr T'}$ for varying $T'$ in our analysis. Evidently, $M_i^{\tr T'}$ is always a subset of $M_i$ and defines a subset of $M^T$ if $T' \leq T+1$. 

With respect to the metrics $g$ on $M^T$, we have a family of Riemannian manifolds with the same local geometry but global geometry becoming increasingly singular. We nevertheless have
\begin{thm}
\label{laplacelowerbound}
Let $(M^T, g)$ be a compact manifold obtained by approximately gluing two asymptotically cylindrical manifolds with ends $M_1$ and $M_2$ as in Definition \ref{defin:newgluingmanifolds}. 

Let $k$ be a positive integer and $\mu \in (0, 1)$. Then if $\alpha$ is orthogonal to harmonic forms on $M^T$, then 
\begin{equation}
\label{eq:laplacelowerbound}
\|d\alpha\|_{C^{k, \mu}} + \|d^* \alpha\|_{C^{k, \mu}} \geq CT^l \|\alpha\|_{C^{k+1, \mu}},
\end{equation}
for some positive constants $C$ and $l$, which may both depend on the geometry of $M_1$ and $M_2$; $C$ may also depend on the regularity $k+\mu$. Moreover, the same holds if $g$ is only uniformly close (with all derivatives) to the glued metric. 
\end{thm}
We will not give a proof of this; fundamentally it follows by combining standard elliptic regularity results with the $L^2$ lower bound on exterior derivative given by Chanillo--Treves in \cite[Theorem 1.1]{chanillotreves} and its extension to $d^*$ in \cite[Theorem 1.2]{chanillotreves}. A similar argument using \cite[Theorem 1.3]{chanillotreves} gives an extension of Theorem \ref{laplacelowerbound} to the Laplacian.

We also require a method of patching closed differential forms, and make 
\begin{defin}[cf. Nordstr\"om {\cite[p.490]{nordstromgluing}}]
\label{defin:newformgluing}
Let $(M_1, g_1)$ and $(M_2, g_2)$ be matching asymptotically cylindrical manifolds as in Definition \ref{defin:newgluingmanifolds}. Suppose that $\alpha_1$ and $\alpha_2$ are asymptotically translation invariant closed $p$-forms on $M_1$ and $M_2$ respectively. The diffeomorphism $F$ (extended as in Definition \ref{defin:newgluingmanifolds}) induces a pullback map $F^*$ from the limiting bundle $\bigwedge^p T^*M_2 |_{N_2}$ to $\bigwedge^p T^*M_1|_{N_1}$. $\alpha_1$ and $\alpha_2$ are said to \emph{match} if $F^*\tilde \alpha_2 = \tilde \alpha_1$. 

Because we have convergence with all derivatives, the limits $\tilde \alpha_i$ are closed on $N_i$, and hence, treated as constants, on the end of $M_i$. As a decaying closed form, as in for instance Kovalev \cite[Proposition 1.1]{kovalevdeformations} or Melrose \cite[Proposition 6.13]{melrose}, we may choose $\eta_i$ with $\alpha_i - \tilde \alpha_i = d\eta_i$ on the end. Let $\psi_T$ be a cutoff function as in Definition \ref{defin:newgluingmanifolds} and define
\begin{equation}
\hat \alpha_i = \alpha_i - d(\psi_T \eta_i)
\end{equation}
on the end, and $\alpha_i$ off the end. Note that $\hat \alpha_i$ are closed. 
On the overlap of $M_1^{\tr (T+1)}$ and $M_2^{\tr (T+1)}$, $\hat\alpha_i = \tilde \alpha_i$, and so the two forms are identified by $F$. Thus they define a closed form $\gamma_T(\alpha_1, \alpha_2)$ on $M^T$.
\end{defin}

Finally, we note the following variant of the result of Nordstr\"om \cite[Theorem 3.1]{nordstromgluing}, giving a more quantified estimate.
\begin{prop}
\label{harmonicgluinglowerbound}
Let $M_1$ and $M_2$ be matching asymptotically cylindrical Riemannian manifolds. Suppose we have a metric $g$ on $M^T$ and constants $\epsilon$ and $C_k$ such that 
\begin{equation}
\label{eq:hgexpdecay}
\|g - g^T\|_{C^k} \leq C_k e^{-\epsilon T} 
\end{equation}
where $g^T$ is the metric on $M^T$ given by Definition \ref{defin:newgluingmanifolds}, and this norm is taken with respect to either metric. Given a pair of harmonic forms $\alpha_1$ and $\alpha_2$ on $M_1$ and $M_2$, matching as in Definition \ref{defin:newformgluing}, they are closed, and so we may define $\gamma_T(\alpha_1, \alpha_2)$. We may then obtain a harmonic form $\Gamma_T(\alpha_1, \alpha_2)$ on $M^T$ by taking the harmonic part. For $T$ sufficiently large, $\Gamma_T$ is an isomorphism. Moreover, we have an estimate
\begin{equation}
\label{eq:gluinglowerbound}
\|\Gamma_T(\alpha_1, \alpha_2)\|_{C^k} \geq C (\|\alpha_1\|_{C^k_\delta} + \|\alpha_2\|_{C^k_\delta}),
\end{equation}
where $C$ is independent of $T$ and the norms on the right hand side are the extended weighted norms of Definition \ref{defin:holderextendedweighted}. 
\end{prop}
\begin{proof}
Nordstr\"om shows in \cite[Theorem 3.1]{nordstromgluing} that this gluing map is an isomorphism except for a finite exceptional set of values of $T$. In particular, the map is a linear map between vector spaces of the same dimension; this can also be proved a little more directly. It will thus suffice to demonstrate \eqref{eq:gluinglowerbound}, which shows the map is injective. 

That this gluing map is an isomorphism follows from \cite[Theorem 3.1]{nordstromgluing}, using \cite[Corollary 5.13]{nordstromacyldeformations} to reduce to the simpler one-form version, so we only have to show the lower bound \eqref{eq:gluinglowerbound}. 

We show first that \eqref{eq:gluinglowerbound} holds when we replace $\Gamma_T$ with $\gamma_T$. Since $\gamma_T(\alpha_1, \alpha_2)$ is just given by identifying the cutoffs $\hat \alpha_1$ and $\hat \alpha_2$ it suffices to show that there exists $C$ independent of $T$ such that 
\begin{equation}
\|\hat \alpha_i\|_{C^k(M)} \geq C \|\alpha_i\|_{C^k_\delta}.
\end{equation}
We argue using the restriction to $\hat \alpha_i$ to the neighbourhood $M_i^{\tr 2}$ of $M_i^\cpt$ and the limit $\tilde \alpha_i$ of $\hat \alpha_i$. Note that this limit has two parts, a part with and a part without $dt$. If $T>3$, then $\hat \alpha_i|_{M_i^{\tr 2}} = \alpha_i|_{M_i^{\tr 2}}$; for any $T$, $\tilde \alpha_i$ is again the limit of $\hat \alpha_i$. We thus obtain for every $T>3$
\begin{equation}
\|\alpha_i|_{M_i^{\tr 2}}\|_{C^k(M_i^{\tr 2})} + \|\tilde \alpha_i\|_{C^k(N)} \leq 2\|\hat \alpha_i\|_{C^k(M)}.
\end{equation}
Hence, it suffices to prove that there is a constant $C$ such that 
\begin{equation}
\|\alpha_i\|_{C^k_\delta} \leq C(\|\alpha_i|_{M_i^{\tr 2}}\|_{C^k(M_i^{\tr 2})} + \|\tilde \alpha_i\|_{C^k(N)}).
\end{equation}
Since the space of harmonic one-forms is finite-dimensional and the map $\alpha_i \mapsto (\alpha_i|_{M_i^{\tr 2}}, \tilde \alpha_i)$ is linear, such a constant must exist provided that this map is injective. Suppose then that $\alpha_i|_{M_i^{\tr 2}} = 0$ and $\tilde \alpha_i = 0$. As in Definition \ref{defin:newformgluing}, we see that, working over $(0, \infty) \times N$, $\alpha_i = d\eta$ for some decaying $\eta$. Moreover, $\eta$ is closed on $(0, 2) \times N$, and as $\alpha_i$ is zero there, we may choose $\eta$ to be translation invariant there as well. Thus $\eta|_{(0, 2) \times N}$ defines a closed translation invariant form on the whole end. Subtracting this translation invariant form from $\eta$ and extending over $M^\cpt$ by zero, we find that $\alpha_i$ is the differential of an asymptotically translation invariant form on $M$. In particular, $\alpha_i$ must define the trivial class in $H^1(M_i)$. It follows from Nordstr\"om \cite[Theorem 5.9]{nordstromacyldeformations} that $\alpha_i$ is zero.  This proves that \eqref{eq:gluinglowerbound} holds when we replace $\Gamma_T$ by $\gamma_T$. 

It only remains to prove that $\Gamma_T(\alpha_1, \alpha_2)$ can be bounded below in terms of $\gamma_T(\alpha_1, \alpha_2)$; it will suffice to show that the part of $\gamma_T(\alpha_1, \alpha_2)$ orthogonal to harmonic forms satisfies an estimate like \eqref{eq:hgexpdecay}. We will use Theorem \ref{laplacelowerbound}: that in our setting, a form orthogonal to the harmonic forms satisfies \eqref{eq:laplacelowerbound}. It follows that if $d^*\gamma_T(\alpha_1, \alpha_2)$ satisfies an estimate like \eqref{eq:hgexpdecay}, so too does the part orthogonal to harmonic forms. We consider $d^* \hat\alpha_1$. We know that this is zero for $t< T-1$, as then $\hat \alpha_1$ coincides with $\alpha_1$ and the metric $g$ on $M^T$ coincides with $g_1$. It remains to consider the neck. Clearly, $\psi_{T}$ and its derivatives are bounded above independent of $T$; this implies that the glued metric $g^T$ and $\hat \alpha_1$ are exponentially close to $\tilde g$ and $\tilde \alpha$ here. We know that $\tilde \alpha$ is $\tilde g$-harmonic, and so we see that $d^* \hat \alpha_1$ satisfies \eqref{eq:hgexpdecay}. Hence $\Gamma_T(\alpha_1, \alpha_2) - \gamma_T(\alpha_1, \alpha_2)$ satisfies \eqref{eq:hgexpdecay} and the result follows. 
\end{proof}
\subsection{Asymptotically cylindrical submanifolds and their patching}
\label{ssec:basicacylsubmanifolds}
In this subsection, we describe what it will mean for a submanifold to be asymptotically cylindrical (Definition \ref{defin:acylsubmfd}), following Joyce--Salur \cite[Definition 2.9]{joycesalur} and Salur--Todd \cite[Definition 2.6]{salurtodd}. We show that an asymptotically cylindrical submanifold of an asymptotically cylindrical manifold is itself an asymptotically cylindrical manifold (globalising a variant of Theorem \ref{localrestrictionthm}) when equipped with the restricted metric. We then extend the patching or approximate gluing of subsection \ref{ssec:basicacylmanifolds} and give such a patching of matching submanifolds, and explain how this relates to the other patchings in that subsection. For instance, in cohomology, the patching map of forms commutes with the restriction map to a submanifold. 
\begin{defin}
\label{defin:acylsubmfd}
Let $(M, g)$ be an asymptotically cylindrical Riemannian manifold, with cross-section $N$. Let $L$ be a submanifold of $M$. $L$ is called \emph{asymptotically cylindrical (with cross-section $K$)} if there exists $R \in \R$ and an exponentially decaying normal vector field $v$ on $K \times (R, \infty)$ (in the sense that its inclusion and derivatives are exponentially decaying in the metric on $M$) such that $L$ is the union of a compact part with boundary $\exp_v(K \times \{R\})$ and $\exp_v(K \times (R, \infty)$. 

If $v \equiv 0$ far enough along $K \times (R, \infty)$, we shall call $L$ \emph{cylindrical}. 
\end{defin}
For notational clarity, we shall write the coordinate on $(R, \infty)$ as $r$. We note that $K$ need not be connected, even if $L$ is. 

We note the following
\begin{prop}[cf. Joyce--Salur{\cite[top of p.1135]{joycesalur}}]
\label{acylacyl}
Suppose that $(M, g)$ is an asymptotically cylindrical manifold with cross-section $N$, and that $L$ is an asymptotically cylindrical submanifold in the sense of Definition \ref{defin:acylsubmfd}. Then $L$, with its restricted metric, is itself an asymptotically cylindrical manifold.
\end{prop}
\begin{proof}[Sketch proof]
It suffices to work on each end. Since $v$ is exponentially decaying, $\exp_v^* g$ also defines an asymptotically cylindrical metric, so it suffices to suppose that $L$ is cylindrical. On the other hand, by an analogue to Theorem \ref{localrestrictionthm} and the fact the second fundamental form depends continuously on the inclusion and metric which both converge as $t \to \infty$, exponentially decaying tensors remain exponentially decaying on restriction, so it suffices to suppose the metric on $M$ is cylindrical. The result follows. 
\end{proof}
We may thus define extended weighted spaces of forms on asymptotically cylindrical submanifolds. 

Since a similar argument applies to the metric on the bundle $TM|_L$, we may then define spaces of exponentially decaying normal vector fields on asymptotically cylindrical submanifolds just as in subsection \ref{ssec:basicacylmanifolds}. Note that to define extended weighted spaces we need a notion of translation invariant normal vector fields, which we do not yet have. See Definition \ref{defin:newatinvf} below. 

We now extend the notion of approximate gluing to submanifolds. 
\begin{defin}
\label{defin:newsubmanifoldapproxgluing}
Let $M_1$ and $M_2$ be matching asymptotically cylindrical manifolds in the sense of Definition \ref{defin:newgluingmanifolds}, so that we have orientation-reversing $F: N_1 \to N_2$. Let $L_1$ and $L_2$ be asymptotically cylindrical submanifolds of $M_1$ and $M_2$; by definition, they have limits $K_1 \times (R_1, \infty)$ and $K_2 \times (R_2, \infty)$. Say that they \emph{match} if $F(K_1) = K_2$. 

Given such a pair, by definition 
\begin{equation}
L_i = L_i^\cpt \cup (\exp_v((R_i, \infty) \times K_i)).
\end{equation}
Consider the cutoff function $\varphi_{T}$, with $\varphi_T(t) = 0$ for $t\geq T$, and $\varphi_T(t) = 1$ for $t< T-1$. Note that $\varphi_{T}\circ\tilde\iota$ is a well-defined function on $K \times (R, \infty)$ with $\varphi_T \circ\iota(k, r) = 1$ for $r<T-1$ as $r = t\circ\tilde \iota$ for the cylindrical inclusion $\tilde \iota$. Hence for $T>R_i+1$
\begin{equation}
\hat L_i = L_i^\cpt \cup (\exp_{(\varphi_{T}\circ\tilde \iota) v}((R_i, \infty) \times K_i))
\end{equation}
is a submanifold. 

We want to be able to identify $\hat L_1$ and $\hat L_2$ over the identified regions $(T, T+1) \times N_i$ of $M_i$; that is, we want to show that $\hat L_i \cap (N_i \times (T, T+1))$ are identified by $F$. We show that $\hat L_i \cap (N_i \times (T, T+1)) = K \times (T, T+1)$, and so clearly is identified. For $r>T$, we have $\varphi_{T}\tilde\iota(k, r) = 0$, so that $\hat\iota(k, r) = \tilde\iota(k, r)$, and so $t\circ\hat\iota = r$. In particular, $t\circ\hat\iota(k, T) = T$ for all $k \in K_i$. 

Furthermore, note that by choosing $T$ sufficiently large $\hat \iota_* \ddr$ and $\tilde \iota_* \ddr$ can be chosen as close as we like. Since $\tilde \iota_* \ddr = \ddt$, it follows that $dt(\hat \iota_* \ddr)>0$ throughout $\hat L_i$. In particular, $t\circ\hat\iota(k, r)$ is always a strictly increasing function of $r$ for fixed $k$. It then follows that for $r<T$, $t \circ \hat \iota(k, r) <T$, and for $r >T+1$, $t\circ\hat\iota(k, r)>T+1$, so that $\hat L_i \cap (N \times (T, T+1)) = K_i \times (T, T+1)$. Since $L_i$ match, these are identified by $F$ and we can form the gluing of $\hat L_1$ and $\hat L_2$, a submanifold $L^T$ of $M^T$. 
\end{defin}

We may also consider $L_1$, $L_2$ and $L^T$ constructed in Definition \ref{defin:newsubmanifoldapproxgluing} as ambient manifolds. $L_1$ and $L_2$ match in the sense of Definition \ref{defin:newgluingmanifolds} and $L^T$ is their gluing, though not necessarily with parameter $T$. Suppose that a matching pair of closed forms on $L_1$ and $L_2$ is induced from a matching pair of closed forms on $M_1$ and $M_2$, \ie the pair of forms is $\alpha_1|_{L_1}$ and $\alpha_2|_{L_2}$. Then both $\gamma_T(\alpha_1, \alpha_2)|_{L^T}$ and $\gamma_T(\alpha_1|_{L_1}, \alpha_2|_{L_2})$ define forms on $L^T$. These forms need not be equal in general. 

Nevertheless, slightly weaker results will hold. To state these, we shall assume that the ends of $M_i$ and $L_i$ have the same parametrisation, so that $\gamma_T$ is the gluing map on $L_1$ and $L_2$ corresponding to their gluing as submanifolds with gluing parameter $T$ and similarly for the cutoff functions; by reparametrising $M_i$ this may assumed without loss of generality. 
\begin{lem}
\label{gluingconsistency}
Let $M_1$ and $M_2$ be a matching pair of asymptotically cylindrical manifolds, and let $L_1$ and $L_2$ be matching asymptotically cylindrical submanifolds. Let $L^T$ be the glued submanifold of Definition \ref{defin:newsubmanifoldapproxgluing}, and $\alpha_1$ and $\alpha_2$ be a matching pair of asymptotically translation invariant closed forms on $M_1$ and $M_2$. Then as cohomology classes
\begin{equation}
\label{eq:gluingconsistencycohom}
[\gamma_T(\alpha_1|_{L_1}, \alpha_2|_{L_2})] = [\gamma_T(\alpha_1, \alpha_2)]|_{L^T},
\end{equation}
and there exist constants $C_k$ and $\epsilon$ such that for all $k$
\begin{equation}
\label{eq:gluingconsistencynorm}
\|\gamma_T(\alpha_1|_{L_1}, \alpha_2|_{L_2}) - \gamma_T(\alpha_1, \alpha_2)|_{L^T} \|_{C^k} \leq C_ke^{-\epsilon T}.
\end{equation}
\end{lem}
This is easy to prove: indeed, if $L_1$ and $L_2$ are cylindrical and the cutoff functions are chosen appropriately, we get equality, and it is easy to see that changing cutoff functions and passing to the asymptotically cylindrical case introduces exact and exponentially decaying in $T$ terms. 

A similar argument proves
\begin{prop}
\label{metricgluingconsistency}
Let $M^T$ and $L^T$ be as in Definition \ref{defin:newsubmanifoldapproxgluing}. Consider the metric given on $L^T$ by direct gluing of the metrics on $L_1$ and $L_2$ by cutting them off and identifying them over an appropriate region; consider also the metric given on $L^T$ by restricting the metric on $M^T$ defined similarly to $L^T$. The difference of these two metrics satisfies an estimate like \eqref{eq:gluingconsistencynorm}, with respect to either of them. 
\end{prop}
We also note at this point that the second fundamental form of $L^T$ in $M^T$ can be bounded independently of $T$; this follows by looking at the compact parts and the neck separately, and noting that things converge either to the behaviour of $L_i$ in $M_i$ or to the cylindrical behaviour. 
\section{Special Lagrangians}
\label{sec:defsandefs}
In this section, we first define special Lagrangian submanifolds and make some elementary remarks on the structure of the ends of asymptotically cylindrical such submanifolds. Then, in subsection \ref{ssec:speclagexamples}, we show that there are asymptotically cylindrical special Lagrangian submanifolds in some asymptotically cylindrical Calabi--Yau manifolds. Finally, in subsection \ref{ssec:speclagdeformation} we summarise the deformation theory due to McLean \cite{mclean} in the compact case and extended by Salur and Todd \cite{salurtodd} to the asymptotically cylindrical case (where the limit may alter). We give fuller details in the asymptotically cylindrical case, carefully stating the result (Theorem \ref{acyldeformationtheorem}) that the deformations form a manifold with specified dimension, as the argument in \cite{salurtodd} is somewhat unclear and the dimension found in \cite[Theorem 1.2]{salurtodd} is not quite correct. 

\subsection{Definitions}
In this subsection we make the necessary definitions of Calabi--Yau structures, their generalisations $SU(n)$ structures, and define special Lagrangian submanifolds. 
We begin with a Calabi--Yau structure on the ambient manifold. This definition is adapted from Hitchin \cite{hitchinmodulispace}. 
\begin{defin}
\label{defin:sunstructure}
Let $M$ be a $2n$-dimensional manifold. An \emph{Calabi--Yau structure on $M$} is (induced by) a pair of closed forms $(\Omega, \omega)$ where $\Omega$ is a smooth complex $n$-form on $M$ and $\omega$ is a smooth real $2$-form on $M$ such that at every point $p$ of $M$:
\begin{enumerate}[i)]
\item$\Omega_p = \beta_1 \wedge \cdots \wedge\beta_n$ for some $\beta_i \in T^*_p M \otimes \C$,
\item$\Omega_p \wedge \bar \Omega_p \neq 0$,
\item$\Omega_p \wedge\bar\Omega_p  = \frac{(-2)^ni^{n^2}}{n!} \omega_p^n$,
\item$\omega_p \wedge \Omega_p = 0$,
\item$\omega_p(v, J_pv) > 0$ for every $v \in T_pM$ where $J$ is the unique complex structure such that $\Omega$ is a holomorphic volume form. 
\end{enumerate}
We shall call $(\Omega_p, \omega_p)$ satisfying (i)-(v) above an \emph{$SU(n)$ structure on the vector space $T_pM$}. 
\end{defin}
A Calabi--Yau structure on a manifold induces a complex structure and a Ricci-flat K\"ahler metric; as indicated earlier, we shall assume our manifolds with Calabi--Yau structures are connected.

Apart from the condition that the forms are closed, a Calabi--Yau structure is a pair of forms that satisfy certain pointwise conditions. Therefore, it is a section of a subbundle of the bundle of forms. We make
\begin{defin}
\label{defin:cystructurebundle}
Let $\SUn(M)$ be the subset of the bundle $\left({\bigwedge}^{\!n} T^*M \otimes_\R \C\right) \oplus {\bigwedge}^{\!2} T^*M$
\begin{equation}
\bigcup_{p \in M} \{(\Omega_p, \omega_p) \in  \left({\bigwedge}^{\!n} T_p^*M \otimes_\R \C\right) \oplus {\bigwedge}^{\!2} T_p^*M: (\Omega_p, \omega_p) \text{ is an $SU(n)$ structure on $T_pM$}\}.
\end{equation}
A section of $\SUn(M)$ will be called an \emph{$SU(n)$ structure on $M$}. Similarly, if $L$ is a submanifold of $M$, a section of $\SUn(M)|_L$ will be called an \emph{$SU(n)$ structure around $L$}.
\end{defin}
Of course, a Calabi--Yau structure on $M$ is in particular an $SU(n)$ structure on $M$, and any $SU(n)$ structure on $M$  or a tubular neighbourhood of $L$ in $M$ restricts to an $SU(n)$ structure around $L$. 

We now have to combine Definition \ref{defin:sunstructure} with Definition \ref{defin:acylmanifold} to define an asymptotically cylindrical Calabi--Yau structure. In so doing, we make a further restriction. 
\begin{defin}
\label{defin:acylsunstructure}
A Calabi--Yau structure $(\Omega, \omega)$ on a manifold $M$ with an end is said to be \emph{asymptotically cylindrical} if the cross-section $N$ of $M$ is of the form $S^1 \times X$, the induced metric $g$ is asymptotically cylindrical and, with respect to $g$, $\Omega$ and $\omega$ have limits $(dt + i d\theta) \wedge \Omega_\xs$ and $dt \wedge d\theta + \omega_\xs$, where $(\Omega_\xs, \omega_\xs)$ is a Calabi--Yau structure on $X$. 
\end{defin}
\begin{rmk}\label{rmk:restrictedcyreallyis}
In their study of asymptotically cylindrical Ricci-flat K\"ahler manifolds, Haskins--Hein--Nordstr\"om\cite{haskinsheinnordstrom} show that for a somewhat more general definition of asymptotically cylindrical Calabi--Yau manifolds, and provided $n>3$, simply connected irreducible asymptotically cylindrical Calabi-Yau manifolds must have $N = (S^1 \times X)/{\la \Phi \ra}$ for a certain isometry $\Phi$ of finite order (which may be greater than one: see \cite[Example 1.4]{haskinsheinnordstrom}). We restrict to the case where $\Phi$ is simply the identity for simplicity; similar arguments should work in the more general setting. If $n=2$, or for instance there is a torus factor that can be split so that a factor can be taken with $n=2$, then there are further examples of worse limit behaviour; again, we are ignoring these. This restriction follows Salur and Todd \cite{salurtodd}, as we will use the deformation theory of that paper. 
\end{rmk}
Note that it follows from the Cheeger-Gromoll splitting theorem that if $M$ is connected and $N$ (or equivalently $X$) is not, then the manifold reduces to a product cylinder, and so we may freely assume $N$ and $X$ are connected.

We now define special Lagrangian submanifolds and discuss asymptotically cylindrical special Lagrangian submanifolds. 
\begin{defin}
\label{defin:SL}
Let $M$ be a Calabi--Yau $n$-fold and $L$ an $n$-submanifold. $L$ is \emph{special Lagrangian} if and only if $\Im \Omega|_L = 0$ and $\omega|_L = 0$. 
\end{defin}
We now turn to the idea of an asymptotically cylindrical special Lagrangian submanifold of an asymptotically cylindrical Calabi--Yau.  We make the obvious definition that an asymptotically cylindrical special Lagrangian submanifold of an asymptotically cylindrical Calabi--Yau manifold is a special Lagrangian submanifold in the sense of Definition \ref{defin:SL} which is asymptotically cylindrical in the sense of Definition \ref{defin:acylsubmfd}. Note that although the end $N \times (0, \infty) = X \times S^1 \times (0, \infty)$ of the Calabi--Yau manifold $M$ must be connected, asymptotically cylindrical special Lagrangians can have multiple ends contained in this one end.

The following result gives the structure of the cross-section $K$ of $L$. It is straightforward to prove by considering the components of the restrictions of each form.
\begin{prop}
\label{acylsl}
Let $M$ be an asymptotically cylindrical Calabi--Yau manifold with cross-section $N = X \times S^1$ as in Definition \ref{defin:acylsunstructure}, and $L$ be an asymptotically cylindrical special Lagrangian submanifold with cross-section $K$, so that $L = \exp_v(K \times (R, \infty))$ far enough along the end, where $v$ decays exponentially with all derivatives. For each connected component $K'$ of $K$, there is a special Lagrangian $Y$ in $X$, and a point $p \in S^1$, such that $K' = Y \times \{p\} \subset X \times S^1 = N$. 
\end{prop}
\begin{rmk}
Definition \ref{defin:acylsubmfd} and Proposition \ref{acylsl} together form a version of the definition of asymptotically cylindrical special Lagrangian found in \cite[Definition 2.6]{salurtodd}. Our definition, which appears simpler, makes the identifications required, given concretely by Salur and Todd, tacitly and implicitly.
\end{rmk}
\subsection{Examples}
\label{ssec:speclagexamples}
We now give some examples of asymptotically cylindrical special Lagrangian submanifolds in asymptotically cylindrical Calabi--Yau threefolds to which our results will apply. We show (Proposition \ref{everythinginducedisacyl}) that for an asymptotically cylindrical Calabi--Yau threefold obtained as the complement of an anticanonical divisor in a K\"ahler manifold, an antiholomorphic involutive isometry acting on the divisor will yield an antiholomorphic involutive isometry of the asymptotically cylindrical manifold, and hence essentially a special Lagrangian submanifold.  

To explain the construction in detail, we first outline the construction of asymptotically cylindrical Calabi--Yau threefolds. We quote the following from Haskins--Hein--Nordstr\"om \cite{haskinsheinnordstrom}. 

\begin{thm}[\cite{haskinsheinnordstrom}, Theorems D and E, weakened]
\label{tianyautheorem}
Let $\bar M$ be a smooth compact K\"ahler manifold of complex dimension $n \geq 2$, and $D$ be a smooth anticanonical divisor in it which has holomorphically trivial normal bundle. Then for each K\"ahler class $[\omega]$ on $\bar M$, $M = \bar M \setminus D$ admits an asymptotically cylindrical Ricci-flat K\"ahler metric $g_i$, where asymptotically cylindrical means with respect to a diffeomorphism around $D$
\begin{equation}
\label{eq:acylconstructiondiffeo}
U \setminus D \to D \times \{0 < |z| < 1\} \cong D \times S^1 \times (0, \infty),
\end{equation}
with the first map constructed using the trivial normal bundle of $D$ to satisfy \eqref{eq:suitablediffeo} below, and the second given by writing $z = e^{-t+i\theta}$. 

The K\"ahler form is in the cohomology class $[\omega|_M]$, and its limit is $d\theta^2 \otimes g_D$ where $g_D$ is the Calabi--Yau metric on $D$ in the K\"ahler class $[\omega|_D]$. Moreover, the metric is unique subject to the diffeomorphism \eqref{eq:acylconstructiondiffeo} of the tubular neighbourhood to $D \times S^1 \times (0, \infty)$ and these properties. 
\end{thm}
We need such a $\bar M$. There are many possible options; see, for instance, the discussion by Haskins--Hein--Nordstr\"om \cite[top of p.6]{haskinsheinnordstrom}. The simplest example, however, is to take $\bar M$ as complex projective space $\C P^{n+1}$. Anticanonical divisors $D$ are then the zero sets of homogeneous polynomials of degree $n+1$. Provided none of the zeros of such a polynomial are also zeros of its first partial derivatives, the zero set $D$ is smooth by the implicit function theorem. $D$ might nevertheless have nontrivial self-intersection. We thus blow up the self-intersection as in Kovalev \cite{kovalevtwistedconnected}, and as in that paper the resulting submanifold has trivial normal bundle. 

We can improve the uniqueness result of Theorem \ref{tianyautheorem} slightly, and this will be useful for constructing special Lagrangian submanifolds. \eqref{eq:acylconstructiondiffeo} requires a diffeomorphism $\Phi$ between $\Delta \times D$ and an appropriate neighbourhood of $D$ in $\bar M$. There may not be such a diffeomorphism that is biholomorphic, in general. Haskins--Hein--Nordstr\"om prove, however, that provided the diffeomorphism satisfies \cite[Observation A.2]{haskinsheinnordstrom}, Theorem \ref{tianyautheorem} follows. We shall show that the constructed metric is in fact independent of the choice of diffeomorphism here. By the uniqueness part of Theorem \ref{tianyautheorem}, it suffices to show that if $\Phi_1$ and $\Phi_2$ are two such diffeomorphisms, the metrics we construct with each are asymptotically cylindrical with respect to the structure given by the other, since we are by hypothesis using the same K\"ahler class. 

\begin{prop}
\label{newuniqueness}
Let $\Phi_1$ and $\Phi_2$ be diffeomorphisms from $\Delta \times D$ to their images in $\bar M$, satisfying
\begin{equation}
\label{eq:suitablediffeo}
\Phi_i(0, x) = x \text{ for all } x \in D, \Phi_i^* J - J_0 \text{ along $\{0\} \times D$}, \Phi_i^*J - J_0 = 0 \text{ on all of } T\Delta,
\end{equation}
where $J$ and $J_0$ are the complex structures on $\bar M$ and $\Delta \times D$ respectively. (These are the conditions in \cite[Observations A.1 and A.2]{haskinsheinnordstrom}). Given a fixed K\"ahler class on $\bar M$, let $g_1$ and $g_2$ be the Calabi--Yau metrics on $\bar M \setminus D$ constructed by Theorem \ref{tianyautheorem} with these diffeomorphisms so that it is clear that $g_i$ is asymptotically cylindrical with respect to the diffeomorphism $\Psi_i$ given by composing $\Phi_i$ with $z = e^{-t-i\theta}$. Then $g_1$ is also asymptotically cylindrical with respect to $\Psi_2$, and vice versa; the limits are the same. In particular, $g_1 = g_2$. 
\end{prop}
\begin{proof}
The construction in \cite[subsection 4.2]{haskinsheinnordstrom} is that if $\Phi_i$ is such a diffeomorphism then an asymptotically cylindrical metric near $D$ can be chosen by taking near $D$
\begin{equation}
\label{eq:cylkahlermetric}
\frac{i}{2}\frac{dz_i \wedge d\bar z_i}{|z_i|^2} + \omega,
\end{equation}
where $\omega$ is the K\"ahler form on $D$ corresponding to the Calabi-Yau metric in the restriction of the K\"ahler class and $z_i$ is the function $\bar M \to \Delta$ given by composing $\Phi_i^{-1}$ with the projection. Because we are using the same K\"ahler class in both cases, $\omega$ does not depend on $i$. 

Much of the work in \cite[subsection 4.2]{haskinsheinnordstrom} is simply in showing that this can be cut off under the differentials without affecting the asymptotics. For instance, on $\Delta \times D$ \eqref{eq:cylkahlermetric} can be expressed as $\frac{i}{2}\partial \bar \partial (\log |z|)^2 + \omega$; the conditions in and arising from \eqref{eq:suitablediffeo} are primarily to ensure that this still has the same asymptotic when $\partial$ and $\bar \partial$ are given by the complex structure on $\bar M$. 

Using the same ideas, we shall show that 
\begin{equation}
\label{eq:targetcylsarecyls}
\frac{dz_1  \wedge d\bar z_1}{|z_1|^2} = \frac{dz_2  \wedge d\bar z_2}{|z_2|^2} + \text{decaying terms},
\end{equation}
where the decaying terms become $O(e^{-t})$ under the identification. This implies $g_1$ has the same limit as $g_2$. 

To do this, we shall primarily work locally in $D$. We know by \cite[proof of Observation A.3]{haskinsheinnordstrom} that for any local holomorphic defining function $w$ for $D$ in $\bar M$, we may write locally
\begin{equation}
 f_1 z_1 + z_1^2 h_1 = w = f_2 z_2 + z_2^2 h_2,
\end{equation}
where $f_1$ and $f_2$ are nonzero holomorphic functions on an open subset of $D$, and $h_1$ and $h_2$ are smooth functions on a neighbourhood of this open subset in $\bar M$. By rearranging, and expanding using the Leibniz rule, it is easy to see that
\begin{equation}
\frac{dz_1}{z_1} = \frac{dz_2}{z_2} + \frac{f_1}{f_2} d\left(\frac{f_2}{f_1}\right) + \text{ decaying terms}.
\end{equation}
It remains to show that $\frac{f_1}{f_2}$ is a constant; then the middle term vanishes and \eqref{eq:targetcylsarecyls} holds. 

Since $f_1$ and $f_2$ are nonzero holomorphic functions, $\frac{f_1}{f_2}$ is holomorphic. We claim that this holomorphic function is independent of the local defining function $w$. Then, by taking a cover of $D$, $\frac{f_1}{f_2}$ extends to a holomorphic function on the compact complex manifold $D$ and so is a constant.

To do this, we note that $f_i = dw(\frac{\partial}{\partial z_i})$. Let $w'$ be another local holomorphic defining function, and let $p \in D$ be in the relevant open subset. $(dw)_p$ and $(dw')_p$ both lie in $(T^{1, 0})^*_p \bar M$, and moreover as $w$ and $w'$ are defining functions they both vanish on $T^{1, 0}_pD$. It follows that $(dw)_p = a(dw')_p$ for some $a$, and $a$ cancels in the quotient $\frac{f_1}{f_2}$. 
\end{proof}

In particular, we may obtain the following
\begin{cor}
\label{finalcycylinderuniqueness}
Suppose that $\bar M_1$ and $\bar M_2$ are smooth compact K\"ahler manifolds of complex dimension $n \geq 2$. Suppose $D_i$ is a smooth anticanonical divisor in $\bar M_i$, for each $i$, with holomorphically trivial normal bundle. Suppose that $f: \bar M_1 \to \bar M_2$ is a biholomorphism with $f(D_1) = D_2$. Suppose $[\bar \omega_2]$ is a K\"ahler class on $\bar M_2$ so that $[\bar \omega_1] = f^*[\bar \omega_2]$ is a K\"ahler class on $\bar M_1$. Theorem \ref{tianyautheorem} constructs asymptotically cylindrical metrics on $\bar M_1 \setminus D_1$ and $\bar M_2 \setminus D_2$. The map $f$ induces an isometry between these. 
\end{cor}
This follows from Proposition \ref{newuniqueness} by composing $f$ with the required diffeomorphism and noting that it still satisfies \eqref{eq:suitablediffeo}. 

Using Corollary \ref{finalcycylinderuniqueness}, we will now explain a construction of asymptotically cylindrical special Lagrangian submanifolds, that is fairly general in a subset of these asymptotically cylindrical Calabi--Yau manifolds in Proposition \ref{everythinginducedisacyl}; we will then specialise and give a concrete example. 

To do this, we will use antiholomorphic involutive isometries, as mentioned, for instance, by Joyce--Salur \cite[top of p.1118]{joycesalur}. We make two similar definitions, one of which applies for a general complex manifold such as $\bar M$, and one of which is specialised for the Calabi--Yau case. 
\begin{defin}
Let $M$ be a K\"ahler manifold with complex structure $J$ and K\"ahler form $\omega$. A diffeomorphism $\sigma$ of $M$ is called an \emph{antiholomorphic involutive isometry} if 
\begin{equation}
\sigma^2 = \id, \qquad \sigma^* J = -J, \qquad \sigma^* \omega = - \omega,
\end{equation}
If $M$ is in fact a Calabi--Yau manifold with Calabi--Yau structure $(\Omega, \omega)$, the antiholomorphic involutive isometry $\sigma$ has \emph{fixed phase} if also
\begin{equation}
\sigma^* \Omega = \bar \Omega
\end{equation}
\end{defin}
Note that any antiholomorphic involutive isometry $\sigma$ of a Calabi--Yau manifold must satisfy $\sigma^* \Omega = e^{i\alpha} \bar\Omega$ for some $\alpha$. It follows that $\sigma$ is an antiholomorphic involutive isometry with fixed phase for $(e^{-\frac{i\alpha}2} \Omega, \omega)$. Thus, for our purposes it is enough to find antiholomorphic involutive isometries and then adjust the Calabi--Yau structure. 

It is easy to see
\begin{lem}
\label{fixedpointsaresubmfd}
The fixed point set of an antiholomorphic involutive isometry with fixed phase is a special Lagrangian submanifold. 
\end{lem}
Now for $M$ an asymptotically cylindrical Calabi--Yau, we want to find an antiholomorphic involutive isometry $\sigma$ (with fixed phase) such that this special Lagrangian is asymptotically cylindrical. To do this, we induce $\sigma$ from our construction of $M$. Specifically, we have
\begin{prop}
\label{everythinginducedisacyl}
As in Theorem \ref{tianyautheorem}, let $\bar M$ be a smooth compact K\"ahler manifold of complex dimension $n\geq2$ and $D$ a smooth anticanonical divisor with holomorphically trivial normal bundle. Let $\sigma$ be an antiholomorphic involutive isometry of $\bar M$. Suppose moreover that $\sigma$ acts on $D$, and $D$ contains some fixed points of $\sigma$. 

By Theorem \ref{tianyautheorem}, $M = \bar M \setminus D$ admits an asymptotically cylindrical Ricci-flat K\"ahler metric, in the restriction of this K\"ahler class. Moreover, the restriction of $\sigma$ to $M$ defines an antiholomorphic involutive isometry of $M$ which is asymptotically cylindrical. That is, there exist a diffeomorphism $\tilde \sigma$ of $N$ and a parameter $l \in \R$ such that
\begin{equation}
\sigma(n, t) \to (\tilde \sigma(n), t+l)
\end{equation}
exponentially, meaning that on restriction to $N \times (T, \infty)$ for some large $T$, we have $\sigma = \exp V \circ (\tilde \sigma(n), t+l)$ for some vector field $V$ on $M$ decaying exponentially with all derivatives. The fixed point set of $\sigma$  is an asymptotically cylindrical submanifold in the sense of Definition \ref{defin:acylsubmfd}. 
\end{prop}
\begin{proof}
Since $\sigma$ acts on $D$, it also acts on $M$. Hence the restriction of $\sigma$ to $M$ is well-defined and squares to the identity. 

Now the complex structure on $M$ is just the restriction of the complex structure on $\bar M$. Hence, we immediately have $\sigma^* J = -J$. It remains to show that $\sigma$ induces an isometry. To do this, we apply Corollary \ref{finalcycylinderuniqueness}. $\sigma$ is precisely a biholomorphism between $\bar M$ and its complex conjugate, preserving the anticanonical divisor. Hence, Corollary \ref{finalcycylinderuniqueness} gives that it induces an isometry. 

Now \cite[Proposition 6.22]{nordstromacyldeformations} implies that $\sigma$ is an asymptotically cylindrical diffeomorphism, as an isometry of asymptotically cylindrical metrics. It only remains to show that the fixed point set, which we know by Lemma \ref{fixedpointsaresubmfd} is a submanifold, is in fact asymptotically cylindrical. To do this, we use that $\sigma$ is an involution again. 

We start by noting that since $\sigma$ is an involution, we must have $l=0$. By abuse of notation, we shall also write $\tilde \sigma$ for the induced diffeomorphism on $(T, \infty) \times N$. 

We now show that the fixed points of $\sigma$, far enough along the end, are the image of the fixed points of $\tilde \sigma$ under $\exp_{\frac v 2}$. The fixed points of $\tilde \sigma$ are certainly a cylindrical submanifold. Since $v$ is exponentially decaying, and the map from a vector field along a submanifold to the normal vector field with the same image preserves exponential decay (this follows by similar arguments to Proposition \ref{revisedtransferofatis} below), this shows that the fixed points of $\sigma$ are an asymptotically cylindrical submanifold, as required. Given $p$ with $\tilde \sigma(p) = p$, we must have $\sigma(p) = \exp_v(p)$. For $t$ large enough, uniformly in the cross-section, there is a unique minimising geodesic from $p$ to $\exp_v(p)$. Since $\sigma$ is an involution, $\sigma(\exp_v(p)) = p$, and so this geodesic is reversed by the isometry $\sigma$. Hence its midpoint $\exp_{\frac v 2}(p)$ is fixed by $\sigma$. This shows that the fixed points of $\sigma$ contain this image (at least far enough along the end); the reverse argument is only slightly more complicated. 

This shows that the fixed points of $\sigma$ are an asymptotically cylindrical submanifold with the same limit as the fixed points of $\tilde \sigma$ (in other words, the limit is the fixed points of $\tilde \sigma$ on $N$). It remains to prove that $\tilde \sigma$ fixes some points of $N = D \times S^1$, otherwise all we have done is exhibited a closed special Lagrangian submanifold. We shall consider $\tilde \sigma(x, \theta)$. By \cite[Proposition 6.22]{nordstromacyldeformations}, this can be given by the limit approached by the image under $\sigma$ of the unique geodesic half-line approaching $(x, \theta)$. Any curve approaching $(x, \theta)$ must, passing to a curve in $\bar M$, approach $(x, 0)$. Hence the image approaches $(\sigma|_D(x), 0)$. This shows that $\tilde \sigma(x, \theta) = (\sigma|_D(x), \theta')$ for some $\theta'$. Now also we know by continuity that $\tilde \sigma$ is antiholomorphic on the Calabi--Yau cylinder $\R \times S^1 \times D$. Since we clearly have $\tilde \sigma_* \ddt = \ddt$, we must have $\tilde \sigma_* \ddth = -\ddth$. Hence also the isometry $\tilde \sigma_*$ preserves the orthogonal complement $TD$ of $\spn \{\ddth\}$. It follows that $\theta'$ depends only on $\theta$. Moreover, the isometry thus given on $S^1$ is a reflection. Hence it has two fixed points. Thus for each fixed point of $\sigma|_D$ we have two fixed points of $\tilde \sigma$ on $N$. Since by hypothesis $\sigma|_D$ has fixed points, $\tilde \sigma$ has fixed points, and the resulting submanifold is indeed asymptotically cylindrical. 
\end{proof}
Hence, if we choose the appropriate phase for the holomorphic volume form, we get an asymptotically cylindrical special Lagrangian submanifold. 

In particular, we shall consider the simplest examples of such $\bar M$, that is the blowups of Fano manifolds in the self-intersection of their canonical divisor, as described below Theorem \ref{tianyautheorem}. There is no need to worry about metrics here, and so we restrict to \emph{antiholomorphic involutions}, that is involutions $\sigma$ with $\sigma^* J = -J$. We state the following
\begin{prop}[cf. Kovalev {\cite[top of p.19]{kovalevspin7}}]
\label{speclagblowup}
Suppose that $\bar M$ is a Fano manifold with $D$ a smooth anticanonical divisor and $\sigma$ is an antiholomorphic involution which acts on $D$. Suppose further that we may find a submanifold representing the self-intersection of $D$ on which $\sigma$ acts. Then if we blow up this self-intersection, the resulting manifold admits an antiholomorphic involution acting on the anticanonical divisor given by the proper transform of $D$. 
\end{prop}
Essentially this follows because the blowup is unique and determined by the existence of a suitable projection. We can then choose any K\"ahler class $[\omega]$ on the blowup and consider $[\omega] - \sigma^*[\omega]$ to get an antiholomorphic involutive isometry. 

The only explicit example we discussed above was to take 
\begin{equation}
\bar M = \C P^n, D = \{p(x_1, x_2, x_3, x_4, \ldots x_{n+1}) = 0\},
\end{equation}
for a homogeneous polynomial $p$ of degree $n+1$ that is always a submersion at its zeros. If $p$ has real coefficients, $D$ is preserved by the involution of $\C P^n$ given by complex conjugation of $\C^{n+1}$. Since we may perturb $p$ to get a transverse submanifold without losing the real coefficients, we may suppose $\sigma$ acts on the self-intersection, and so apply Proposition \ref{speclagblowup} to find an antiholomorphic involution of the blowup. Choosing an appropriate K\"ahler class, applying Proposition \ref{everythinginducedisacyl}, and then choosing the appropriate phase for the holomorphic volume form, we find an  asymptotically cylindrical Calabi--Yau manifold admitting a special Lagrangian submanifold. 

We know that each end of this special Lagrangian has limit components $\{p\} \times Y$, for some special Lagrangian $Y$ in the proper transform of $D$. It is clear from the construction that $Y$ is a component of the fixed points of the induced map on the blowup restricted to this proper transform. This set of fixed points is identified with the fixed points of $\sigma|_D$, which are given by the product of the solutions to $\{p(x_1, x_2, x_3, x_4, \ldots x_{n+1}) = 0\}$ in the real projective space $\R P^n$ with $\R P^1$ (as we may change the argument). As a basic example, therefore, taking $n=3$ and the divisor $x_1^4 + x_2^4 + x_3^4 + x_4^4 = 0$, this method does not yield an asymptotically cylindrical special Lagrangian. However, $x_1^4 + x_2^4 + x_3^4 - x_4^4$ will; so will other similar examples. Moreover, each component of this set corresponds to two ends of the special Lagrangian, again by the construction. Thus a special Lagrangian constructed by this method always has an even number of ends; since the fixed points of $\sigma|_D$ may not be connected, it follows that there are examples with more than two. Of course, the special Lagrangian constructed in this way may itself be disconnected, so by taking connected components we may be able to get examples with odd numbers of ends. Note also that the two ends are given by antipodal points of $S^1$; we shall show in the next subsection that this condition need not be preserved by deformations, so that there are further examples not constructed in this method. 

Similar arguments will apply more generally to Fano manifolds constructed as complete intersections; we should be able to choose the required polynomials to be real inductively, including the anticanonical divisor and its perturbation. 
\subsection{Deformation theory of special Lagrangians}
\label{ssec:speclagdeformation}
In this subsection, we give a summarised version of McLean's deformation result \cite{mclean} and the extension to asymptotically cylindrical special Lagrangians given by Salur and Todd \cite{salurtodd}. We shall outline the proof in the compact case as our gluing result will rest on applying the same idea to perturb a nearly special Lagrangian submanifold to a special Lagrangian submanifold; in the asymptotically cylindrical case, we shall give some details to improve the results and provide clearer proofs. In particular, note that Theorem \ref{acyldeformationtheorem} is not quite in agreement with \cite[Theorem 1.2]{salurtodd}. 

We begin with the compact case. Let $M$ be a Riemannian manifold. If $L$ and $L'$ are closed submanifolds of $M$ with $L$ and $L'$ close in a $C^1$ sense, then there is a normal vector field $v$ to $L$ such that $L'$ is the image of $v$ under the Riemannian exponential map; in this case, we write $L' = \exp_v(L)$. Conversely, of course, a $C^1$-small normal vector field $v$ defines a $C^1$-close submanifold $L' = \exp_v(L)$. Moreover, for smooth $L$, the regularities of $L'$ and $v$ are the same. 

Thus, to work with submanifolds it suffices to work with normal vector fields. Hence, to understand which submanifolds $L'$ close to $L$ are special Lagrangian, we want to understand the zero set of the nonlinear map 
\begin{equation}
\label{eq:mcleandeformationmap}
F: v \mapsto (\exp_v^* \Im \Omega|_L, \exp_v^* \omega|_L).
\end{equation}
To do this, we restrict to appropriate Banach spaces and apply the implicit function theorem to $F$ to obtain families of special Lagrangian deformations; we then show that the family obtained is independent of the choice of Banach space. 

The derivative $D_0F$ and its surjectivity between appropriate spaces are covered in detail by McLean \cite[Theorem 3-6]{mclean}. A calculation using Cartan's magic formula shows that the linearisation $D_0F$ is 
\begin{equation}
\label{eq:mcleanlinearisation}
D_0F(u) = (d(\iota_u \omega|_L), d* (\iota_u \omega|_L)),
\end{equation}
recalling that since $L$ is Lagrangian, $u \mapsto \iota_u \omega|_L$ is an isomorphism between the normal and cotangent bundles of $L$. We will generalise this isomorphism for our purposes in Lemma \ref{Jbehaves} below. 

Surjectivity is slightly more technical. We need to use some Banach space of forms, say $C^{1, \mu}$. It is easy to see that $F$ maps only to exact forms, and so we consider $F$ as a map from $C^{1, \mu}$ normal vector fields to $C^{0, \mu}$ forms. 

By Baier \cite[Theorem 2.2.15]{baier}, we know that $F$ is then a smooth map of Banach spaces. $D_0F$ then becomes \eqref{eq:mcleanlinearisation} between $C^{1, \mu}$ normal vector fields and exact $C^{0, \mu}$ forms; since $u \mapsto \iota_u \omega|_L$ is an isomorphism between $C^{1, \mu}$ normal vector fields and $C^{1, \mu}$ forms, and $d+d*$ is surjective, it follows that \eqref{eq:mcleanlinearisation} is surjective. 

Finally, to show the result is independent of the Banach space used, the choice of Banach space is a choice of regularity for our solutions of $F$. But these vector fields have the same regularity as the submanifolds concerned. Thus this reduces to smoothness of calibrated (and hence minimal) submanifolds.

We also review an asymptotically cylindrical extension of this theory. This is due to Salur and Todd \cite{salurtodd}. We begin by discussing normal vector fields corresponding to asymptotically cylindrical deformations of an asymptotically cylindrical manifolds. We would like to say that these are asymptotically translation invariant in some sense. We will give a definition of these much later in Definition \ref{defin:newatinvf}.

In particular, if the asymptotically cylindrical deformations have the same limit, we expect to get exponentially decaying normal vector fields. It is clear that by adding an exponentially decaying normal vector field on $K \times (R, \infty)$ to the vector field $v$ of Definition \ref{defin:acylsubmfd} also gives an asymptotically cylindrical submanifold. The setup of the tubular neighbourhood theorem used by Salur and Todd \cite[p. 110-1]{salurtodd} precisely uses this: they define the map $\Xi$ to be applying the exponential map to this sum, and then choose an isomorphism $\zeta$ between the normal bundles of (the end of) $L$ and $K \times (R, \infty)$; this isomorphism is just a pushforward and so it is reasonably obvious that it preserves exponentially decaying normal vector fields. The converse result follows in the same way. 

For the asymptotically translation invariant case Salur and Todd use the rigidity of cylindrical special Lagrangians to reduce the problem (and hence to avoid needing a general definition of asymptotically translation invariant normal vector fields). They say that a cylindrical special Lagrangian deformation of the limit corresponds to a translation-invariant one-form using the isomorphism between one-forms and normal vector fields on a special Lagrangian on $K \times (R, \infty)$, and hence gives a well-defined one-form on the end of $L$, and then they combine these one-forms with exponentially decaying one-forms.

We are proceeding more directly: we have $L$ a well-defined submanifold of $M$, and consider the exponential map. It is simplest to suppose that we take the exponential map corresponding to a cylindrical metric on $M$; our choice of exponential map does not affect the final result. We have to show that there is a uniform lower bound on the injectivity radius of $M$ with the cylindrical metric to ensure our tubular neighbourhood is going to contain all reasonably nearby submanifolds. But this is clear, as with the cylindrical metric the injectivity radius only depends on the geometry of the cylinder and of the compact complement $M^\cpt$. As before, $C^1$-uniformly close $C^{k, \mu}$ submanifolds correspond to $C^1$-uniformly-small $C^{k, \mu}$ normal vector fields. We now have to pass to the subset corresponding to asymptotically cylindrical submanifolds. 

Again using the isomorphism between one-forms and normal vectors, we can assume that the Riemannian exponential map is defined on $T^*L$. If a one-form on an asymptotically cylindrical submanifold is asymptotically translation invariant, then its image under the Riemannian exponential map is also an asymptotically cylindrical submanifold. This result is still needed by Salur and Todd, in the specific cylindrical case. This is much less obvious a priori than the exponentially decaying case: whilst the image certainly decays in a $C^0$ sense to the corresponding cylindrical manifold, it is not clear why the associated normal vector field must decay with all its derivatives. 

If we suppose that the Riemannian metric we use on $M$ is actually cylindrical, this follows because combining one-forms is a smooth operation. Specifically, by the definition of ``asymptotically cylindrical submanifold", we have to show that $\exp_\beta \exp_\alpha (K \times (R, \infty)) = \exp_\gamma(K' \times (R, \infty))$ is asymptotically cylindrical, that is that we can find $K'$ and $\gamma$ so that $\gamma$ is exponentially decaying. $K'$ can be obtained from the limit, and then pointwise, $\gamma$ can be chosen to depend smoothly on $\alpha$ and $\beta$, with a well-defined limit in this map as $t \to \infty$. Indeed, since the metric is cylindrical, we do not need to worry about it in this limit: for everything else, we have smooth dependence and so exponential convergence to this limit. Conversely, given two asymptotically cylindrical submanifolds with limits $\exp_\alpha(K_1 \times (R, \infty))$ and $\exp_\beta(K_2 \times (R, \infty))$ close enough that one is the image under the Riemannian exponential map of a one-form $\gamma$, $\gamma$ must have a well-defined limit (corresponding to the normal vector field between $K_1$ and $K_2$) and again $\gamma$ at each point depends smoothly on $\alpha$ and $\beta$ and continuously on $t$ with a well-defined limit, so decays exponentially to the limit $\tilde \gamma$. For further details of arguments like this, see section \ref{sec:dpatching} below, particularly Proposition \ref{revisedtransferofatis}. 

That is, a sensible definition of asymptotically translation invariant normal vector field is just ``a normal vector field whose corresponding one-form is asymptotically translation invariant". We shall show this is equivalent to the general definition we shall give in Definition \ref{defin:newatinvf} in Proposition \ref{aticonsistency} later. 

In particular, an exponentially decaying normal vector field gives an asymptotically cylindrical submanifold with the same limit. Conversely, given a close asymptotically cylindrical submanifold with the same limit, we can find such an asymptotically translation invariant normal vector field, and this normal vector field must decay exponentially. 

We will now sketch how to prove the following variant of {\cite[Theorem 1.2]{salurtodd}}. 
\begin{thm}
\label{acyldeformationtheorem}
Let $M$ be an asymptotically cylindrical Calabi--Yau manifold, with cross-section $X \times S^1$, where $X$ is a compact connected Calabi--Yau manifold, and let $L$ be an asymptotically cylindrical special Lagrangian in $M$ with cross-section asymptotic to $\bigcup Y_i \times \{p_i\}$ where, for each end $i$, $Y$ is a special Lagrangian submanifold of $X$. 

The space of special Lagrangian deformations of $L$ that are also asymptotically cylindrical with sufficiently small decay rate is a manifold with tangent space at $L$ given by the normal vector fields on $L$ such that $\iota_v \omega|_L$ is a bounded harmonic one-form. 
\end{thm}
The most important difference between this and \cite[Theorem 1.2]{salurtodd} is that there it is claimed that the tangent space has slightly greater dimension. Salur and Todd claim that we can find a deformation curve with limit corresponding to every normal vector field limits $\sum c_i \ddth$, and this is not the case in general, as we shall see in Proposition \ref{speclaglimitdeformationsubmanifold} below. We also give a slightly clearer explanation of which deformations of the $Y_i$ correspond to deformation curves in Proposition \ref{speclaglimitdeformationsubmanifold}, and we will need this argument later in Proposition \ref{manifoldofgluablesl} where we will apply it to show that the space of matching special Lagrangian deformations is a submanifold. Finally, we drop the connectedness assumption from Salur--Todd, though this has no practical effect on the proof. Hence we give some details. 

By considering each component separately, we may suppose $L$ is connected throughout the proof. Suppose that $L'$ is such a deformation and the corresponding asymptotically translation invariant normal vector field $v$ has limit $\tilde v$. Then in the cylindrical Calabi--Yau manifold $\R \times S^1 \times X$, if we extend $\tilde v$ to be translation invariant then $\exp_{\tilde v}(\R \times Y \times \{p\})$ must also be a translation invariant special Lagrangian.  Our first problem is thus to identify the possible limits $\tilde v$. This can be done for each end separately. 

Again, we apply \eqref{eq:mcleandeformationmap} and \eqref{eq:mcleanlinearisation}; since the only noncompactness is $\R$, and we assume translation-invariance, McLean's result applies and locally the $\tilde v$ such that $\exp_{\tilde v} (\R \times Y \times \{p\})$ is special Lagrangian are diffeomorphic to translation-invariant harmonic normal vector fields, meaning the normal vector fields whose corresponding one-form is harmonic. It is easy to see that these are precisely of the form $u + C\ddth$, where $u$ is a harmonic normal vector field on $Y$, and $C$ is constant on $Y$. 

We further have to consider which of these deformations arises as the limit of an asymptotically cylindrical deformation. 
\begin{prop}
\label{speclaglimitdeformationsubmanifold}
Let $L_0$ be an asymptotically cylindrical special Lagrangian submanifold of the asymptotically cylindrical Calabi--Yau manifold $M$, with cross-sections $K_0 = \{p\} \times Y$ and $N = S^1 \times X$ respectively. We know that the set of cylindrical deformations of the special Lagrangian $\R \times \{p\} \times Y$ in $\R \times S^1 \times X$ is a manifold. 

The set of limits of asymptotically cylindrical deformations of $L$ is a submanifold $\mathcal{K}$ of this submanifold of cylindrical deformations. Its tangent space is precisely those harmonic  normal vector fields on $K$ that arise as limits of bounded harmonic normal vector fields on $L$. 
\end{prop}
\begin{proof}
Let $K'$ be a deformation of the limit $K$, and extend it to $L'$. We have to show that we can choose $L'$ to be special Lagrangian. This can be done if and only if the exact and exponentially decaying forms $\Im \Omega|_{L'}$ and $\omega|_{L'}$ are the differentials of decaying forms for any such extension: see \cite[p.121]{salurtodd}. 

We know that $\Im \Omega|_{L'}$ and $\omega|_{L'}$ are decaying closed forms. The decaying forms form a complex under exterior differentiation. We call the cohomology of this complex $H^*_\rel(L)$; the argument of Melrose \cite[Proposition 6.13]{melrose} says that this the same as the compactly supported cohomology. To ask for $\Im \Omega|_{L'}$ and $\omega|_{L'}$ to be the differential of decaying forms is to ask for their classes in these cohomology groups to be the trivial classes. 

As in \cite{salurtodd}, we know that $\omega$ and $\Im \Omega$ restrict to exact forms on a tubular neighbourhood of $L_0$ (since they restrict to zero on $L_0$). Consequently we may find $\tau_1$ and $\tau_2$ asymptotically translation invariant with limits $\tilde \tau_1$ and $\tilde \tau_2$ such that $\tau_i|_{L_0} = 0$ and $d\tau_1 = \omega$, $d\tau_2 = \Im \Omega$ on a tubular neighbourhood of $L_0$. 

We have a coboundary map $\partial: H^p(K) \to H^{p+1}_\rel(L)$ given by the long exact sequence of relative cohomology. Examining the definition of $\partial$, we see that $[\omega|_{L'}] \in H^{2}_\rel(L)$ is given by $\partial([\tilde \tau_1|_{K'}])$ and $[\Im \Omega|_{L'}] \in H^{n}_\rel(L)$ is $\partial([\tilde \tau_2|_{K'}])$. 

That is, we need to show that the space of deformations $K'$ such that $([\tilde\tau_1|_{K'}], [\tilde\tau_2|_{K'}])$ are both in $\ker \partial$ is a submanifold with the desired tangent space. 

By Proposition \ref{acylsl}, we know each component of the cross-section $K$ is of the form $\{p\} \times Y$ for some point $p \in S^1$ and some special Lagrangian $Y$ in $X$. Thus we may write
\begin{equation}
K = \bigcup \{p^i\} \times Y^i 
\end{equation}
It similarly follows that any special Lagrangian deformation is 
\begin{equation}
K_s = \bigcup \{(p^i)'\} \times (Y^i)' 
\end{equation}
We now note that since $\Im \tilde \Omega = d\theta \wedge \Re \Omega_\xs + dt \wedge \Im \Omega_\xs$ restricts to zero on each space $\{(p^i)'\} \times X$, $\tilde \tau_2$ is closed on this submanifold. Consequently, 
\begin{equation}
[\tilde \tau_2|_{K_s}] = [\tilde \tau_2|_{\bigcup \{(p^i)'\} \times (Y^i)'} 
\end{equation}
$Y^i$ is special Lagrangian so that $\Re \Omega_\xs|_{Y^i} = \vol_{Y^i}$, for a suitable choice of orientation. Calculation shows that this choice can be made so that $dr \wedge \vol_{Y_i}$ is the limit of $\vol_L$ on this end. Applying Stokes' theorem, we find
\begin{equation}
[\tilde \tau_2|_{K_s}] = \sum ((p^i)' - p^i) [\vol_{Y_i}]
\end{equation}
Now the kernel of $\partial$ on $H^{n-1}(K)$ is the image of the restriction map $H^{n-1}(L) \to H^{n-1}(K)$. It is known (\eg Nordstr\"om \cite[Proposition 5.12]{nordstromacyldeformations}) that for any given metric this is equivalent to being $L^2$-orthogonal to those $\beta$ so that there exists a harmonic form on $L$ with limit $dr \wedge \beta$. As $L$ is connected by assumption, there is also one such harmonic form, $\vol_L$, with limit $dr \wedge (\sum \vol_{Y_i})$. It immediately follows that $K_s$ satisfies $\partial ([\tilde \tau_2|_{K_s}]) = 0$ if and only if 
\begin{equation}
\label{eq:relboundarycondition}
\sum ((p^i)' - p^i) \Vol(Y_i) = 0
\end{equation}
As the end is cylindrical, this defines a linear subspace of the potential deformations $K_s$ (the $p^i$ and $Y^i$ components are essentially independent of each other, so this is still a manifold). We now restrict to this subspace, and show that within it the deformations satisfying $\partial([\tilde \tau_1|_{K'}]) = 0$ form a manifold with the desired properties. We consequently consider the nonlinear map 
\begin{equation}
\label{eq:ethmap}
K' \mapsto \partial[\tilde \tau_1|_{K'}].
\end{equation}
This linearises using Cartan's magic formula and the assumption $\tilde \tau_1|_{K} = 0$ to give
\begin{equation}
\label{eq:ethlinearisation}
v \mapsto \partial [\iota_v \tilde \omega|_{K}].
\end{equation}
To prove that the kernel of \eqref{eq:ethmap} is a submanifold, we only have to prove that its linearisation \eqref{eq:ethlinearisation} is surjective to the image of $\partial$ in $H^2_\rel(L)$. We begin by noting that if we allowed all normal vector fields $v$ with $\iota_v \omega|_K$ harmonic, $\iota_v \omega$ would certainly include all of $H^1(K)$, and hence \eqref{eq:ethlinearisation} would be surjective. Thus it suffices to show that the normal vector field we removed by the condition \eqref{eq:relboundarycondition} maps to zero. But this is $\sum \ddth$, and $dr|_{K} = 0$. 

Thus we have shown that the deformations of $K$ that extend to deformations of $L$ form a submanifold $\K$ of all special Lagrangian deformations of $K$. Its tangent space is normal vector fields $\sum c_i \ddth + v_i$ with $v_i$ normal to $Y_i$ such that $\iota_{\sum c_i \ddth + v_i} \omega|_K$ is harmonic, 
\begin{equation}
\sum c_i Vol(Y_i) = 0, \sum \partial (\iota_{v_i} \tilde \omega|_{Y_i}) = 0 
\end{equation}
It remains to check that for every such normal vector field, there is a harmonic normal vector field $v$ on $L$ with these limits, and conversely. In terms of one-forms, the limit becomes
\begin{equation}
\sum c_i dr + \iota_{v_i} \tilde \omega|_{Y^i} 
\end{equation}
We know (compare \eg Nordstr\"om \cite[Definition 5.8]{nordstromacyldeformations}) that we can treat the two parts of this separately by considering exact and coexact harmonic forms. Thus we have to show a harmonic one-form with limit either of the form $\sum c_i dr$ or a one-form on the $Y^i$ exists if and only if it satisfies the relevant condition. The second is obvious from exactness; the first can be dealt with again using $L^2$-orthogonality of the two kinds of limit. $\sum c_i$ is a harmonic $0$-form; there is a harmonic form with limit $\sum c_i dr$ if and only if $\sum {c_i}$ is $L^2$-orthogonal to the limits of harmonic $0$-forms; there is only one of these, namely $1$, so this is 
\begin{equation}
\sum c_i \Vol(Y_i) = 0
\end{equation}
exactly as required. Hence $\K$ has tangent space vector fields corresponding to limits of bounded harmonic one-forms. 
\end{proof}
\begin{rmk}
Proposition \ref{speclaglimitdeformationsubmanifold} is essentially claimed in \cite[Proposition 6.4]{salurtodd}, and we have followed the ideas of their proof. Their proof itself is somewhat unclear and possibly erroneous: we know that the $\tau_i$ decay on $L$, but Salur--Todd seem in that proposition to claim they have zero limit on $L'$ too, so we have expanded on it. On \cite[p. 123]{salurtodd}, Salur and Todd further claim that the normal vector fields $C\ddth$ give further perturbations $L'$ of $L$, which is unfortunately untrue also. 
\end{rmk}
Note that in the examples from subsection \ref{ssec:speclagexamples}, assuming we only have two ends, these ends are $\{p\} \times Y$ and $\{p+\pi\} \times Y$. Thus the condition on the $c_i$ becomes $c_1 + c_2 = 0$, and so this structure is not preserved in general. 

To obtain our final result, we take a space of vector fields $v$ defined as ``decaying vector fields" $\oplus$ ``acceptable limits". In order to do this, we have to take some choice of vector field for each acceptable limit $\tilde v$. Whilst there is no canonical way to do this, if we assume that they are eventually constant any differences can be absorbed into the decaying vector fields. The image of such a vector field under the map $v \mapsto (\exp_v^* \Im \Omega, \exp_v^* \omega)$ is obviously exponentially decaying, and is still exact by homotopy, and so we may still consider a right hand side of decaying exact forms; furthermore, by the above, the right hand side can be considered as differentials of decaying forms. We thus work on the (nonlinear) subspace of all normal vector fields corresponding to asymptotically cylindrical submanifolds that decay to a limit in $\mathcal{K}$; the tangent space to this is normal vector fields decaying to harmonic normal vector fields that are limits of harmonic normal vector fields on $L$, and so the linearisation  comes down to the effect of $(d, d*)$ from forms that decay to harmonic limits that are the limits of harmonic forms to the differentials of decaying forms. This is surjective exactly as in the decaying case, and evidently its kernel is precisely bounded harmonic forms. We get dimension of the moduli space the dimension of the bounded harmonic $1$-forms, which is $b^1(L)$. 

Finally, we have to show that the resulting moduli space is independent of the regularities used. By standard minimal surface regularity as in the compact case, we certainly know that everything involved is smooth. We thus only need to show that the moduli space is independent of the decay rate implicitly chosen in the previous paragraph. It is clear that if we consider things decaying slower, we get the same solutions and maybe some more: we have to show that as we consider decay rate getting faster, the space of deformations does not shrink to $\{L\}$, which would correspond to all the deformations decaying slower than $L$. 

Suppose $L'$ is any such deformation. Then there is a curve of special Lagrangians $L_s$ with $L_0 = L$ and $L_1 = L'$. The tangents to this curve are asymptotically translation invariant harmonic normal vector fields on $L_s$. The limits of the submanifolds $L_s$ lie within an open subset of the limit $K_0$ which we can identify; hence, for all of these we can find a uniform bound on the first nonzero eigenvalue of the Laplacian on the cross-section and it follows that there exists $\delta > 0$ (depending only on $L$) such that all the normal vector fields decay at least at rate $\delta$ (essentially by Nordstr\"om \cite[Proposition 5.7]{nordstromacyldeformations}). Hence, so does the whole curve $L_s$ and in particular $L'$.

This completes the proof of Theorem \ref{acyldeformationtheorem}. 
\section{Constructing an approximate special Lagrangian}
\label{sec:constructingapproxspeclag}
In this section, we apply the approximate gluing definitions given in subsection \ref{ssec:basicacylsubmanifolds} to the asymptotically cylindrical special Lagrangian submanifolds of section \ref{sec:defsandefs}, to construct a submanifold $L^T$. We will primarily be showing that $L^T$ is nearly special Lagrangian. We first discuss the hypothesis we shall assume for Calabi--Yau gluing, which can be obtained from \cite{su3g2story} in the case of $n=3$; then we deduce in Proposition \ref{ambientimplications} that $L^T$ is nearly special Lagrangian. 

We shall work with a pair of matching asymptotically cylindrical Calabi--Yau manifolds, and we shall assume the following gluing result for Calabi--Yau manifolds. 
\begin{hypothesis}[Ambient Gluing]
\label{hyp:ambientgluing}
Let $M_1$ and $M_2$ be a matching pair of asymptotically cylindrical Calabi--Yau manifolds. There exist $T_0>0$, a constant $\epsilon$ and a sequence of constants $C_k$ such that for gluing parameter $T>T_0$ there exists a Calabi--Yau structure $(\Omega^T, \omega^T)$ on the glued manifold $M^T$ constructed in Definition \ref{defin:newgluingmanifolds} with
\begin{gather}
\|\Omega^T - \gamma_T(\Omega_1, \Omega_2)\|_{C^k} + \|\omega^T - \gamma_T(\omega_1, \omega_2)\|_{C^k} \leq C_k e^{-\epsilon T}\label{eq:ambientnorm}, \\
[\Im \Omega^T] = [\gamma_T(\Im \Omega_1, \Im \Omega_2)], \qquad [\omega^T] =  c[\gamma_T(\omega_1, \omega_2)],\label{eq:cohomologyequality}
\end{gather}
for some $c>0$, where $\gamma_T$ is the patching map of closed forms defined in Definition \ref{defin:newformgluing}. 
\end{hypothesis}
This hypothesis says that we can perturb the approximate gluing $(\gamma_T(\Omega_1, \Omega_2), \gamma_T(\omega_1, \omega_2))$ to get a Calabi--Yau structure. \eqref{eq:ambientnorm} says that the perturbation from our approximate gluing to the glued structure is small, and \eqref{eq:cohomologyequality} says that these perturbations are basically by exact forms. 

Hypothesis \ref{hyp:ambientgluing} holds if $n=3$. This follows from the previous work in \cite{su3g2story}, especially Proposition 5.8. Note that that proposition would give \eqref{eq:cohomologyequality} for $\Re \Omega$, rather than $\Im \Omega$ as here: but this is fine as $(i\Omega, \omega)$ is a Calabi--Yau structure whenever $(\Omega, \omega)$ is.

If $n=4$, a gluing result for $SU(4)$ structures is known. For instance, Doi--Yotsutani have given a gluing result by passing through $\Spin(7)$: see \cite{doiyotsutanispin7}. It is not immediate from this work that we can achieve \eqref{eq:ambientnorm} and \eqref{eq:cohomologyequality}. Indeed, the analogue of \eqref{eq:cohomologyequality} for the $\Spin(7)$ structures is very unlikely to be true, as (as with $SU(4)$ structures) applying $\gamma_T$ to $\Spin(7)$ structures yields a four-form that is very unlikely to define a $\Spin(7)$ structure. We would have to show that there was an exact perturbation that was a $\Spin(7)$ structure, rather than merely any perturbation as Doi--Yotsutani use (by taking the nearest $\Spin(7)$ structure pointwise and then perturbing as in \cite[Theorem 13.6.1]{joycebook}). Even then, the $SU(4)$ structure induced from a $\Spin(7)$ structure is not unique, and further work would be required to get \eqref{eq:cohomologyequality}.  

For higher $n$ still, it is not immediately obvious that the complex structure parts of asymptotically cylindrical Calabi--Yau structures can be glued easily (because the set of decomposable complex $n$-forms is highly nonlinear). 

Hence, it seems unlikely that Hypothesis \ref{hyp:ambientgluing} (in particular \eqref{eq:cohomologyequality}) holds in the case of $n>3$. In Proposition \ref{ambientimplications} below, we will obtain the consequences of Hypothesis \ref{hyp:ambientgluing} for our approximately glued submanifold; we will then discuss what weaker conditions than Hypothesis \ref{hyp:ambientgluing} might yield these consequences. 

In Joyce's work on the desingularisation of cones (\eg \cite{joycecones3}), \eqref{eq:cohomologyequality} is not difficult to obtain. The gluing is carried out in such a way that $\omega|_L = 0$. For $[\Im \Omega|_L]$, we observe that $H^n(L)$ is one-dimensional so it suffices to consider the pairing with the homology class of $[L]$: but the desingularised submanifold $L$ is homologous to the original special Lagrangian and so this pairing is zero.

Hypothesis \ref{hyp:ambientgluing} has the following consequence, which may be proved by considering the compact parts and the neck separately, and noting that we have convergence either to the behaviour on $M_i$ or to the cylindrical behaviour. 
\begin{prop}
\label{ambientmetric}
Let $M_1$ and $M_2$ be a matching pair of asymptotically cylindrical Calabi--Yau manifolds. Let $g^T$ be the gluing of the asymptotically cylindrical metrics $g_1$ and $g_2$ on $M^T$ and $g(\Omega^T, \omega^T)$ be the metric given by the Calabi--Yau structures in Hypothesis \ref{hyp:ambientgluing}. Then there exist constants $C_k$ and $\epsilon$ such that 
\begin{equation}
\label{eq:ambientmetricdecay}
\|g^T - g(\Omega^T, \omega^T)\|_{C^k} \leq C_k e^{-\epsilon T},
\end{equation}
where these norms are taken with either metric. 
\end{prop}
Combining Proposition \ref{ambientmetric} with Proposition \ref{metricgluingconsistency} and the second fundamental form bound mentioned below it, we obtain also
\begin{cor}
\label{newgluedmetricsallthesame}
Let $M_1$ and $M_2$ be a matching pair of asymptotically cylindrical Calabi--Yau manifolds, and let $L_1$ and $L_2$ be a matching pair of asymptotically cylindrical submanifolds. Let $M^T$ and $L^T$ be the glued manifold and submanifold as in Definition \ref{defin:newsubmanifoldapproxgluing} and suppose Hypothesis \ref{hyp:ambientgluing} holds. 

We have two metrics on $L^T$. Firstly, we have a metric obtained by direct gluing of the metrics on $L_1$ and $L_2$ as in Definition \ref{defin:newgluingmanifolds}. Secondly, we have the metric given on $L^T$ by restricting the metric on $M^T$ induced by $(\Omega^T, \omega^T)$. The difference of these two metrics decays exponentially in $T$ to zero with all derivatives, with respect to either of them. Moreover, with respect to the metric induced from $(\Omega^T, \omega^T)$, the second fundamental form of $L^T$ in $M^T$ is bounded in $C^k$ for all $k$ uniformly in $T$. In particular, the restriction maps of $C^k$ $p$-forms are bounded independently of $T$. 
\end{cor}

Combining Hypothesis \ref{hyp:ambientgluing} with Lemma \ref{gluingconsistency} and Corollary \ref{newgluedmetricsallthesame}, we obtain 
\begin{prop}
\label{ambientimplications}
Suppose that $(M_1, M_2)$ is a pair of matching asymptotically cylindrical Calabi--Yau manifolds, and $L_1$ and $L_2$ are asymptotically cylindrical special Lagrangian submanifolds matching in the sense of Definition \ref{defin:newsubmanifoldapproxgluing}. Let $L^T$ be the glued submanifold and suppose that Hypothesis \ref{hyp:ambientgluing} holds. Then
\begin{enumerate}[i)]
\item $[\Im \Omega^T|_{L^T}] = 0 = [\omega^T|_{L^T}]$,
\item There exist constants $C_k$ and $\epsilon$ such that
\begin{equation}
\label{eq:ambientimplicationsnormest}
\|\Im \Omega^T|_{L^T}\|_{C^k} + \|\omega^T|_{L^T}\|_{C^k} \leq C_k e^{-\epsilon T}.
\end{equation}
\end{enumerate}
\end{prop}
That is, as $T$ becomes larger, $L^T$ becomes close to special Lagrangian. 

We now return to the question of what we may expect when $n>3$. We certainly need to suppose that our Calabi--Yau structures can be glued, and it seems likely that a perturbative argument ought to work. Hence, suppose that we have Calabi--Yau structures $(\Omega^T, \omega^T)$ and that \eqref{eq:ambientimplicationsnormest} holds; the question is how we may obtain something resembling $[\Im \Omega^T|_{L^T}] = 0 = [\omega^T|_{L^T}]$. 

If we have a K\"ahler class $[\omega^T_0]$ whose restriction to $L^T$ is zero, then, rescaling $\Omega^T$ if necessary, by the Calabi conjecture we may find $\omega^T_1$ in this class such that $(\Omega^T, \omega_1^T)$ is a Calabi--Yau structure and $[\omega_1^T|_{L^T}] = 0$. 

In particular, if the holonomy of $(\Omega^T, \omega^T)$ is exactly $SU(n)$, then we know that there are no parallel $(2, 0) + (0, 2)$ forms and hence by a Bochner argument no harmonic such forms; see, for instance, \cite[p.125]{joycebook}. With $J$ the complex structure corresponding to $(\Omega^T, \omega^T)$ we thus know by Hodge theory that  $J \gamma_T(\omega_1, \omega_2)$ lies in the same cohomology class as $\gamma_T(\omega_1, \omega_2)$, since they have the same harmonic representative. Then $\omega^T_0 = \frac12 (J \gamma_T(\omega_1, \omega_2) + \gamma_T(\omega_1, \omega_2))$ is a closed $(1, 1)$ form close to $\gamma_T(\omega_1, \omega_2)$, and in particular defines a K\"ahler metric. $\omega^T_0$ is in the cohomology class $[\gamma_T(\omega_1, \omega_2)]$, and so restricts to an exact form on $L^T$ just as before. We can then find a Ricci-flat metric in this K\"ahler class as above. 

We always have holonomy $SU(n)$ if $n$ is odd and $M^T$ is simply connected; see \cite[Proposition 6.2.3]{joycebook}. If $n$ is even, nevertheless the asymptotically cylindrical Calabi--Yau manifolds $M_1$ and $M_2$ have holonomy exactly $SU(n)$ if they are simply connected and irreducible (see \cite[Theorem B]{haskinsheinnordstrom}). It seems unlikely to be difficult to prove that then $M^T$ must have holonomy exactly $SU(n)$ as well. 

It remains to deal with the holomorphic volume form; we suppose $\omega_1^T$ has been found as above in general. Now, $[\Re \Omega^T]|_{L^T}$ and $[\Im \Omega^T]|_{L^T}$ lie in the one-dimensional space $H^n(L^T)$. Hence, there exists $\alpha$ such that $[\Im (e^{i\alpha} \Omega^T)]|_{L^T}$ is zero. Thus $(e^{i\alpha} \Omega^T, \omega_1^T)$ is a Calabi-Yau structure for which Proposition \ref{ambientimplications} (i) holds. 

It remains to verify (ii) for this structure, that is it remains to show that $\omega_1^T$ and $\alpha$ are close to $\omega^T$ and $0$ respectively. This, again, should not be difficult. 

In particular, assuming the details above can be filled in, the argument we give should extend to a gluing map of minimal Lagrangian submanifolds provided that the holonomy of $M^T$ is exactly $SU(n)$. Alternatively, we can glue special Lagrangians provided our $\omega_1^T$ can be chosen in general: but it may be necessary to choose this form depending on the submanifolds we are interested in. 
\section{The SLing map: existence and well-definition}
\label{sec:slingmap}
Suppose that $M_1$ and $M_2$ are a pair of matching asymptotically cylindrical Calabi--Yau manifolds, and $L_1$ and $L_2$ are a pair of matching asymptotically cylindrical special Lagrangian submanifolds and that Hypothesis \ref{hyp:ambientgluing} applies. Then we know from Proposition \ref{ambientimplications} that the submanifold $L^T$ constructed by gluing $L_1$ and $L_2$ as in Definition \ref{defin:newsubmanifoldapproxgluing} is close to special Lagrangian. To prove Theorem A, we now have to show that $L^T$ can be perturbed into a special Lagrangian submanifold. More generally, we seek to show that any such nearly special Lagrangian submanifold $L$ can be perturbed into a special Lagrangian submanifold, and to define a uniform ``SLing map" giving this perturbation. This is carried out in this section. We first state Condition \ref{slingcondition} on a submanifold $L$ of a Calabi--Yau manifold $M$, such that an inverse function (or contraction mapping) argument always gives that then $L$ can be perturbed into a special Lagrangian. One of the most important conditions is then dealt with in Lemma \ref{Jbehaves}. This will be absolutely essential for all our later work, and extends the isomorphism between one-forms and normal vector fields used by McLean to the nearly special Lagrangian case. We then show that the ``remainder term" can be sensibly bounded, in Proposition \ref{speclaggluingremainderestimate}, and complete the proof of Theorem A in Theorem \ref{slperturbthm}: $L^T$ satisfies Condition \ref{slingcondition} and so can be perturbed into a special Lagrangian. We then give a definition (Definition \ref{defin:slingmap}) of an ``SLing map" applicable whenever Condition \ref{slingcondition} holds. Finally, we deal with Theorem A1 (Theorem \ref{thm:finalslperturbthm}): we explain that Condition \ref{slingcondition} holds whenever $\Im \Omega|_L$ and $\omega|_L$ are sufficiently small and exact, with sufficient smallness depending on $L$, so that any nearly special Lagrangian submanifold can be perturbed to a special Lagrangian. 

\subsection{General setup}
In this first subsection, we identify the inverse function (or contraction mapping) argument that we will use and describe the required hypotheses in Condition \ref{slingcondition}. We shall then prove our first results towards obtaining them, Lemma \ref{Jbehaves} and Proposition \ref{parallelgluingprop}. Lemma \ref{Jbehaves} says firstly that if $L^T$ is nearly special Lagrangian then the map $v \mapsto \iota_v \omega|_L$ gives an isomorphism between normal vector fields and one-forms, generalising the special Lagrangian case, and secondly gives (not particularly explicit) bounds on the $C^{k, \mu}$ norms of this in both directions. Proposition \ref{parallelgluingprop} says that the linearisation of the map to which we apply the inverse function theorem does not change very much if we change the $SU(n)$ structure by a small amount. 

For the inverse function argument, we rely on essentially the same idea as in the deformation theory of subsection \ref{ssec:speclagdeformation}. That is, we consider the nonlinear map of \eqref{eq:mcleandeformationmap}
\begin{equation}
F:\begin{tikzcd}[row sep=0pt]
\{\text{normal vector fields on $L^T$}\} \ar{r} &\{\text{$n$-forms and $2$-forms on $L$}\}, \\[0pt]
v \ar[mapsto]{r} & ((\exp_v^* \Im \Omega^T)|_{L^T}, (\exp_v^* \omega^T)|_{L^T}).
\end{tikzcd}
\end{equation}
Since $(\Omega^T, \omega^T)$ is a torsion-free Calabi--Yau structure, this linearises as in subsection \ref{ssec:speclagdeformation} using Cartan's magic formula to give
\begin{equation}
D_0F: v \mapsto (d(\iota_v \Im \Omega^T)|_{L^T}, d(\iota_v \omega^T)|_{L^T}).
\end{equation}
However, since $L^T$ is not special Lagrangian, we do not have that $\iota_v \omega^T$ has no normal component, and we certainly do not have the equality $\iota_v \Im \Omega^T|_{L^T} = *(\iota_v \omega^T|_{L^T})$ of \cite[equation (3.7)]{mclean}. 

If we suppose that these do hold, and further that the inverse of the linearisation $D_0F$ can be controlled appropriately, then an inverse function argument would prove that if $\Im \Omega^T|_{L^T}$ and $\omega^T|_{L^T}$ are sufficiently small then $L^T$ can be perturbed to special Lagrangian. This is moreover the case if these equalities only nearly hold (indeed, since the second is only used for bounding the linearisation we not need to assume it separately), and thus we formalise this as
\begin{condition}[SLing conditions]
\label{slingcondition}
Let $M$ be a Calabi--Yau manifold. The closed submanifold $L$ satisfies the SLing conditions if for some $k$ and $\mu$
\begin{enumerate}[i)]
\item The map $u \mapsto \iota_u \omega|_L$ is an isomorphism between $C^{k+1, \mu}$ normal vector fields to $L$ and $C^{k+1, \mu}$ $1$-forms on $L$ and there exists $C_1 \geq 1$ such that 
\begin{equation}
\frac{1}{C_1} \|u\|_{C^{k+1, \mu}} \leq \|\iota_u \omega|_{L^T}\| \leq C_1 \|u\|_{C^{k+1, \mu}}.
\end{equation}
\item The linear map $u \mapsto (d\iota_u \omega|_L, d\iota_u \Im \Omega|_L)$ is an isomorphism from $C^{k+1, \mu}$ normal vector fields $u$ such that $\iota_u \omega|_L$ is $L^2$-orthogonal to harmonic forms (these normal vector fields will later be called orthoharmonic: see the discussion after Definition \ref{defin:laplacianonnvfs}) onto exact $C^{k, \mu}$ $2$- and $n$-forms, and there exists $C_2 \geq 1$ such that for any such $u$, 
\begin{equation}
\|u\|_{C^{k+1, \mu}} \leq C_2 (\|d\iota_u \omega|_L\|_{C^{k, \mu}} + \|d \iota_u \Im \Omega|_L\|_{C^{k, \mu}}).
\end{equation}
\item There exists $r>0$ such that whenever $\|u\|_{C^{k+1, \mu}} < r$ and  $\|v\|_{C^{k+1, \mu}} < r$, the remainder term satisfies the following bound for some constant $C_3$ depending on $r$
\begin{equation}
\begin{aligned}
\|(&\exp^*_u \Im \Omega|_L - \exp^*_v \Im \Omega|_L - d\iota_u \Im \Omega|_L + d\iota_v \Im \Omega|_L,\\& \exp^*_u \omega|_L - \exp^*_v \omega|_L - d\iota_u \omega|_L + d\iota_v \omega|_L)\|_{C^{k, \mu}} \\&\leq C_3\|u-v\|_{C^{k+1, \mu}}(\|u\|_{C^{k+1, \mu}}+\|v\|_{C^{k+1, \mu}}).
\end{aligned}
\end{equation}
\item $\|(\Im \Omega|_L, \omega|_L)\|_{C^{k, \mu}} \leq \min\{\frac{1}{8 C_1^2 C_2^2 C_3}, \frac{r}{2 C_1 C_2}\}$, and we have the cohomology conditions $[\Im \Omega|_L] = 0$ and $[\omega|_L] = 0$. 
\end{enumerate}
\end{condition}
If $L$ satisfies Condition \ref{slingcondition}, we can perturb it. A standard application of the contraction mapping principle yields the following precise implicit function theorem. 
\begin{prop}
\label{carefulift}
Suppose that $L$ is a closed submanifold of the Calabi--Yau manifold $M$ and that $L$ satisfies Condition \ref{slingcondition}. Then there exists a normal vector field $v$ to $L$ such that $\exp_v(L)$ is special Lagrangian. We have $\|v\|_{C^{k+1, \mu}} \leq 2C_1 C_2 \|(\Im \Omega|_L, \omega|_L)\|_{C^{k, \mu}}$; note that by Condition \ref{slingcondition} (iv), this is always less than $\min\{\frac{1}{4 C_1 C_2 C_3}, r\}$. 
\end{prop}
It therefore suffices to show that if $M_1$ and $M_2$ are a matching pair of Calabi--Yau manifolds and $L_1$ and $L_2$ a matching pair of special Lagrangian submanifolds and Hypothesis \ref{hyp:ambientgluing} holds, then the submanifold $L^T$ of $M^T$ constructed as in Definition \ref{defin:newsubmanifoldapproxgluing} satisfies Condition \ref{slingcondition}. 

Parts (iii) and (iv) of Condition \ref{slingcondition} are relatively straightforward; we need to discuss parts (i) and (ii). We begin with (i). We know from McLean \cite[Remark 3-4]{mclean} that $v \mapsto \iota_v \omega|_L$ defines an isomorphism between normal vectors to a special Lagrangian submanifold, and cotangent vectors to this submanifold. Moreover, this map is an isometry, so that if $L$ is special Lagrangian (i) holds with $C_1 = 1$. We extend this to the case where $(\Omega, \omega)$ is only defined locally and $L$ is close to special Lagrangian. 

\begin{lem}
\label{Jbehaves}
Suppose that $(\Omega, \omega)$ is an $SU(n)$ structure around $L$ in the sense of Definition \ref{defin:cystructurebundle} and $p$ is a point of $L$ such that $|(\omega|_L)_p|<1$ (i.e. $|\omega(u, v)| < |u||v|$ for all $u, v \in T_pL$). Then the complex structure $J_p: T_pM \to T_pM$ does not take any tangent vector to another tangent vector, and at $p$ the map $u \mapsto \iota_u \omega|_L$ from normal vectors to tangent covectors is an isomorphism.

Now suppose the uniform norm $\|\omega|_L\|_{C^0} < 1$, so that the preceding paragraph holds at all points. Then $v \mapsto \iota_v \omega|_L$ defines a map from smooth normal vector fields on $L$ to smooth $1$-forms on $L$. Provided that $\|\omega|_L\|_{C^{k, \mu}}$ is sufficiently small compared to the $C^{k, \mu}$ norm of $J$ and the $C^k$ norm of the second fundamental form of $L$ in $M$, we have bounds
\begin{equation}
\label{eq:Jbehavesbounds}
c\|v\|_{C^{k, \mu}} \leq \|\iota_v \omega|_L\|_{C^{k, \mu}} \leq  C\|v\|_{C^{k, \mu}},
\end{equation}
where the constants $c$ and $C$ depend on the $C^{k, \mu}$ norms of $\omega|_L$ and the induced almost complex structure $J$ and the $C^k$ norm of the second fundamental form of $L$ in $M$. 
\end{lem}
\begin{proof}
For the first part, we work at $p \in L$. Suppose that $v \in T_pL \subset T_pM$ such that $Jv \in T_pL$. $J_p$ is an isometry and so
\begin{equation}
|v|^2 = g(v, v) = \omega(v, J_pv) < |v||Jv| = |v|^2.
\end{equation}
This is a contradiction, proving the first claim.

It follows that the map from tangent vectors to normal vectors given by taking the normal component of $J_pv$ is an isomorphism as it is an injective linear map between spaces of the same dimension. That is, given a normal vector $v$ we can find a tangent vector $u$ such that $v$ is the normal component of $Ju$; \ie we can find a pair $u$ and $u'$ of tangent vectors such that $v = J_pu + u'$. 

To show $u \mapsto \iota_u \omega|_L$ is an isomorphism, we show that it too is injective. It is easy to see that if not there exists $u$ such that $J_pu$ is normal. As above, we can find tangent vectors $v$ and $v'$ such that $u = J_pv + v'$, and hence $J_pv = v' - u$. Applying $J_p$ to this and rearranging, 
\begin{equation}
J_pv' = -v - J_pu.
\end{equation}
Since $v$ and $v'$ are tangential, $u$ and $J_pu$ are normal and nonzero, and $J_p$ is an isometry, we obtain 
\begin{equation}
|v| < |-v-J_pu| = |J_pv'| = |v'| < |u +v'| = |J_pv| = |v|,
\end{equation}
which is evidently a contradiction. This completes the first part of the proof. 

For the second part, we shall work with the norms on sections of $TM|_L$ induced by the ambient connection; as described after Theorem \ref{localrestrictionthm}, the restrictions of these norms to normal vector fields are Lipschitz equivalent to the standard norms, with Lipschitz constant depending on the $C^k$ norm of the second fundamental form. 

We introduce the notations $\pi^1_1$ and $\pi^1_0$ for the normal and tangent parts of vector fields and one-forms. By the remark on Lipschitz equivalence after, and the argument indicated for, Theorem \ref{localrestrictionthm}, we obtain 
\begin{equation}
\|\pi^1_i \alpha\|_{C^{k, \mu}} \leq C' \|\alpha\|_{C^{k, \mu}},
\end{equation}
where $C'$ depends on the $C^k$ norm of the second fundamental form. The upper bound in \eqref{eq:Jbehavesbounds} follows immediately; it remains to show the lower bound. 

To do this, we use a similar idea to the first part. Given a normal vector field $v$, we can find local tangent vector fields $u$ and $u'$ such that 
\begin{equation}
\label{eq:vJuu}
v = Ju + u'.
\end{equation}
Now we note that $u'$ is small: if $w$ is also a local tangent vector field, taking the inner product of $w$ with \eqref{eq:vJuu} yields
\begin{equation}
0 = \omega(u, w) + g(u', w).
\end{equation}
$u$ and $w$ are both tangential, so $\|\omega(u, w)\|_{C^{k, \mu}}$ can be bounded by  $\|\omega|_L\|_{C^{k, \mu}}\|u\|_{C^{k, \mu}}\|w\|_{C^{k, \mu}}$. Hence, since the $C^{k, \mu}$ norm of a covector field can be identified as its operator norm as a map from $C^{k, \mu}$ vector fields to $C^{k, \mu}$ functions, we have that $\|u'\|_{C^{k, \mu}}$ can be bounded by $\|u\|_{C^{k, \mu}}\|\omega|_L\|_{C^{k, \mu}}$. In particular, we find
\begin{equation}
\|u\|_{C^{k, \mu}} \leq C_1 \|Ju\|_{C^{k, \mu}} \leq C_1 \|v\|_{C^{k, \mu}} + C_1 \|u\|_{C^{k, \mu}}\|\omega|_L\|_{C^{k, \mu}},
\end{equation}
where $C_1$ depends on the $C^{k, \mu}$ norm of $J$, and hence that if $C_1 \|\omega|_L\|_{C^{k, \mu}}$ is sufficiently small
\begin{equation}
\|u\|_{C^{k, \mu}} \leq \frac{C_1} {1-C_1 \|\omega|_L\|_{C^{k, \mu}}} \|v\|_{C^{k, \mu}}.
\end{equation}

Now we apply $J$ to \eqref{eq:vJuu}, obtaining
\begin{equation}
Jv = -u + Ju'.
\end{equation}
Consequently, the normal part of $Jv $ is the normal part of $Ju'$ and we have the estimate
\begin{equation}
\begin{aligned}
\|\pi^{1}_1 Jv\|_{C^{k, \mu}} &\leq C_2 \|J u'\|_{C^{k, \mu}} \leq C_3 \|u'\|_{C^{k, \mu}} \\&\leq C_3 \|u\|_{C^{k, \mu}}\|\omega|_L\|_{C^{k, \mu}} \leq \frac{C_1C_3 \|\omega|_L\|_{C^{k, \mu}}} {1-C_1 \|\omega|_L\|_{C^{k, \mu}}} \|v\|_{C^{k, \mu}},
\end{aligned}
\end{equation}
where $C_2$ depends on the $C^k$ norm of the second fundamental form and $C_3$ depends on the $C^k$ norm of the second fundamental form and the $C^{k, \mu}$ norm of $J$. If $\frac{C_1C_3\|\omega|_L\|_{C^{k, \mu}} } {1-C_1 \|\omega|_L\|_{C^{k, \mu}}}$ is sufficiently small we then have the desired lower bound
\begin{equation}
\|\pi^{1}_0 Jv\|_{C^{k, \mu}} \geq \left( 1- \frac{C_1C_3 \|\omega|_L\|_{C^{k, \mu}}}{1-C_1 \|\omega|_L\|_{C^{k, \mu}}}\right) \|v\|_{C^{k, \mu}}.\qedhere
\end{equation}
\end{proof}
We now turn to part (ii) of Condition \ref{slingcondition}. To understand this, we will show that $v \mapsto \iota_v \Im \Omega^T|_{L^T}$ and $v \mapsto  *\iota_v {\omega^T}|_{L^T}$ are similar. 

We will work locally, and take a different $SU(n)$ structure $(\Omega', \omega')$ around an open subset $U$ of $L^T$, again in the sense of Definition \ref{defin:cystructurebundle}, so that $U$ is special Lagrangian with respect to $(\Omega', \omega')$, and then we know from \cite[equation (3.7)]{mclean} that $v \mapsto \iota_v \Im \Omega'|_{U} = v \mapsto *' \iota_v  \omega'|_{U}$. Then we have to show that $v \mapsto \iota_v \Im \Omega^T|_{U}$ is close to $v \mapsto \iota_v \Im \Omega'|_{U}$ and similarly for $v \mapsto * \iota_v \omega^T|_{U}$. That these follow provided that $(\Omega', \omega')$ is close enough to $(\Omega^T, \omega^T)$ is the content of the following. 

\begin{prop}
\label{parallelgluingprop}
Suppose that $M$ is a $2n$-dimensional manifold, and $L$ is an $n$-dimensional submanifold. Let $\SUn(M)$ be the bundle of $SU(n)$ structures over $M$ from Definition \ref{defin:cystructurebundle}. 

Suppose that $(\Omega_1, \omega_1)$ and $(\Omega_2, \omega_2)$ are two sections of $\SUn(M)|_L$. Assume further that $(\Omega_1, \omega_1)$ is the restriction to $L$ of a Calabi--Yau structure on $M$, so that there is a well-defined metric $g$ on $M$, giving a well-defined $C^k$ norm on $TM|_L$, and suppose that the second fundamental form with respect to $g$ of $L$ in $M$ is bounded in $C^{k-1}$ by $R$.

Then there exist $C$ and $\delta_0$ depending on $R$ such that if 
\begin{equation}
\|\Omega_1 - \Omega_2\|_{C^k} + \|\omega_1 - \omega_2\|_{C^k} < \delta < \delta_0,
\end{equation}
then 
\begin{equation}
\label{eq:pgpfinal}
\begin{aligned}
&\|(u \mapsto *_1(\iota_u \omega_1|_L)) - (u \mapsto *_2(\iota_u \omega_2|_L))\|_{C^k} \\&+ \|(u \mapsto \iota_u \Im \Omega_1|_L) - (u \mapsto \iota_u \Im \Omega_2|_L)\|_{C^k} < C\delta,
\end{aligned}
\end{equation}
where these norms are the induced norms on the bundle $\nu^*_L \otimes \bigwedge^{n-1} T^*L$.
\end{prop}
\begin{proof}
Since the map $(\Omega, \omega) \mapsto g$, contraction, and the map $g \mapsto *_g$ are smooth bundle maps, 
\begin{equation}
\label{eq:pgpmap}(\Omega, \omega) \mapsto (u \mapsto *\iota_u \omega|_L - u \mapsto \iota_u \Im \Omega|_L)
\end{equation}
is a smooth bundle map. We may then apply Proposition \ref{bundlemapsckbounded} using the ambient Levi-Civita connection. Since $(\Omega_1, \omega_1)$ is parallel and $\delta_0$ is small, this proposition applies provided $(\Omega_1, \omega_1)$ and $(\Omega_2, \omega_2)$ lie in a compact subset of $\SUn(M)$. By trivialising so that $(\Omega_1, \omega_1)$ is the standard structure, we may assume this locally. It follows just as in Proposition \ref{bundlemapsckbounded} by extending to a smooth, hence Lipschitz continuous, map of jet bundles that \eqref{eq:pgpfinal} holds when the $C^k$ norms are those induced using the ambient connection. 

But as described after Theorem \ref{localrestrictionthm}, the $C^k$ norms induced using the ambient connection are Lipschitz equivalent with a constant depending on the $C^{k-1}$ norm of the second fundamental form to the $C^k$ norms induced using the restricted connection, which is bounded by $R$. \eqref{eq:pgpfinal} with the $C^k$ norms induced using the restricted connection follows. 
\end{proof}
\subsection{The gluing theorem}
In this subsection we combine Theorem \ref{laplacelowerbound} with a special case of Proposition \ref{parallelgluingprop} to complete the proof of Theorem A by proving Theorem \ref{slperturbthm}, that the approximate special Lagrangian constructed in section \ref{sec:constructingapproxspeclag} can be perturbed to be special Lagrangian. As a preliminary, we will consider the nonlinearity and show that Condition \ref{slingcondition} (iii) holds with a constant independent of $T$.

We are interested in the map $u \mapsto (\exp_u^* \Im \Omega|_L - d\iota_u \Im \Omega|_L, \exp_u^* \omega|_L - d\iota_u \omega|_L)$. We note that the restriction to $L$ is always controlled by the second fundamental form, and so it suffices to consider the rest of the map, which we formalise as follows. 

\begin{defin}
Let $\F$ be the subset of $(J^1 TM \oplus J^1 \bigwedge^n T^*M \oplus J^1 \bigwedge^2 T^*M) \times (\bigwedge^n T^*M \oplus \bigwedge^2 T^*M)$ consisting of pairs $(((u, \nabla u), (\alpha_0, \nabla \alpha_0),  (\beta_0, \nabla \beta_0)), (\alpha, \beta))$ such that 
\begin{equation}
\pi_{\bigwedge^n T^*M \oplus \bigwedge^2 T^*M} (\alpha, \beta) = \exp(u).
\end{equation}

\end{defin}
It is easy to see
\begin{lem}
\label{gluingpointwisespacesubmanifold}
The intersection of $\F$ with the set of elements with $(u, \nabla u)$ small is a submanifold and is a smooth bundle over $M$. 
\end{lem}
This simply follows because this open subset of $\F$ is the preimage of the diagonal submanifold of $M \times M$ under an obvious submersion. 

We may now make
\begin{defin}
\label{defin:remaindermapG}
Let $G$ be the bundle map from $\F$ to $\bigwedge^n T^*_pM \oplus \bigwedge^2 T^*_pM$ given by
\begin{equation}
(((u, \nabla u), (\alpha_0, \nabla \alpha_0),  (\beta_0, \nabla \beta_0)), (\alpha, \beta)) \mapsto (\exp_u^* \alpha - d\iota_u \alpha_0, \exp_u^* \beta - d\iota_u \beta_0).
\end{equation}
\end{defin}
Note that $\exp_u^*$ does only depend on the first jet. Note further that isometrically embedding an incomplete manifold $M$ into $M'$ will enlarge $\F$, as more points $\exp_u(p)$ will be in $M'$. The extended map $G$ will of course depend on the metric on $M' \setminus M$. This idea will be needed below to deal with behaviour on the neck of a glued manifold. 

We may then prove
\begin{prop}
\label{basicgluingsmoothness}
The map $G$ is smooth. Moreover, it depends continuously on the metric on $M$ in the sense that  if $g_s$ is a finite-dimensional family of metrics, then $G$ depends continuously on $s$. Moreover, we have that for every $\epsilon$, $\delta_0$ and $p$ there exists $\delta$ depending on $p$, $\epsilon$ and $\delta_0$ such that if $\pi(u) = p$, $|s-s'| < \delta$, and $|(u, \nabla u)| < \delta_0$
\begin{equation}
\label{eq:newcomplicatedsctyest}
\begin{aligned}
&|G(s, (((u, \nabla u), (\alpha_0, \nabla \alpha_0),  (\beta_0, \nabla \beta_0)), (\alpha, \beta))) \\&- G (s', (((u, \nabla u), (\alpha_0, \nabla \alpha_0),  (\beta_0, \nabla \beta_0)), (\alpha, \beta)))| \\<& \epsilon|(u, \nabla u)| |(\alpha, \beta)|.
\end{aligned}
\end{equation}
In particular, if there is a metric $g_\infty$ such that $g_s \to g_\infty$ as $s \to \infty$, then $G$ converges to the corresponding $G_\infty$ as $s \to \infty$, and we have that for every $\epsilon$, $\delta_0$ and $p$ there exists $K$ such that if $\pi(u) = p$, $s > K$ and $|(u, \nabla u)| < \delta_0$, \eqref{eq:newcomplicatedsctyest} holds. 
\end{prop}
\begin{proof}
The contraction and exterior derivative terms are obviously smooth, and are independent of the metric. The difficulty is the terms $\exp_u^* \alpha$ and $\exp_u^* \beta$. We extend the jet $(u, \nabla u)$ to a local field, so defining a local diffeomorphism. 

We begin by showing that the pushforward map $(\exp_u)_*$ from $T_pM$ to $T_{\exp_p(u)}M$ has these properties (viewed as a map on $J^1(TM) \oplus TM$); then we may dualise. Given $((u, \nabla u), v)$ at $p \in M$, by hypothesis we have a unique geodesic $\gamma$ with initial velocity $u$. Construct the Jacobi field $X$ along $\gamma$ with $X_0 = v$ and $(\nabla_\ddt X)_0 = \nabla_v u$. It is straightforward to see that the pushforward $(\exp_u)_* v$ is precisely the final value of this Jacobi field.

To prove that pushforward has the desired properties, therefore, we just have to show that this Jacobi field does. But locally, we are just solving a system of ordinary differential equations. The smoothness of the map then reduces to the fact that the solution to a system of ordinary differential equations depends smoothly on the initial conditions, which is well-known; see \cite[chapter 2, Corollaries 9 and 10]{arnoldodes}. Similarly, continuity in $s$ for a smooth family of metrics $g_s$ is simply continuity in the finite-dimensional parameter $s$, and again this is well-known. The analogue to \eqref{eq:newcomplicatedsctyest} follows by noting that the map giving the derivative of pushforward in $u$ is also continuous in $s$ and that the pushforward is independent of the metric if $(u, \nabla u) = 0$.

Dualising by choosing a smooth local trivialisation of $TM$ immediately gives the corresponding results for the pullback map from the relevant components of $\F$ to $T^* M$; the final result is then immediate. 
\end{proof}
\begin{rmk}
Similarly, in \cite[Definition 5.5]{joycecones3}, Joyce prepares for analysis of the nonlinear term by passing to work with finite-dimensional spaces. He goes on to give a more direct analysis of this term than we will; we will only show that we can bound things uniformly in $T$. 
\end{rmk}
We now argue directly using this. 
\begin{prop}
\label{speclaggluingremainderestimate}
Let $M_1$ and $M_2$ be a matching pair of asymptotically cylindrical Calabi--Yau manifolds and let $L_1$ and $L_2$ be a matching pair of asymptotically cylindrical special Lagrangian submanifolds. Let $L^T$ be the approximate gluing of $L_1$ and $L_2$ defined in Definition \ref{defin:newsubmanifoldapproxgluing}. Suppose Hypothesis \ref{hyp:ambientgluing} applies, so that $M^T$ is Calabi--Yau with a suitable Calabi--Yau structure $(\Omega^T, \omega^T)$. Then for any $k$ and $\mu$, for $T$ sufficiently large,  Condition \ref{slingcondition}(iii) holds with constants $C_3$ and $r$ independent of $T$. 
\end{prop}
\begin{proof}
We may work locally around each point $p$ of $L^T$, and look for a pointwise estimate,
\begin{equation}
\begin{aligned}
&|G(((u, \nabla u), (\Im \Omega_p, \nabla \Im \Omega_p), (\omega_p, \nabla \omega_p)), (\Im \Omega_{\exp_u(p)}, \omega_{\exp_u(p)}))  \\&- G(((v, \nabla v), (\Im \Omega_p, \nabla \Im \Omega_p), (\omega_p, \nabla \omega_p)), (\Im \Omega_{\exp_v(p)}, \omega_{\exp_v(p)}))| \\<& C (|(u, \nabla u)| + |(v, \nabla v)|)|(u-v, \nabla u - \nabla v)|.
\end{aligned}
\end{equation}
As $G$ has zero linearisation around $u = 0$, and all of these maps are smooth in $p$ (including that the sections $\Omega$ and $\omega$ are),  some such local constant can be found for each $p$ when we restrict to $(u, \nabla u)$ in a ball around zero of radius $r_p$.  Since $M^T$ is compact, we can find $C$ and $r_p$ independently of $p$ for each $T$; it remains to prove that they are independent of $T$, as the difference between this and the remainder term we wish to bound is just the restriction to $L^T$, which we know is uniformly  bounded in $T$. 

We appeal to the continuity of $G$ and hence of these constants in the metric. Similarly, by differentiating $G$ we obtain an estimate for the derivatives, analogously to Proposition \ref{bundlemapsckbounded}, and using \eqref{eq:newcomplicatedsctyest} these derivatives are also continuous, so the estimates can again by chosen continuously. We will note when we do this that $(\Omega, \omega)$ also converge. The most obvious space of parameters to consider is just $T$ (large enough). Unfortunately, as $T \to \infty$ the metric on $M^T$ becomes increasingly singular and so there is no $g_\infty$ with which we may compare. Hence we need to be slightly more careful, and we work locally. 

We begin by choosing $T'_0 > T_0$, with $T_0$ from Hypothesis \ref{hyp:ambientgluing}, and considering $G$ on $M_1^{\tr T'_0} \subset M^T$. $M_1^{\tr T'_0}$ is incomplete, so we further consider it as a subspace of $M_1^{\tr (T'_0+1)}$ as indicated above to make sure $\F$ is large enough; $G$ is then well-defined on $M_1^{\tr T'_0}$ for $(u, \nabla u)$ small enough, uniformly in $t$. Now by Hypothesis \ref{hyp:ambientgluing}, the Calabi--Yau structure $(\Omega^T|_{M_1^{\tr (T'_0 +1)}}, \omega^T|_{M_1^{\tr (T'_0 + 1)}})$ converges to $(\Omega_1|_{M_1^{\tr (T'_0 + 1)}}, \omega_1|_{M_1^{\tr (T'_0 + 1)}})$ as $T \to \infty$; hence, the metric converges to the metric $g_1$ induced by $(\Omega_1, \omega_1)$.  $C$ and $r_p$ can now be bounded independently of $T$ on $M_1^{\tr T'_0}$, by first finding $C$ and $r_p$ for $g_1$ and using smoothness in the forms and the continuity part of Proposition \ref{basicgluingsmoothness}. Similar arguments apply for $M_2^{\tr T'_0}$. 

If $T < T'_0 -\frac12$, then $M_1^{\tr T'_0}$ and $M_2^{\tr T'_0}$ intersect and so we have covered all of $M^T$. Otherwise, it remains to consider the subset $(-T - 1 + T'_0, T+1 - T'_0) \times N$ of the neck of $M^T$. Since we want to work on a fixed manifold, we shall consider $(-1, 1) \times N$, and include this as the subset $(t - 1, t+1) \times N$ for $|t| < T - T'_0$. Again, the Calabi--Yau structures $(\Omega^T, \omega^T)|_{(t-1, t+1) \times N}$ define incomplete metrics, and so to make $\F$ large enough we extend to the corresponding $(t-2, t+2) \times N$, and then $G$ applied to sufficiently small tangent vectors over $(-1, 1) \times N$ depends continuously on the metric on $(-2, 2) \times N$. 

Now $(\Omega^T, \omega^T)|_{(-t-2, t+2) \times N}$ gives a family of Calabi--Yau structures on $(-2, 2) \times N$. This family can be parametrised by the pair $(t, T)$ in the set $\{T \geq T'_0 - \frac12, |t+1| < T - T'_0\}$. Clearly, for each $t$ as $T$ approaches infinity, we approach the cylindrical metric $\tilde g$ on $(-2, 2) \times N$ and the Calabi--Yau structure approaches the cylindrical Calabi--Yau structure. Thus by choosing $T'_0$ large enough we may consider the $g|_{(-t-2, t+2) \times N}$ as perturbations of $\tilde g$ and the forms as perturbations of the cylindrical forms. Hence, we can find $C$ and $r_p$ independent of both $t$ and $T$.

This provides the required constants uniformly in $T$.
\end{proof}
We may now prove Theorem A.
\begin{thm}[Theorem A]
\label{slperturbthm}
Let $M_1$ and $M_2$ be a matching pair of asymptotically cylindrical Calabi--Yau manifolds and let $L_1$ and $L_2$ be a matching pair of asymptotically cylindrical special Lagrangian submanifolds. Let $L^T$ be the approximate gluing of $L_1$ and $L_2$ defined in Definition \ref{defin:newsubmanifoldapproxgluing}. Suppose Hypothesis \ref{hyp:ambientgluing} applies, so that $M^T$ is Calabi--Yau with a suitable Calabi--Yau structure $(\Omega^T, \omega^T)$. Then for any $k$ and $\mu$, for $T$ sufficiently large,  Condition \ref{slingcondition} holds for the submanifold $L^T$ of $M^T$ so there is a normal vector field $v$ such that $\exp_v(L^T)$ is a special Lagrangian submanifold for $(\Omega^T, \omega^T)$. $v$ is smooth and decays exponentially with $T$, in the sense that there is $\epsilon>0$ and a sequence $C_k$ so that $\|v\|_{C^k} \leq C_k e^{-\epsilon T}$. 
\end{thm}
\begin{proof}
Proposition \ref{ambientimplications} gives that $\|\omega^T|_{L^T}\|$ decays exponentially. Lemma \ref{Jbehaves} gives precisely Condition \ref{slingcondition}(i), with constant $C_1$ depending on how small $\omega^T|_{L^T}$ is, provided it is small enough. Hence, Condition \ref{slingcondition}(i)  holds for $T$ sufficiently large and $C_1$ can be taken uniform in $T$. Proposition \ref{speclaggluingremainderestimate} says that Condition \ref{slingcondition}(iii) holds for $T$ sufficiently large with $C_3$ bounded independently of $T$. It remains to show that Condition \ref{slingcondition}(ii) and (iv) hold. We shall first show that (ii) holds and that $C_2$ grows at most polynomially in $T$. Proposition \ref{ambientimplications} will then precisely give (iv). That is, we will have proved that Condition \ref{slingcondition} holds for $L^T$, so by Proposition \ref{carefulift} there is such a $v$ perturbing $L^T$ to a special Lagrangian. The norm of $v$ is controlled again by Proposition \ref{carefulift}: since $C_1$, $C_2$ and $C_3$ grow at most polynomially, and Proposition \ref{ambientimplications} says that $\|\Im \Omega^T|_{L^T}\| + \|\omega^T|_{L^T}\|$ decays exponentially in $T$, this norm decays exponentially in $T$. 

To prove Condition \ref{slingcondition}(ii), we consider the two linear maps 
\begin{equation}
\label{eq:linearisation}
v \mapsto (d\iota_v \omega^T|_{L^T}, d\iota_v \Im \Omega^T|_{L^T})
\end{equation}
and 
\begin{equation}
\label{eq:sllaplacian}
v \mapsto (d\iota_v \omega^T|_{L^T}, d* \iota_v \omega^T|_{L^T}).
\end{equation}
Corollary \ref{newgluedmetricsallthesame} says that the metric of $(\Omega^T, \omega^T)$ is close to the glued metric. Applying Theorem \ref{laplacelowerbound} on a lower bound for $d+d^*$ and Condition \ref{slingcondition}(i), we see \eqref{eq:sllaplacian} is an isomorphism with a lower bound of the form 
\begin{equation}
\label{eq:polynomiallowerbound}
\|u\|_{C^{k+1, \mu}} \leq CT^r ( \|d\iota_u \omega^T|_{L^T}\|_{C^{k, \mu}} + \| d* \iota_v \omega^T|_{L^T}\|_{C^{k, \mu}}).
\end{equation}
We shall apply Proposition \ref{parallelgluingprop} to show that the difference of \eqref{eq:linearisation} and \eqref{eq:sllaplacian} is exponentially small in $T$. It will follow by openness of isomorphisms of Banach spaces that \eqref{eq:linearisation} also has a lower bound of the form \eqref{eq:polynomiallowerbound}, which is Condition \ref{slingcondition}(ii) with $C_2$ growing at most polynomially in $T$. 

\eqref{eq:linearisation} and \eqref{eq:sllaplacian} are the composites with the exterior derivative of
\begin{equation}
\label{eq:predlinearisation}
v \mapsto (\iota_v \omega^T|_{L^T}, \iota_v \Im \Omega^T|_{L^T})
\end{equation}
and 
\begin{equation}
\label{eq:predsllaplacian}
v \mapsto (\iota_v \omega^T|_{L^T}, *\iota_v \omega^T|_{L^T}).
\end{equation}
To show that \eqref{eq:linearisation} and \eqref{eq:sllaplacian} are exponentially close as maps from $C^{k+1, \mu}$ normal vector fields to $C^{k, \mu}$ forms, therefore, it suffices to show that \eqref{eq:predlinearisation} and \eqref{eq:predsllaplacian} are exponentially close as maps from $C^{k+1, \mu}$ normal vector fields to $C^{k+1, \mu}$ forms, and this is an application of Proposition \ref{parallelgluingprop}.

Proposition \ref{parallelgluingprop} is essentially a local result. We pick some fixed $A(T)$ decaying exponentially in $T$, and have to show around every point $p$ of $L^T$ we can find an open neighbourhood $U \cap L^T$ of $p$ and an $SU(n)$ structure $(\Omega_U, \omega_U)$ around $U \cap L^T$ with respect to which $U \cap L^T$ is special Lagrangian and so that 
\begin{equation}
\|\Omega_U - \Omega^T\|_{C^{k+2}(U \cap L^T)} + \|\omega_U - \omega^T\|_{C^{k+2}(U \cap L^T)} \leq A(T).
\end{equation}

Proposition \ref{parallelgluingprop} would then imply that, as maps from $C^{k+2}$ to $C^{k+2}$, \eqref{eq:predlinearisation} and \eqref{eq:predsllaplacian} differ by $A(T) B$ for some fixed constant $B$. Note that $B$ may depend on the second fundamental form of $L^T$ in $M^T$, but since we know that the second fundamental form is bounded uniformly in $T$ by Corollary \ref{newgluedmetricsallthesame}, $A(T) B$ also decays exponentially in $T$. 

On the other hand, \eqref{eq:predlinearisation} and \eqref{eq:predsllaplacian} are function-linear and so are given by tensors. Thus, this bound is a bound on the difference of the tensors, and so maps also differ by at most $A(T)B$ as maps from $C^{k+1, \mu}$ normal vector fields to $C^{k+1, \mu}$ forms. As $A(T) B$ is exponentially decaying, we would obtain that \eqref{eq:predlinearisation} and \eqref{eq:predsllaplacian} are exponentially close, and so that Condition \ref{slingcondition}(ii) holds with $C_2$ growing at most polynomially.

We shall now find $(\Omega_U, \omega_U)$. In order to do this, we will work with three open subsets $U_1$, $U_2$ and $U_3$ forming an open cover of $M^T$, so that $\{U_i \cap L^T\}$ form an open cover of $L^T$. For each $U_i$ we may find $A_i(T)$ decaying exponentially, and by taking the maximum we obtain $A(T)$ decaying exponentially. 

Specifically, let $U_1 = M_1^{\tr (T-2)} \subset M^T$, $U_2 = M_2^{\tr (T-2)} \subset M^T$, and $U_3$ be the subset $(-3, 3) \times N$ of the neck. 

We first deal with $U_i$ for $i=1, 2$. Here, we have by construction that 
\begin{equation}
\gamma_T(\Omega_1, \Omega_2) = \Omega_i, \qquad \gamma_T(\omega_1, \omega_2) = \omega_i,
\end{equation}
and for $T$ sufficiently large $L^T \cap U_i = L_i \cap U_i$. Consequently, $(\Omega_i, \omega_i)$ defines an $SU(n)$ structure around $U_i  \cap L^T$ for which $U_i \cap L^T$ is special Lagrangian. By Hypothesis \ref{hyp:ambientgluing}, $(\Omega_i, \omega_i) = (\gamma_T(\Omega_1, \Omega_2), \gamma_T(\omega_1, \omega_2))$ is exponentially close to $(\Omega_T, \omega_T)$ on $U_i \cap L^T$, so we are done. 

It remains to find an $SU(n)$ structure around $U_3 \cap L^T$ with the desired properties. We know that $U_3 \cap L^T$ is the intersection with $U_3$ of the image under an exponentially small normal vector field of $K \times (T-4, T+4)$ with respect to the asymptotically cylindrical metrics. Note that where these metrics are not defined or do not agree, the vector field is zero, so this is a well-defined notion. Since the glued metric is exponentially close to the asymptotically cylindrical metrics, this is still the case for the glued metric. That is, there is some normal vector field $v_T$, so that $L^T \cap U_3 = \exp_{v_T}(K \times (T-4, T+4)) \cap U_3$ and $v_T$ decays exponentially in $T$. 

Note that $K \times (T-4, T+4)$ is special Lagrangian with respect to the cylindrical $SU(n)$ structure $(\tilde \Omega, \tilde \omega)$. By construction, we have that $(\gamma_T(\Omega_1, \Omega_2), \gamma_T(\omega_1, \omega_2))$ is exponentially close to $(\tilde \Omega, \tilde \omega)$ on $U_3$ and by Hypothesis \ref{hyp:ambientgluing} again we consequently have that $(\tilde \Omega, \tilde \omega)$ is exponentially close to $(\Omega^T, \omega^T)$ on $U_3$. 

We may extend $v_T$ to a tubular neighbourhood $V_1$ of $K \times (T-4, T+4)$, containing $N \times (T-4, T+4) \cap L^T$, which $\exp_{v_T}$ maps locally diffeomorphically to a tubular neighbourhood $V_2$ of $L^T \cap U_3$. We may choose this extension, which we shall also call $v_T$, to also decay exponentially in $T$. Let $F$ be the inverse of the diffeomorphism $\exp_{v_T}: V_1 \to V_2$. Let  our $SU(n)$ structure $(\Omega_{U_3}, \omega_{U_3})$ around $L^T \cap U_3$ be $F^*(\tilde \Omega, \tilde \omega)$ (which we think of as defined on $V_1$ for simplicity). Since $v_T$ is exponentially small, and $\tilde \Omega$ and $\tilde \omega$ are bounded with their derivatives independently of $T$, 
\begin{equation}
(\Omega_{U_3}, \omega_{U_3}) - (\tilde \Omega, \tilde \omega)
\end{equation}
is exponentially small, and so too is 
\begin{equation}
\label{eq:cylindricalisedlessproper}
(\Omega_{U_3}, \omega_{U_3}) - (\Omega^T, \omega^T)
\end{equation}
on $V_1 \cap V_2$. In particular, \eqref{eq:cylindricalisedlessproper} is exponentially small on $L^T \cap U_3$. However, it is easy to see
\begin{equation}
\begin{aligned}
&\left.F^*((dt + id\theta) \wedge \Omega_\xs, dt\wedge d\theta + \omega_\xs)\right|_{\exp_v(K \times (T-4, T+4))} \\=& \left.((dt + id\theta) \wedge \Omega_\xs, dt\wedge d\theta + \omega_\xs)\right|_{(K \times (T-4, T+4))} = (\vol, 0).
\end{aligned}
\end{equation}
That is, $L^T \cap U_3$ is special Lagrangian with respect to $(\Omega_{U_3}, \omega_{U_3})$. That is, $(\Omega_{U_3}, \omega_{U_3})$ has the desired properties. This completes the proof of Condition \ref{slingcondition} (ii). 

$v$ is smooth since the special Lagrangian is smooth and it is a normal vector field between smooth submanifolds, as at the beginning of subsection \ref{ssec:speclagdeformation}. For the exponential decay of all the $C^{k, \mu}$ norms at a fixed rate, we note that how large $T$ needs to be in this argument depends on $k$ and $\mu$. However, for all $k$ and $\mu$ we have $\|v\|_{C^{k+1, \mu}}$ decaying exponentially in $T$ for $T$ large enough; moreover, with the same rate. Since $v$ is smooth, it follows that we can extend this to smaller $T$. 
\end{proof}
This is the gluing theorem for special Lagrangians. It follows from the proof of Proposition \ref{carefulift} that there is more than one possible special Lagrangian perturbation of $L^T$, and the family of perturbations corresponds to the harmonic normal vector fields on $L^T$. In particular, if we vary the harmonic normal vector field we can construct a whole local deformation space to our glued special Lagrangian; the estimate in Proposition \ref{carefulift} follows from choosing the normal vector field $v$ so that its corresponding one-form is $L^2$-orthogonal to harmonic forms. 
\begin{rmk}
We may compare this with the similar analysis in Joyce\cite{joycecones3} and Pacini\cite{pacini}. Both of these construct a submanifold that is close to special Lagrangian and then argue that it may be deformed. However, in both those cases, careful application of the Lagrangian neighbourhood theorem is used to ensure that the initial submanifold corresponding to $L^T$ is itself Lagrangian. If we defined $L^T$ with similar care we could presumably obtain that $\gamma_T(\omega_1, \omega_2)|_{L^T} = 0$, but as we have had to introduce a perturbation to $\omega$ to obtain a Calabi--Yau structure, we cannot obtain that $\omega^T|_{L^T} = 0$. 

If $L^T$ is Lagrangian, then we may again apply the Lagrangian neighbourhood theorem to infer that the one-forms corresponding to Lagrangian deformations of it under the specialised isomorphism of Condition \ref{slingcondition}(i) are closed. The assumption that they are orthoharmonic thus becomes that they are exact, and this rather simplifies the analysis. 
\end{rmk}

In section \ref{sec:opennessthm}, we will analyse a gluing map to show that it defines a local diffeomorphism on deformations of special Lagrangians. This means we need one single gluing map defined on nearly special Lagrangian submanifolds $L$, and we need to make a uniform choice of this harmonic normal field. Thus we make the following definition
\begin{defin}
\label{defin:slingmap}
The map $\SLing$ is defined from submanifolds of $M^T$ satisfying Condition \ref{slingcondition} to special Lagrangian submanifolds of $M^T$ by, given a submanifold $L$, finding a normal vector field $v$ to $L$ such that the corresponding one-form $\iota_v \omega^T|_L$ is $L^2$-orthogonal to harmonic forms as in Proposition \ref{carefulift} and letting $\SLing(L) = \exp_v(L)$. 
\end{defin}
Theorem \ref{slperturbthm} can then be interpreted as saying that the domain of $\SLing$ contains all approximate gluings of asymptotically cylindrical special Lagrangians for sufficiently large $T$.  
\subsection{Generalisations}
In this subsection, we explain that the domain of $\SLing$, or equivalently the set of submanifolds $L$ satisfying Condition \ref{slingcondition}, is larger than merely the patchings of section \ref{sec:constructingapproxspeclag}. We first observe that this is an open subset of submanifolds; then, in Theorem \ref{thm:finalslperturbthm} we prove Theorem A1: any ``nearly special Lagrangian" in the sense that $\omega|_L$ and $\Im \Omega|_L$ are exact and sufficiently small (unfortunately depending on $L$) can also be perturbed to be special Lagrangian. The work of this subsection is not required for the remainder of the paper, and so our treatment is somewhat brief. 

We begin with the openness result.
\begin{prop}
Condition \ref{slingcondition} is open in submanifolds with the $C^\infty$ topology, so that the domain of $\SLing$ is an open set. 
\end{prop}
This is essentially immediate: the four conditions all depend continuously on the structure, and so the set where they are satisfied is open. We now turn to a more direct generalisation. We note first of all that the proof of Theorem \ref{slperturbthm} gives the following statement.
\begin{prop}
\label{slperturbgeneralised}
Suppose $M$ is a $2n$-dimensional Calabi--Yau manifold with Calabi--Yau structure $(\Omega, \omega)$ and that $L$ is a closed submanifold. Let $k$ be a positive integer, and $A$ be a positive constant. Suppose that around each point $p$ of $L$ there is a local $SU(n)$ structure $(\Omega'_p, \omega'_p)$ around a neighbourhood of $p$ in $L$, in the sense of Definition \ref{defin:cystructurebundle}, with 
\begin{equation}
\|\Omega'_p - \Omega\|_{C^{k+2}}+  \|\omega'_p - \omega\|_{C^{k+2}} \leq A,
\end{equation}
and so that a neighbourhood of $p$ in $L$ is a special Lagrangian with respect to $(\Omega'_p, \omega'_p)$. 

If $A$ is sufficiently small, then for any $\mu \in (0, 1)$ Condition \ref{slingcondition} (ii)  holds for $k$ and $\mu$. We note that (iii) holds automatically with some constants $C_3$ and $r$. Hence, if  $\|\Im \Omega|_L\|_{C^{k+1}} + \|\omega|_L\|_{C^{k+1}}$ is also sufficiently small depending on $A$, the second fundamental form of $L$ in $M$, the inverse Laplacian bound, $r$, and $C_3$, and also $[\Im \Omega|_L] = 0 = [\omega|_L]$ then (iv) holds and $\SLing(L)$ exists.
\end{prop}
\begin{rmk}
When $k=0$ and restricting to the Lagrangian case, Joyce \cite[Proposition 5.8]{joycecones3} gave a rather more direct estimate of the constant corresponding to $C_3$ in (iii). Fundamentally, that argument shows that for any given $r$, estimating $C_3$ rests on estimating the derivatives of the nonlinear map $v \to (\exp_v^* \Im \Omega|_L, \exp_v^* \omega|_L)$; that is, by the chain rule, linearity of restriction to $L$, and Theorem \ref{localrestrictionthm}, it rests on the derivatives of $v \to (\exp_v^* \Im \Omega|_L, \exp_v^* \omega|_L)$ and the second fundamental form. Since $\Im \Omega$ and $\omega$ are parallel and of fixed size, estimating these derivatives depends entirely on estimating the pullback map $J^1(\nu_L) \to \bigwedge^n T^*M \oplus \bigwedge^2 T^*M|_L$, and its derivatives; $r$ controls exactly how large a ball in $J^1(\nu_L)$ we admit. Based on the Rauch comparison result (Proposition \ref{rauchcomparison}) compared with the identification of pushforward with Jacobi fields in Proposition \ref{basicgluingsmoothness}, it seems plausible that this should only depend on the curvature of $M$ and its derivatives. In any case, it should be possible to choose $r$ and then estimate these derivatives independently of $L$, by extending to work with the corresponding map $J^1(TM) \to \bigwedge^n T^*M \oplus \bigwedge^2 T^*M$. 

This estimate on the derivatives of pullback corresponds roughly to (iii) of \cite[Theorem 5.3]{joycecones3}, which is an estimate on certain adapted derivatives of $\Im \Omega$ considered as a form on $T^*L$, under the identification given by the Lagrangian neighbourhood theorem (that is, on the pullback of $\Im \Omega$ under an appropriate diffeomorphism). In \cite[Proposition 5.8]{joycecones3}, this yields the required derivatives, because the pushforward just reduces to an algebraic map under this identification. 
\end{rmk}

We shall now explain that the existence of $(\Omega'_p, \omega'_p)$ follows from smallness of $\Im \Omega|_L$ and $\omega|_L$. Combining this with Proposition \ref{slperturbgeneralised} leads to Theorem \ref{thm:finalslperturbthm} (Theorem A1).

We state the following result.
\begin{lem}
\label{newsunpointwise}
Let $V$ be a $2n$-dimensional vector space, and $L$ an $n$-dimensional subspace. For every $\epsilon>0$ there exists $\delta >0$  such that if $(\Omega, \omega)$ is an $SU(n)$ structure on $V$ and $|\Im\Omega|_L|+ | \omega|_L| + |\Re\Omega|_L - \vol_L| < \delta$ with respect to the metric induced by $(\Omega, \omega)$, then there is an $SU(n)$ structure $(\Omega', \omega')$ such that $L$ is special Lagrangian with respect to $(\Omega', \omega')$ and $|\Omega'-\Omega| + |\omega'-\omega| < \epsilon$.
\end{lem}
\begin{proof}[Sketch proof.]
A full proof of this is technically slightly involved, but conceptually straightforward. We may assume without loss of generality that $(\Omega, \omega)$ is the standard $SU(n)$ structure $\Omega = (e_1 + i e_2) \wedge \cdots \wedge (e_{2n-1} + i e_{2n}), \omega = e_1 \wedge e_2 + \cdots + e_{2n-1} \wedge e_{2n}$. Then setting 
\begin{equation}
\omega' = \omega - \sum_{i, j} \omega(e_i, e_j) \left(\xi_i \wedge \xi_j + J\xi_i \wedge J\xi_j\right),
\end{equation}
where the $\xi_i$ and $-J \xi_i$ form the dual basis to the $e_i$ and $Je_i$, yields another hermitian metric with $\omega'|_L = 0$. We then simply have to rescale $\Omega$ in absolute value so that it satisfies Definition \ref{defin:sunstructure}(iii) and adjust its phase so that $\Im \Omega'|_L = 0$. It is not hard to show that these only yield a small change in the structure.
\end{proof}
In fact, the construction in Lemma \ref{newsunpointwise} defines a bundle map, and so we can simply apply it globally on $L$ to obtain
\begin{prop}
\label{newsunconstruction}
Let $M$ be a $2n$-dimensional manifold, and $L$ an $n$-dimensional submanifold. For every $\epsilon>0$ there exists $\delta >0$  such that if $(\Omega, \omega)$ is a torsion-free $SU(n)$ structure on $M$ and $\|\Im\Omega|_L\|_{C^k}+ \| \omega|_L\|_{C^k} +\|\Re \Omega|_L - \vol_L\|_{C^k} < \delta$, then there is an $SU(n)$ structure $(\Omega', \omega')$ around $L$ such that $L$ is special Lagrangian with respect to $(\Omega', \omega')$ and $\|\Omega'-\Omega\|_{C^k} + \|\omega'-\omega\|_{C^k} < \epsilon$.
\end{prop}

Consequently, we find the following theorem saying that all ``nearly special Lagrangian" submanifolds, in some sense, can be perturbed to special Lagrangian. 
\begin{thm}[Theorem A1]
\label{thm:finalslperturbthm}
Suppose $M$ is a $2n$-dimensional Calabi--Yau manifold with Calabi--Yau structure $(\Omega, \omega)$ and that $L$ is a closed submanifold. Suppose that for some $k$, $\|\Im \Omega|_L\|_{C^{k+2}} + \|\omega|_L\|_{C^{k+2}}$ is sufficiently small, in terms of a constant depending on the second fundamental form of $L$ in $M$, a lower bound on $d+d^*$ on $L$ in the sense of Theorem \ref{laplacelowerbound}, and the constant $C_3$ of Condition \ref{slingcondition}(iii) (with appropriate regularities). Suppose further that $[\Omega|_L] =0$ and $[\omega|_L] = 0$. Then Condition \ref{slingcondition} holds (for any $\mu$) and we can thus perturb $L$ to a special Lagrangian submanifold. 
\end{thm}
\begin{rmk}
As mentioned after Proposition \ref{slperturbgeneralised}, a bound on $C_3$ could presumably be obtained by the argument of Joyce \cite[Proposition 5.8]{joycecones3}. Hence Theorem \ref{thm:finalslperturbthm} could be regarded as a rather less precise but somewhat extended version of Joyce's result \cite[Theorem 5.3]{joycecones3} which says that Lagrangian submanifolds that are sufficiently close to special Lagrangian can be perturbed to be special Lagrangian. 
\end{rmk}

\section{The derivative of \texorpdfstring{$\SLing$}{SLing}} 
\label{sec:dsling}
In the remainder of the paper, we shall show Theorem B: that the gluing map of special Lagrangians defined by combining the approximate gluing map of Definition \ref{defin:newsubmanifoldapproxgluing} with the $\SLing$ map of Definition \ref{defin:slingmap} is a local diffeomorphism from the space of matching special Lagrangian deformations of a matching pair to the space of special Lagrangian deformations of its gluing. 

To do this, we show that these maps are smooth and find their derivatives, and then apply the inverse function theorem. 

Smoothness is rather complicated. Although we know in the compact case that closed submanifolds of a given manifold $M$ form a Fr\'echet manifold (see Hamilton \cite[Example 4.1.7]{hamiltonnashmoser}), we will need to apply the inverse function theorem, and this does not hold for Fr\'echet manifolds without further assumptions. On the other hand, it's not at all straightforward to define a Banach manifold of submanifolds as the transition functions often involve composition. Consequently, we will restrict to finite-dimensional spaces of submanifolds. 

To prove that maps of forms and of normal vector fields (hence of submanifolds) are smooth, we will exhibit them as composites with smooth maps of bundles. For instance a map $F: TM \oplus TM \to TM$ such that $\pi(F(u, v)) = \exp(u)$ will, on restriction to a fixed family $u_s$, induce a smooth map of vector fields. We can then find their derivatives by looking at the derivatives of these bundle maps, and give bounds in this way. 

This analysis is simpler for $\SLing$, though there are other terms involved in that case making things more difficult. Thus, we will deal with that case first in this section before extending the relevant analysis to the approximate gluing map in section \ref{sec:dpatching}. Finally, in section \ref{sec:opennessthm}, we obtain the estimates on the derivatives which we will use to prove the result. 

We shall do much of this section and the next in the case where $M$ is Riemannian and $L$ is any submanifold, rather than restricting to the case where $M$ is Calabi--Yau and the submanifold $L$ is close to special Lagrangian. This is chiefly because the derivatives involved turn out, in terms of normal vector fields, to relate to Jacobi fields, and to work instead with one-forms simply adds additional complexity. 

In this section, we deal with closed manifolds and submanifolds, though our arguments will mostly be local and so will readily generalise. 

\subsection{General Riemannian manifolds}\enlargethispage{\baselineskip}
\label{ssec:riemanniannormalvectorfields}
We begin by working with identifications between various tubular neighbourhoods, which we explain with an example. Suppose that we have a curve of submanifolds $L_s$, and another curve of submanifolds $L'_s$ (in due course we will take $L'_s = \SLing(L_s)$). We suppose also that we know  about the normal vector field $w_s$ to $L_s$ such that $L'_s = \exp_{w_s}(L_s)$; for instance, if $L'_s = \SLing(L_s)$, this vector field is identified by the construction. Since we are interested in the derivative, there are further natural normal vector fields $u$ on $L_0$ giving the tangent of $L_s$ and $v$ on $L'_0$ giving the tangent of $L'_s$. Intuitively, it seems clear that ``$v = w' + u$". This is in fact the case, and we shall prove it carefully. The above explanation also shows why this looks like a general discussion of identifications between normal vector fields on different submanifolds, since we need to know precisely how to identify $u$, $v$, and $w'$ so that this makes sense; for the answer, see Proposition \ref{combinedtransfer}. 

We also will prove a basic smoothness result Proposition \ref{basictransfersmoothness}. This result is structurally the same as Proposition \ref{basicgluingsmoothness} that we needed in the previous section; it is used first to show that the maps are smooth as indicated above, and second as in Proposition \ref{speclaggluingremainderestimate} to obtain analytic controls on the maps. 

 We begin by defining some ways of identifying normal vector fields on $L_0$ with normal vector fields on $L'_0$. Suppose that $L$ is a submanifold of the Riemannian manifold $(M, g)$ and that $L' = \exp_v(L)$ is a small deformation of it.We note the following estimate.
\begin{prop}[corollary of ``Rauch Comparison Theorem", \eg Cheeger--Ebin {\cite[Theorems 1.28 and 1.29]{cheegerebin}}]
\label{rauchcomparison}
Let $(M, g)$ be a Riemannian manifold with sectional curvature bounded below and above by the constants $-c$ and $C$ respectively. Let $\gamma$ be a geodesic in $M$ with initial velocity $v_0$. Let $X$ be a Jacobi field along $\gamma$ with $|X_0| = A$ and $|\nabla_{\ddt}(X)_0|= B$. Then for $|v_0| < \frac{\pi}{2\sqrt{C}}$ we have estimates
\begin{equation}
A \cos\left(|v_0|\sqrt{C}\right) - \frac{B}{\sqrt{c}} \sinh\left(|v_0|\sqrt{c}\right)  \leq |X_1| \leq A \cosh\left( |v_0|\sqrt{c}\right) + \frac{B}{\sqrt{c}} \sinh\left(|v_0|\sqrt{c}\right).
\end{equation}
\end{prop}
Moreover, we have the following similar result, using a result in Jost \cite{jostrgga}.
\begin{prop}[corollary of {\cite[Theorem 4.5.3]{jostrgga}}]
\label{prop:paralleljacobiestimate}
Let $(M, g)$ be a Riemannian manifold with sectional curvature bounded below and above by the constants $-C$ and $C$, for some fixed $C>0$. Let $\gamma$ be a geodesic in $M$ with initial velocity $v_0$. Let $X$ be a Jacobi field along $\gamma$ with $|X_0| = A$ and $|(\nabla_\ddt X)_0| = B$. Let $P_t$ be parallel transport from $\gamma(0)$ to $\gamma(t)$ along $\gamma$. Then for each $t$,
\begin{equation}
|X_t - P_t(X_0 + t (\nabla_\ddt X)_0)| \leq A (\cosh (\sqrt {C} |v_0| t)-1) + B (\frac{1}{\sqrt {C} |v_0|} \sinh(\sqrt{C} |v_0| t) - t).
\end{equation}
\end{prop}
To give our definitions, we note that at a point we may represent $L$ by its tangent space. This is a point of the Grassmannian fibre bundle $\Gr_k(TM)$. To determine $L'$, we use the normal vector field $v$ such that $L' = \exp_v(L)$. The value of $v$ at a point of $L$ just determines a point of $L'$; we will need the tangent space to $L'$ also. We will use the first jet of $v$. This is comparable to the use of the first jet bundle for $u$ in Definition \ref{defin:remaindermapG}. 

As we will mostly be working with $n$-dimensional special Lagrangian submanifolds of $2n$-dimensional Calabi--Yau manifolds, we will restrict to the case of $n$-dimensional submanifolds. Consequently, we will define our transfer maps on the Whitney sum $TM \oplus J^1(TM) \oplus \Gr_n(TM)$. Note that the $J^1(TM)$ term will be the $1$-jet at a point of a small deformation and so can be expected to be small. 
\begin{defin}
\label{defin:pointwisetransfers}
Let $M$ be a $2n$-dimensional Riemannian manifold. Suppose that 
\begin{equation}
(u, (v, \nabla v), \ell) \in TM \oplus J^1(TM) \oplus \Gr_{n}(TM),
\end{equation}
with $\pi(u, (v, \nabla v), \ell) = p$ and $(v, \nabla v)$ small. We note that $v$, being small, defines a geodesic $\gamma$ in $M$. 

Choose a basis $\{e_1, \ldots, e_n\}$ for the subspace $\ell$ of $T_pM$. For each $e_i$ construct the Jacobi field $X^{(i)}$ along $\gamma$ with $X^{(i)}_0 = e_i$ and $(\nabla_\ddt X^{(i)})_0 = \nabla_{e_i} v$. The final values of the $X^{(i)}$ form a subset of $T_{\exp_p(v)}M$. If this subset is not linearly independent, then it follows by linearity of the Jacobi equation that we could choose $e_1$ so that the final value of $X^{(1)}$ is zero; but for $v$ and $\nabla v$ sufficiently small this is impossible by Proposition \ref{rauchcomparison}. Hence, this set is linearly independent and forms a basis for an $n$-dimensional subspace $\ell'$ of $T_{\exp_p(v)}M$. 

Now find Jacobi fields $X_1$ and $X_2$ along $\gamma$ with initial conditions $(X_1)_0 = 0 = (\nabla_\ddt X_2)_0$ and $(\nabla_\ddt X_1)_0 = u = (X_2)_0$ and suppose that $X_i$ has final value $w_i \in T_{\exp_p(v)}M$. We then take $T_i(u, (v, \nabla v), \ell)$ to be the part of $w_i$ orthogonal to the subspace $\ell'$. 
\end{defin}
$T_i$ is only defined on an open subset of $TM \oplus J^1(TM) \oplus \Gr_{n}(TM)$, as the geodesic $\gamma$ must exist. This cannot necessarily be stated as ``$|(v, \nabla v)| < \epsilon$" for some uniform $\epsilon$, because $M$ might be incomplete. In this case, we shall sometimes extend to $M'$ into which we can embed $M$ isometrically, as for $G$ in Definition \ref{defin:remaindermapG}. 

We now consider what the $T_i$ look like for fixed submanifolds. 
\begin{defin}
\label{defin:sectiontransfers}
Let $M$ be a Riemannian manifold, and let $L$ and $L'$ be submanifolds of $M$. Suppose that there is a normal vector field $v$ on $L$ such that $\exp_{v}(L) = L'$. Suppose that $\|v\|_{C^1}$ is sufficiently small.

Define two bundle maps $\nu_L \to TM|_{L'}$ by writing $T_i(u) = T_i(u, (v, \nabla v), T_pL)$, where $u \in (\nu_L)_p$, $(v, \nabla)$ is some extension of the first jet of $v$ at $p$ to $J^1(TM)|_L$, and $T_i$ is as in Definition \ref{defin:pointwisetransfers}, for $i=1, 2$. 
\end{defin}
It is easy to see by examining Definition \ref{defin:pointwisetransfers} that these bundle maps are independent of the extension of the jet required. It is also easy to see using the identification of pushforward in the proof of Proposition \ref{basicgluingsmoothness} that the projection in Definition \ref{defin:pointwisetransfers} is just the orthogonal projection onto the normal bundle of $L'$, so that $T_1$ and $T_2$ define bundle maps from $\nu_L$ to $\nu_{L'}$. 

We now show that these bundle maps are bundle isomorphisms.   
\begin{prop}
\label{transferisomorphism}
The bundle maps $T_1$ and $T_2$ of Definition \ref{defin:sectiontransfers} define bundle isomorphisms $\nu_L \to \nu_{L'}$ if $\|v\|_{C^1}$ is sufficiently small. 
\end{prop}
\begin{proof}
By definition $T_i: \nu_L \to \nu_{L'}$ are bundle maps over $\exp_v:L \to L'$. Both $\nu_L$ and $\nu_{L'}$ have rank $n$, so it suffices to prove that there is no normal vector $u \in (\nu_L)_p$ such that $T_i(u) = 0$.

Let $\gamma$ be the geodesic with initial velocity $v_p$. We must show that if we have $e \in T_pL$ and $u \in (\nu_L)_p$, and Jacobi fields $X^0$ and $X^1$ along $\gamma$ with $(X^0)_0 = e $ and $(\nabla_\ddt X^0)_0 = \nabla_e v$, and  either $(X^1)_0 = u $ and $(\nabla_\ddt X^1)_0 = 0$, or vice versa, then the final values of $X^0$ and $X^1$ differ. 

To do this, we will apply Proposition \ref{prop:paralleljacobiestimate}. Let $P_t$ be the parallel transport along $\gamma$ and then let 
\begin{equation}
Y^0 = P_1( e + \nabla_e v), \qquad Y^1 = P_1(u)
\end{equation}
By Proposition \ref{prop:paralleljacobiestimate}, we have
\begin{align}
|X^0_1 - Y^0| &\leq |e| \left( \cosh (\sqrt{C} |v_p|) - 1 + |(\nabla v)_p|\left( \frac{\sinh (\sqrt{C} |v_p|)}{\sqrt{C} |v_p|}  - 1\right) \right),\\
|X^1_1 - Y^1| &\leq |u| \max \{\cosh(\sqrt{C} |v_p|) - 1, \frac{\sinh (\sqrt{C} |v_p|)}{\sqrt{C} |v_p|}  - 1\}.
\end{align}
This tells us that by choosing $\|v\|_{C^1}$ small enough, we can make $X^i_1$ and $Y^i$ as close as we like, in terms of $|e|$ and $|u|$. It's easy to see that the inner product of $Y^0$ and $Y^1$ is small in terms of $|e|$ and $|u|$, since $e$ and $u$ are orthogonal, and also that $Y^0$ and $Y^1$ are similar in size to $|e|$ and $|u|$. It follows that for $\|v\|_{C^1}$ small enough, $|\la X^0_1, X^1_1\ra| <  |X^0_1| |X^1_1|$, and thus $X^0_1$ and $X^1_1$ cannot be equal for $\|v\|_{C^1}$. 
\end{proof}
Hence, the $T_i$ define isomorphisms from normal vector fields on $L$ to normal vector fields on $L'$. Note also that the construction of $\ell'$ in Definition \ref{defin:pointwisetransfers} precisely gives a further isomorphism
\begin{equation}
\label{eq:pointwisepushforward}
T_pM \cong T_{\exp_p(v)} M,
\end{equation}
with $\ell'$ being the image of $\ell$. Passing to a map of tangent vectors as in Definition \ref{defin:sectiontransfers}, this isomorphism corresponds to the pushforward of tangent vectors by the diffeomorphism $\exp_v$ from $L$ to $L'$. We have shown as part of Proposition \ref{basicgluingsmoothness} that it is smooth and depends continuously on the metric as a map of triples $(u, (v, \nabla v), \ell)$; in fact, of course, this map is essentially independent of the space $\ell$.  

We now pass to a finite-dimensional family $\U$ of submanifolds $L_s$. We assume we choose the parametrisation diffeomorphically, so that smooth dependence on $s$ and $L_s$ is equivalent. We will show that the transfer maps $T_i$ define smooth maps from normal vector fields on $L_0$ to normal vector fields on $L_s$ in $\U$. As with the definition, we begin with a pointwise result. 
\begin{prop}
\label{basictransfersmoothness}
Let $M$ be a Riemannian manifold. Consider the maps 
\begin{equation}
T_i: TM \oplus J^1(TM) \oplus \Gr_n(TM) \to TM
\end{equation}
defined as in Definition \ref{defin:pointwisetransfers}. These maps are smooth. They depend continuously on the metric on $M$ in the sense described in Proposition \ref{basicgluingsmoothness}, and \eqref{eq:newcomplicatedsctyest} becomes
\begin{equation}
\label{eq:oldcomplicatedsctyest}
|T_i(s, u, (v, \nabla v), \ell) - T_i (s', u, (v, \nabla v), \ell)| < \epsilon|(u, (v, \nabla v))|.
\end{equation}

Moreover, if we extend these maps by the identity to 
\begin{equation}
\label{eq:Tiextended}
T_i: TM \oplus J^1(TM) \oplus \Gr_n(TM) \to TM \oplus J^1(TM) \oplus \Gr_n(TM),
\end{equation}
they are injective immersions. That is, their images are submanifolds and the inverse maps defined on these images are smooth, and depend continuously on the metric in the same sense. 
\end{prop}
We omit the proof, which is essentially identical to that of Proposition \ref{basicgluingsmoothness}: note that since the fibres of the Grassmannian are compact, we do not need to estimate the size of $\ell$ in \eqref{eq:oldcomplicatedsctyest}. 

We may now prove
\begin{prop}
\label{transfersmoothness}
Let $L_s$ and $L'_{s'}$ be finite-dimensional smooth compact families of submanifolds of $M$, with $L_0$ and $L'_0$ close enough that there exists a normal vector field $v(0, 0)$ to $L_0$ whose image under the Riemannian exponential map is $L'_0$.

Then for $s$ and $s'$ small enough and $v(0, 0)$ small enough, there exists a normal vector field to $L_s$ whose image under the Riemannian exponential map is $L'_{s'}$. This normal vector field is the image under $T_2$ of a normal vector field $v(s, s')$ on $L_0$. The map $(s, s') \mapsto v(s, s')$ is a smooth map.
\end{prop}
\begin{proof}
The existence of the normal vector fields is immediate, so we have to prove that $v(s, s')$ depends smoothly on $(s, s')$. We know that $\exp$ is a smooth map $TM \to M$. Write $u_{s}$ for the normal vector field to $L_0$ whose image is $L_s$, and consider the map
\begin{equation}
\label{eq:generalisedexp}
\begin{tikzcd}[row sep=0pt]
\nu_L \times \U \ar{r}& M\times \U, \\[0pt]
(v, s) \ar[mapsto]{r}& (\exp(T_2(v, j^1(u_s)_{\pi(v)}, T_{\pi(v)}L_0)), s).
\end{tikzcd}
\end{equation}
where $j^1$ is the jet prolongation map of Proposition \ref{newjetbundles}, and we choose some smooth extension of this to $J^1(TM)|_L$ as before. Using Proposition \ref{basictransfersmoothness}, we have that \eqref{eq:generalisedexp} is smooth. It is also easy to see that its derivative at $(0_p, s)$ is an isomorphism if and only if the component $T_{0_p} \nu_L \to T_{\exp_p(u_s)} M$ is an isomorphism.

We show that this is the case if $s$ is sufficiently small, by finding curves for each tangent vector and considering their image under \eqref{eq:generalisedexp}. $T_{0_p} \nu_L$ can be decomposed as the direct sum $T_{0_p}((\nu_L)_p) \oplus T_pL$; clearly, $T_{0_p}((\nu_L)_p) \cong (\nu_L)_p$. For $v' \in T_pL$, we can choose some curve $\gamma$ in $L$ through $p$ with tangent $v'$ and then consider the corresponding curve $0_\gamma$ in $\nu_L$. The  image of this curve under \eqref{eq:generalisedexp} is precisely $\exp_{u_s} \circ \gamma$, so for $u_s$ sufficiently small, $T_pL \subset T_{0_p} \nu_L$ is mapped isomorphically to $T_{\exp_p(u_s)}L \subset T_{\exp_p(u_s)}M$. On the other hand, for $v' \in T_{0_p}((\nu_L)_p)$, we choose the curve $sv'$. The image of this curve is $\exp(T_2(sv'))$; since the derivative of $\exp$ at zero is the identity, the tangent at zero is just $T_2 (v')$ and hence $T_{0_p}((\nu_L)_p) \subset T_{0_p} \nu_L$ is mapped isomorphically to $(\nu_{L_s})_{\exp_p(u_s)} \subset T_{\exp_p(u_s)}M$, by Proposition \ref{transferisomorphism}. 

Hence, for $s$ sufficiently small, the derivative $T_{0_p} \nu_L \to T_{\exp_p(u_s)} M$ is an isomorphism, and so \eqref{eq:generalisedexp} is a local diffeomorphism. Since $L_0$ is compact, we can find a tubular neighbourhood $T$ of $L_0$ on which the inverse map
\begin{equation}
\label{eq:generalisedexpinv}
T \times \U \to \nu_{L_0} \times \U
\end{equation}
is well-defined. If $v(0, 0)$ is sufficiently small, then $L'_0$ lies in $T$, and hence so does $L'_{s'}$ for $s'$ sufficiently small. Clearly, $L'_{s'}$ can be essentially equivalently viewed as a smooth family of inclusions of $L$ into $M$. The map desired is given by the composition of these inclusions with \eqref{eq:generalisedexpinv}; as a composition with a smooth map, this is smooth. 
\end{proof}
\begin{rmk}
The same holds if $L_s$ and $L'_{s'}$ are asymptotically cylindrical submanifolds with decay rate uniformly bounded from below. The argument to construct a tubular neighbourhood on which \eqref{eq:generalisedexpinv} is smooth goes through just as before, noting that the metric converges to a limit, so \eqref{eq:generalisedexp} does, essentially by the same argument as in Proposition \ref{basictransfersmoothness}. We have to check that composition behaves in this case: but this works exactly as in the compact case.
\end{rmk}
We now explain how the maps $T_1$ and $T_2$ yield the identifications required at the beginning of this subsection. We begin with $T_1$. 
\begin{prop}
\label{basetransferfundamental}
Let $M$, $L$, $L'$, $v$ and $T_1$ be as in Definition \ref{defin:sectiontransfers}. Let $v_s$ be a smooth curve of normal vector fields on $L$ with $v_0 = v$.  Then $\exp_{v_s}(L)$ defines a smooth curve of submanifolds, and when $s=0$ it passes through $\exp_v(L) = L'$. Therefore, there exists a smooth curve $w_s$ of normal vector fields to $L'$ such that
\begin{equation}
\exp_{v_s}(L) = \exp_{w_s}(L').
\end{equation}
We have $w' = T_1 v'$ where $w'$ and $v'$ are the derivatives at $0$ of these curves of vector fields.  
\end{prop}
\begin{proof}
To see $w_s$ is a smooth curve, we apply Proposition \ref{transfersmoothness} with $L_s = L'$ for all $s$, and $L'_{s'} = \exp_{v_s}(L)$. Then $w_s$ is precisely the normal vector field to $L'$ given by Proposition \ref{transfersmoothness} and therefore depends smoothly on $s$. We now need to show that the derivative of this curve at zero is given by $T_1(v')$. 

We begin by noting that we can find sections $x_s$ of $TM|_{L'}$, such that for all $p \in L_0$ we have $\exp_{v_s}(p) = \exp_{x_s}\exp_{v_0}(p)$. We can evaluate $x'$, as follows. By construction, and the fact the derivative of $\exp$ at $0$ is the identity, $x'_{\exp_{v_0}(p)}$ is the derivative of the curve $\exp_p(v_s)$ at $s=0$. For each $s$, $\exp_p(v_s)$ is given by the final position of a geodesic; since $v_s(p)$ is smooth in $s$, we have a smooth variation through geodesics. $x'$ is then the final value of the corresponding Jacobi field $X$ along the geodesic $\gamma$ corresponding to $v_0(p)$. Since all the geodesics start at the same point $p$, we certainly have $X_0 = 0$. For $(\nabla_{\dot\gamma} X)_0$, we note that this as usual is equal to the derivative in $s$ of the initial velocities, hence $v'$. That is, $x'$ is given by the final value of the Jacobi field along $\gamma$ with initial conditions $X_0 = 0$, $(\nabla_{\dot\gamma} X)_0 = v'$. Note that this is just evaluating the derivative of the Riemannian exponential map: for an alternative proof, see \cite[Corollary 3.46]{ghl}.

We now have to move from $x'$ to $w'$. To determine $w_s$ from $x_s$, we compose $\exp_{x_s}$ with the projection $\pi$ of a tubular neighbourhood map around $L'$, invert the resulting diffeomorphism of $L'$, and then compose this inverted diffeomorphism with $\exp_{x_s}$. To evaluate the derivative of this is fairly straightforward. The derivative of the curve $\pi \circ \exp_{x_s}$ of diffeomorphisms of $L$ is just the tangential part of $x'$ (see \cite[Example 4.4.5]{hamiltonnashmoser}). On inversion around the identity $\pi \circ \exp_{x_0}$, we get the negative of this tangential part. Recomposing this with $\exp_{u_s}$ gives addition of the derivatives, and so 
\begin{equation}
w' = \npt x' = T_1 v'.\qedhere
\end{equation}
\end{proof}
$T_2$ is not quite the most obvious other way of constructing a curve of submanifolds through $L'$. We introduce the following notation.
\begin{defin}
\label{defin:dernpt}
Let $M$ be a Riemannian manifold and let $L$ be a submanifold. Let $u$ be a normal vector field on $L$, so that $\exp_{su}(L)$ forms a curve of submanifolds for $u$ sufficiently small. Let $v$ be a normal vector to $L$, with $\pi(v) = p \in L$. Let $\gamma$ be the geodesic with initial velocity $u_p$, so that $\gamma(s) \in \exp_{su}(L)$ for each $s$. Let $J$ be the Jacobi field along $\gamma$ with initial conditions $J_0 = v$ and $(\nabla_\dds J)_0 = 0$. Then $J_s \in TM|_{\exp_{su}(L)}$ for each $s$, so we can consider its normal part. This gives another curve of tangent vectors $N_s$ with $\pi(N_s) = \gamma(s)$. We may consider the derivative $N'$ of $N_s$ at $s=0$. We define
\begin{equation}
\npt'_u v = N'.
\end{equation}
\end{defin}
Concretely, if $\{e_i(s)\}$ is a smooth family of orthonormal bases for $T_{\gamma(S)}\exp_{su}(L)$, 
\begin{equation}
\npt'_u v = -\sum g'(e_i, v) e_i - g(e'_i, v) e_i.
\end{equation}
We will not use this form explicitly, but will rely on the fact that it is small if $v$ is. Now $\npt'_u v$ is not necessarily a normal vector on $L$, but we can extend the map $T_1$ of Definition \ref{defin:sectiontransfers} to sections of $TM|_L$. We can thus make
\begin{defin}
\label{defin:t3}
Let $M$ be a Riemannian manifold and let $L$ be a submanifold. Let $u$ and $v$ be normal vector fields on $L$, so that $\npt'_u v$ is a section of $TM|_L$. Then let $T_3u$ be $T_1(\npt'_u v)$ where we extend the map $T_1$ of Definition \ref{defin:sectiontransfers} constructing a normal vector field on $\exp_v(L)$. 
\end{defin}
Note that $T_1 u$ and $T_2 u$ are linear maps of $u$, but $T_3 u$ is not. Moreover, we can write $T_3 u$ also as a map induced from a map on a bundle as in Definition \ref{defin:pointwisetransfers}, but we would have to define it on $J^1(TM) \oplus J^1(TM) \oplus \Gr_n(TM)$, as the condition in Definition \ref{defin:dernpt} that we take the normal part of $J_s$ on $\exp_{su}(L)$ requires the first jet of $u$ to determine the tangent space $T_{\exp_{su}(p)}\exp_{su}(L)$. We will not give details of these arguments. 

We can now explain the relevance of $T_2$. 
\begin{prop}
\label{fieldtransferfundamental}
Let $M$, $L$, $L'$, $v$ and $T_2$ be as in Definition \ref{defin:sectiontransfers}. Let $u_s$ be a smooth curve of normal vector fields on $L$ with $u_0 = 0$ and derivative $u'$; let $T_3$ be as in Definition \ref{defin:t3}. Let $T_{2, s}$ be the transfer maps from $L$ to $\exp_{u_s}(L)$ given by taking the appropriate map $T_2$ from Definition \ref{defin:sectiontransfers}. Then $\exp_{T_{2, s}(v)}\exp_{u_s}(L)$ defines a smooth curve of submanifolds, and when $s=0$ it passes through $\exp_v(L) = L'$. Therefore, there exists a smooth curve $w_s$ of normal vector fields to $L'$ such that
\begin{equation}
\exp_{T_{2, s}(v)} \exp_{u_s} (L) = \exp_{w_s}(L').
\end{equation}
We have $w' = T_2 u' + T_3u'$. 
\end{prop}
\begin{proof}
To see that $w_s$ is a smooth curve, we show that $\exp_{T_{2, s}(v)} \exp_{u_s} (L)$ is a smooth curve of submanifolds; then we may apply the argument at the beginning of Proposition \ref{basetransferfundamental}. We have to show that if $\iota$ is the inclusion of $L$ then $\exp_{T_{2, s}(v)} \exp_{u_s} \iota$ is a smooth curve of maps. But this is the image under the smooth map
\begin{equation}
\label{eq:fieldtransfersmoothness}
\begin{tikzcd}[row sep=0pt]
J^1(\nu_L) \ar{r} & M,\\
(u, \nabla u) \ar[mapsto]{r} & \exp(T_2(v_{\pi(u)}, u,\nabla u, T_{\pi(u)}L)),
\end{tikzcd}
\end{equation}
where $T_2$ is as in Definition \ref{defin:pointwisetransfers} of the smooth curve $j^1(u_s)$ of first jets, hence is indeed smooth. 

To find $w'$, as in Proposition \ref{basetransferfundamental}, we first construct for each $p \in L$ a family of tangent vectors $x_s$ to $\exp_p(v)$ such that
\begin{equation}
\label{eq:generalisedexptwo}
\exp_{T_{2, s}(v)}\exp_{u_s}(p) = \exp_{x_s} \exp_v(p).
\end{equation}

Again as in Proposition \ref{basetransferfundamental}, we identify $\exp_{T_{2, s}(v)} \exp_{u_s}(p)$ as the curve of final positions of a variation through geodesics, so that $x'(p)$ is the final value of the Jacobi field $X$ corresponding to this variation. These geodesics have initial positions $\exp_{u_s}(p)$ and initial velocities $T_{2, s}(v_p)$. Differentiating in $s$, we see that $X_0 = u'_p$. 

Now also $x_s$ at $\exp_p(v_p)$ is given by composing $\exp^{-1}_{\exp_p(v)}$ with the left hand side of \eqref{eq:generalisedexptwo}. Hence, it depends smoothly on the curve of $1$-jets  of $u_s$ at $p$. In particular, $x'$ at $\exp_p(v_p)$ depends smoothly on the $1$-jets of $u_0 = 0$ and of $u'$. These $1$-jets remain the same if we replace $(u_s)$ by $(su')$, and so to find $x'$ we may suppose that $u_s = su'$. 

With this curve, we are in the situation of Definition \ref{defin:dernpt}. We let $\gamma(s)$ be the geodesic with initial velocity $u'_p$, and let $J_s$ be the Jacobi field along $\gamma$ with initial conditions $J_0 = v_p$ and $(\nabla_\dds J)_0 = 0$. Then $T_{2, s}(v_p)$ is just the normal part of $J_s$. Hence, the derivative in $s$ of $T_{2, s}(v_p)$ is just $\npt'_{u'}v$. This is equivalently $(\nabla_\ddt X)_0$. 

That $w'$ is given by taking the normal part of $x'$ follows exactly as in Proposition \ref{basetransferfundamental}. Hence $w'$ is given by taking the normal part of the final value of Jacobi fields, and so is a linear combination of $T_1$ and $T_2$: examining the initial conditions precisely gives $w' = T_2 u' + T_3u'$. 
\end{proof}
Putting the ideas of Propositions \ref{basetransferfundamental} and \ref{fieldtransferfundamental} together, we obtain
\begin{prop}
\label{combinedtransfer}
Let $M$, $L$, $L'$ and $v$ be as in Definition \ref{defin:sectiontransfers}. Suppose that $u_s$ is a curve of normal vector fields to $L_0$ and $w_s$ is a curve of normal vector fields to $L'$. For each $s$, we can find a normal vector field $v_s$ to $\exp_{u_s}(L_0)$ so that 
\begin{equation}
\exp_{v_s}\exp_{u_s} (L_0) = \exp_{w_s}\exp_{v_0} (L_0).
\end{equation}
The transfer map $T_2$ of Definition \ref{defin:sectiontransfers} defines an isomorphism between the normal bundles to $L_0$ and $\exp_{u_s}(L_0)$. Hence, we may identify $v_s$ with a normal vector field on $L$, and so $(v_s)$ is a smooth curve of such normal vector fields. We have 
\begin{equation}
\label{eq:bigtransferimplicit}
w' =  T_1 v' + T_2 u' + T_3u',
\end{equation}
and so the derivative of the map $(u_s, w_s) \mapsto v_s$ is given by
\begin{equation}
\label{eq:bigtransferexplicit}
v' = T_1^{-1}(T_2 u' + T_3u' - w' ).
\end{equation}
\end{prop}
\subsection{Nearly special Lagrangian submanifolds}
\label{ssec:normalvectorfieldsspeclag}
We now pass to the case where our submanifolds are close to special Lagrangian. By Lemma \ref{Jbehaves}, on such submanifolds normal vector fields can be identified with one-forms. This means that we can write the maps $T_1$ and $T_2$ of Definition \ref{defin:pointwisetransfers} and hence of Definition \ref{defin:sectiontransfers}, and the map $T_3$ of Definition \ref{defin:t3}, in terms of one-forms: we carry this out carefully. 

To do this, we first need to introduce the requirement that the submanifold be ``nearly special Lagrangian" to the bundle $TM \oplus J^1(TM) \oplus \Gr_n(TM)$ used in Definition \ref{defin:pointwisetransfers} and to the corresponding ``target bundle" in which the image of the extended map \eqref{eq:Tiextended} lies. We also have to restrict to bundles containing normal vectors to deal with inverses, and make a corresponding definition for one-forms. We thus make the following
\begin{defin}
\label{defin:newbigsubbundles}
Let $M$ be a Calabi--Yau manifold. Let $\cO$ be the subbundle of $\Gr_{n}(TM)$ consisting of those subspaces $\ell$ for which $|\omega|_\ell|<1$. Let $\cO'$ be the subbundle of $J^1(TM) \oplus \cO$ such that also the subspace $\ell'$ constructed in Definition \ref{defin:pointwisetransfers} satisfies $|\omega|_{\ell'}|<1$. 

Let $N$ be the subbundle of $TM \oplus J^1(TM) \oplus \Gr_n(TM)$ (over $M$) consisting of those $(u, (v, \nabla v), \ell)$ such that $u$ is normal to the subspace $\ell$. Let $N'$ be the subspace
\begin{equation}
\begin{split}
\{(u, (v, \nabla v), \ell) &\in TM \times (J^1(TM) \oplus \Gr_n(TM)): \\& \pi_{TM}(u) = \exp_{\pi_{J^1(TM) \oplus \Gr_n(TM)}((v, \nabla v), \ell)} (v) \text{ and } u \text{ is normal to $\ell'$}\},
\end{split}
\end{equation}
where $\ell'$ is the subspace constructed in Definition \ref{defin:pointwisetransfers}. 

Let $\F$ be the subbundle of $T^*M \oplus J^1(TM) \oplus \Gr_n(TM)$ consisting of $(\alpha, (v, \nabla v), \ell)$ with $\alpha^\sharp$ in $\ell$. 
\end{defin}
$\cO$ and $\cO'$ are open subbundles. It is easy to see, similarly to Lemma \ref{gluingpointwisespacesubmanifold} that $N$, $N'$ and $\F$ are smooth manifolds and the projection maps define bundle structures. The intersection of $N$, $N'$ and $\F$ with $\cO'$ in their $\Gr_n(TM)$ component are then also open subbundles. Note that $N$ and $N'$ may be defined with merely a Riemannian metric, and this extension will be used in Proposition \ref{revisedtransferofatis} below. 

We may make
\begin{defin}
\label{defin:iandiprime}
Let $I: N \cap \cO' \to \F \cap \cO'$ be given by taking $(u, (v, \nabla v), \ell)$, contracting $u$ with $\omega$, and then taking the orthogonal projection to the tangential part. 

Let $I': N' \cap \cO' \to \F \cap \cO'$ be given similarly by contracting $u$ with $\omega$ at $\exp_v(p)$, pulling back the resulting covector using $(v, \nabla v)$ by applying the dual of the map \eqref{eq:pointwisepushforward}, and the taking the orthogonal projection to the tangential part. 
\end{defin}
We have, reasonably simply, and using the argument in Proposition \ref{basicgluingsmoothness} for pullback
\begin{prop}
\label{pullbackisomorphismssmooth}
$I$ and $I'$ are smooth maps, depending continuously on the metric as in Proposition \ref{basictransfersmoothness}, and locally around any $(u, 0, 0, \ell)$ are diffeomorphisms. 
\end{prop}

This enables us to make our definitions of what $T_1$ and $T_2$ look like in terms of one-forms. 
\begin{defin}
\label{defin:taui}
Let $M$ be a Calabi--Yau manifold. 

First consider $\F \cap \cO'$. Define maps $\tau_1$ and $\tau_2$ on this space by $\tau_i = I' \circ T_i \circ I^{-1}$, with $T_i$ the extended map from Proposition \ref{basictransfersmoothness}. These maps are also local diffeomorphisms around any $(\alpha, 0, 0, \ell)$. Similarly define a map $\tau_3$ on a suitable bundle by composition of the pointwise version of the map $T_3$ of Definition \ref{defin:t3} with suitably extended versions of $I'$ and $I$. 

Now let $L$ be a submanifold of $M$ such that $\|\omega|_L\|_{C^0} < 1$. Given a sufficiently small normal vector field $v$ on $L$, we can construct $L' = \exp_v(L)$ and again have $\|\omega|_{L'}\|_{C^0} < 1$. For $i=1, 2$, we can then define a map $\tau_i$ from one-forms on $L$ to one-forms on $L'$ by taking $\alpha$ to $\tau_i(\alpha, (v, \nabla v), T_pL)$ where $\nabla v$ is extended somehow as in Definition \ref{defin:sectiontransfers}. Similarly, we can define a map $\tau_3$ by taking some extension of $\alpha$; the map will not depend on the extension. 
\end{defin}
Note from the discussion after Definition \ref{defin:t3} that we need the first jet of $\alpha$ to define our pointwise $\tau_3$. This is thus more complicated, and throughout this discussion we will omit the details. 

It is then immediate from the definitions that these maps of one-forms correspond to the maps $T_i$ of normal vector fields, since $I$ and $I'$ are the pointwise versions of $v \mapsto \iota_v \omega|_L$ and $v \mapsto \iota_v \omega|_{L'}$. Specifically, we have 
\begin{prop}
\label{tauiareti}
Let $M$, $L$ and $L'$ be as in Definition \ref{defin:taui}, and $i \in \{1, 2\}$. Then the following diagram commutes:
\begin{equation}
\begin{tikzcd}[column sep=large]
\text{normal vector fields on $L$} \ar[leftrightarrow]{r}{v \mapsto \iota_v \omega|_L}  \ar[leftrightarrow]{d}{T_i \text{ (Def. \ref{defin:sectiontransfers})}} & \text{one-forms on $L$} \ar[leftrightarrow]{d}{\tau_i \text{ (Def. \ref{defin:taui})}} \\
\text{normal vector fields on $L'$} \ar[leftrightarrow]{r}{v \mapsto \iota_v \omega|_{L'}} & \text{one-forms on $L = L'$}.
\end{tikzcd}
\end{equation}
where the horizontal maps are determined by Lemma \ref{Jbehaves}. For example, given a one-form $\alpha$ on $L$, we have the one-form $\tau_i\alpha$ on $L$. We find another one-form on $L$ by taking a normal vector field $u$ on $L$ such that $\iota_u \omega|_L = \alpha$, applying $T_i$ to $u$ to get a normal vector field on $L'$, and then finding the corresponding one-form $\iota_{T_i u} \omega|_{L'}$, and these two one-forms are equal. Similarly, given a one-form on $L=L'$, the one-forms on $L$ constructed by $\tau_i^{-1}$ and $T_i^{-1}$ are the same. Finally, the corresponding diagram with $T_3$ and $\tau_3$ also commutes. 
\end{prop}
Furthermore, we have the following regularity result for the $\tau_i$ of Definition \ref{defin:taui} corresponding to Proposition \ref{basictransfersmoothness}, simply following from that and Proposition \ref{pullbackisomorphismssmooth}. 
\begin{prop}
\label{tauismoothness}
For $i=1, 2, 3$, $\tau_i$ are smooth maps, and depend continuously on the Calabi--Yau structure in the same sense as $T_i$ depends continuously on the metric (described in Proposition \ref{basictransfersmoothness}); in particular, we have the estimate corresponding to \eqref{eq:oldcomplicatedsctyest}. 
\end{prop}

\subsection{The Laplacian on normal vector fields}
\label{ssec:laplacianonnvfs}
Finally, $\SLing$ relies on the identification of the harmonic part of a given normal vector field (that is, the normal vector field corresponding to the harmonic part of the corresponding one-form). We thus need to check this is also smooth and find its derivative. In this subsection, we use Lemma \ref{Jbehaves} similarly to subsection \ref{ssec:normalvectorfieldsspeclag} to define a Laplacian on the normal vector fields on a submanifold that is close to special Lagrangian. This consequently defines a notion of the harmonic part of a normal vector field. We then show that, with an appropriate identification of normal vector fields taken from section \ref{sec:dsling}, this harmonic part depends smoothly on the submanifold concerned and in Proposition \ref{harmptderivative} we identify the derivative of the map $v \mapsto \hpt v$, in terms of the derivative of the Laplacian induced by the identifications. 

We make
\begin{defin}
\label{defin:laplacianonnvfs}
Suppose that $L \subset M$ is a submanifold and $(\Omega, \omega)$ is an $SU(n)$ structure around it in the sense of Definition \ref{defin:cystructurebundle}. This induces a Riemannian metric on $L$ and consequently a Laplace--Beltrami operator on the differential forms on $L$. Suppose that Lemma \ref{Jbehaves} holds for $(\Omega, \omega)$, that is 
\begin{equation}
\label{eq:slingcondiiso}
v \mapsto \iota_v \omega|_L
\end{equation}
is an isomorphism between normal vectors and one-forms. Then \eqref{eq:slingcondiiso} induces linear differential operators $\Delta$ and $d+d^*$ on normal vector fields. 
\end{defin}
Since by the estimates of \eqref{eq:Jbehavesbounds} we have an isomorphism between differential one-forms and normal vector fields of given regularity, all the properties of $d+d^*$ and $\Delta$ carry over. Hence we have \emph{harmonic} normal vector fields, whose corresponding one-forms are harmonic, and \emph{orthoharmonic} normal vector fields, whose corresponding one-forms are $L^2$-orthogonal to harmonic forms (or which lie in the image of $\Delta$). Note that as \eqref{eq:slingcondiiso} need not be an isometry if $L$ is not Lagrangian, harmonic normal vector fields and orthoharmonic normal vector fields need not be $L^2$-orthogonal. 

We now note that $\hpt$ and a left inverse $\Delta^{-1}$, considered as maps of forms, depend smoothly on the metric. This seems to be well-known, and can be shown by first fixing a cohomology class and showing its representative depends smoothly on the metric, and then simply applying inner products. As in, for instance, Proposition \ref{transfersmoothness}, we shall choose a finite-dimensional family of metrics parametrised by $\U$. 
\begin{prop}
\label{revised:hptsmoothinmetric}
Let $M$ be a compact manifold and $g_s$ be a finite-dimensional smooth family of smooth metrics on it parametrised by $s \in \U \subset \R^m$, with a base point $0 \in \U$. Then the map
\begin{equation}
\label{eq:hptmap}
\hpt: \U \times \Omega^p(M) \to \Omega^p(M)
\end{equation}
is smooth, perhaps after shrinking the neighbourhood $\U$ of $g_0$.
\end{prop}

Equally, we can find a left inverse to the Laplacian smoothly in $s$. 

These metric smoothness results pass immediately to smoothness of the corresponding operators on normal vector fields on a compact nearly special Lagrangian submanifold $L$ of a Calabi--Yau manifold $M$. We begin by defining these operators. 
\begin{defin}
\label{defin:hptonsubmanifolds}
Let $M$ be a Calabi--Yau manifold and let $L_s$ be a finite-dimensional smooth family of smooth submanifolds parametrised by $s \in \U \subset \R^m$, with a base point $0 \in \U$. Suppose that for all $s \in \U$, $v \mapsto \iota_v \omega|_{L_s}$ is an isomorphism of bundles. Let $T_s$ be the family of transfer maps from $L_0$ to $L_s$, given by $T_2$ from Definition  \ref{defin:sectiontransfers}, so that $T_s = T_{2, s}$ in Proposition \ref{fieldtransferfundamental}. 

Define maps 
\begin{equation}
\label{eq:hptsmoothinsubmanifold}
\hpt: \U \times \{\text{normal vector fields on } L_0\} \to \{\text{normal vector fields on } L_0\}
\end{equation}
and
\begin{equation}
\label{eq:laplaceinvmapsubmfd}
\Delta^{-1}: \U \times \{\text{normal vector fields on } L_0\} \to \{\text{orthoharmonic normal vector fields on } L_0\},
\end{equation}
as follows.

$\hpt(s, v)$ is given by taking $T_s(v)$ a normal vector field on $L_s$, taking the harmonic part $u$ using the correspondence between $1$-forms and normal vector fields, and then setting $\hpt(s, v) = T_s^{-1}(u)$. 

$\Delta^{-1}(s, v)$ is given by taking the normal vector field $T_s(v)$ on $L_s$ and then finding an orthoharmonic normal vector field $u$ on $L_0$ such that $\Delta T_s u = v-\hpt(s, v)$. 
\end{defin}
\begin{prop}
\label{revised:hptsmoothinsubmanifold}
Let $M$, $\{L_s\}$, $\hpt$ and $\Delta^{-1}$ be as in Definition \ref{defin:hptonsubmanifolds}. Then the maps $\hpt$ and $\Delta^{-1}$ are well-defined and smooth, perhaps after reducing the neighbourhood $\U$ of $L_0$. 
\end{prop}
Since the transfer maps $T_s$ and the identifications between normal vector fields and one-forms are smooth, this follows immediately from composition. Indeed, we can explicitly evaluate the derivative of $\hpt$:
\begin{prop}
\label{harmptderivative}
Suppose that $M$ is a Calabi--Yau manifold and $L_s$ is a finite-dimensional smooth family of smooth submanifolds as in Definition \ref{defin:hptonsubmanifolds}. Let $T_s$ and $\hpt$ be as in Definition \ref{defin:hptonsubmanifolds}. Let $v_s$ be a smooth curve of normal vector fields to $L_0$. Then $\hpt_s v_s$ is again a smooth curve of normal vector fields to $L_0$, and its derivative satisfies
\begin{align}
\label{eq:harmptderivativefirst}
\harmpt_0 \left(\left.\frac{d}{ds}\right|_{s=0} \hpt_s v_s\right) &= \hpt_0 (v' - \Delta' u), \\
\Delta_0 \left(\left.\frac{d}{ds}\right|_{s=0} \hpt_s v_s\right) &= - \Delta' \hpt_0 v_0.
\end{align}
where $u$ satisfies $\Delta_0 u = v - \hpt_0 v$, and $\Delta'$ is the derivative of the induced operator $\Delta_s$ in $s$ (where $\Delta_s$ is again induced by the transfer operator $T_s$). 
\end{prop}
\begin{proof}
That $\harmpt_s(v_s)$ is a smooth curve follows immediately from Proposition \ref{revised:hptsmoothinsubmanifold}.

To identify the derivative, we observe that we must have 
\begin{equation}
\label{eq:hptderivativeone}
\Delta_s \hpt_s v_s \equiv 0 \qquad v_s - \hpt_s v_s = \Delta_s u_s,
\end{equation}
for some smooth curve of normal vector fields $u_s$. Note that this curve exists and is smooth again by Proposition \ref{revised:hptsmoothinsubmanifold} (and $u_s$ is chosen to be orthoharmonic on $L_0$). 

Differentiating \eqref{eq:hptderivativeone} in $s$, and applying $\hpt_0$ to the second equation, gives the stated result. 
\end{proof}

\subsection{The derivative of \texorpdfstring{$\SLing$}{SLing}}
We may now prove that, on restriction to a finite-dimensional manifold $\U$ of perturbations $L_s$ of a nearly special Lagrangian submanifold $L_0$, $\SLing$ is smooth and give its derivative, under a hypothesis that the transfer maps of subsection \ref{ssec:riemanniannormalvectorfields} behave well with respect to the harmonic normal vector fields, which is essentially a condition that the base point $L_0$ is close enough to $\SLing(L_0)$.  Recall that $\SLing(L_s)$ is defined as the special Lagrangian $L'$ given by $\exp_v(L_s)$ where $\hpt(v) = 0$ in the sense of the discussion after Definition \ref{defin:laplacianonnvfs}. 

Let $\mathcal U$ be an open set in a finite-dimensional space of perturbations of the submanifold $L_0$ and $\U'$ an open set of the special Lagrangian deformations of $L'_0 = \SLing(L_0)$. We write 
\begin{equation}
\label{eq:SLingimplicit}
F: \begin{tikzcd}[row sep=0pt]
\mathcal{U} \times \U' \ar{r} &\U \times H^1(L), \\[0pt]
(L_s, L'_{s'}) \ar[mapsto]{r} &(L, \alpha),
\end{tikzcd}
\end{equation}
where if $v$ is the normal vector field on $L_s$ giving $L'_{s'}$, $\alpha$ is the cohomology class of the harmonic part of $\iota_{v}\omega|_{L_s}$. $F$ is a well-defined map if $\U$ and $\U'$ are small enough.  We have
\begin{equation}
F(L_s, \SLing(L_s)) = (L_s, 0),
\end{equation}
for all $L_s \in \U$. We shall show that $F$ is smooth and establish its derivative at $(L_0, \SLing(L_0))$, $D_{(L_0, \SLing(L_0))} F$, give a condition under which $D_{(L_0, \SLing(L_0))} F$ is an isomorphism, and consequently invert it under this condition to find $D_{L_0}\SLing$. 
\begin{prop}
\label{slingimplicitderivative}
Let $M$ be a Calabi--Yau manifold, and let $L_0$ be a submanifold of $M$ in the domain of the SLing map of Definition \ref{defin:slingmap}. Let $\U$, $L'_0$ and $\U'$ be as described above. Let $v_0$ be the orthoharmonic normal vector field on $L_0$ such that $\exp_{v_0}(L_0) = \SLing(L_0)$. 
Then the map $F$ of \eqref{eq:SLingimplicit} is smooth. We have
\begin{equation}
\label{eq:dslingimplict}
D_{(L_0, L'_0)}F:
\begin{tikzcd}[row sep=0pt]
T_{L_0} \U \oplus T_{L'_0} \U' \ar{r} &T_{L_0} \U \oplus H^1(L),\\[0pt]
(v, u) \ar[mapsto]{r} &(v, [\harmpt \iota_{T_1^{-1} (T_2 v + T_3 v - u) - \Delta'_v x}\omega|_{L_0}]),
\end{tikzcd}
\end{equation}
where $T_1$ and $T_2$ are as in Definition \ref{defin:sectiontransfers} (for the transfer from $L_0$ to $L'_0$ via $v_0$), $T_3$ is similarly as in Definition \ref{defin:t3}, $\Delta'_v$ is the derivative defined in Proposition \ref{harmptderivative}, and $x$ is (any) normal vector field to $L_0$ with $\Delta x = v_0$. 
\end{prop}
\begin{proof}
We consider $F$ as a composition of two maps. Firstly we find a normal vector field $w(s, s')$ on $L_0$ such that $T_{2, s}(w(s, s'))$ is the normal vector field on $L_s$ giving $L'_{s'}$, where $T_{2, s}$ is taken between $L_0$ and $L_s$. Then we take the cohomology class of the harmonic part of the one-form $\iota_{T_{2, s}(w(s, s'))}\omega|_{L_s}$ with respect to the metric induced on $L_s$. 

By Proposition \ref{transfersmoothness}, $w(s, s')$ depends smoothly on $(s, s')$. Analogously to Proposition \ref{revised:hptsmoothinsubmanifold} (composing a further map $I'$ induced from $I'$ of Definition \ref{defin:iandiprime} to translate normal vector fields into one-forms) we know that $\harmpt$ is smooth in precisely the sense required. It follows that $F$ is smooth. 

We now find the derivative $D_{(L_0, L'_0)} F$. We first recall from Proposition \ref{combinedtransfer} that the derivative of the map $(s, s') \mapsto w(s, s')$ is given by 
\begin{equation}
(v, u) \mapsto T_1^{-1}(T_2 v + T_3 v - u),
\end{equation}
where $v$ is a normal vector field to $L_0$ and $u$ is a normal vector field to $L'_0$. Given curves $L_s$ and $L'_s$ with normal vector fields $v$ and $u$, let $w' = T_1^{-1}(T_2 v + T_3 v - u)$ be the derivative of the corresponding curve $w(s)$ of normal vector fields on $L_0$. 

We computed in Proposition \ref{harmptderivative} that the derivative of a curve $\harmpt_s T_2(w(s))$ has harmonic part $\harmpt(w' - \Delta' x)$, where $\Delta x$ gives the orthoharmonic part of $w(0)$, and Laplacian $-\Delta' \hpt w(0)$. In this case, $w(0) = v_0$ and consequently it is orthoharmonic. Consequently the derivative of the curve $\harmpt_s T_2(w(s))$ is harmonic, and so it is
\begin{equation}
\harmpt_0(w' - \Delta' x) = \harmpt_0(T_1^{-1}(T_2 v + T_3 v - u) - \Delta' x),
\end{equation}
where $\Delta x = v_0$. 
\end{proof}

From Proposition \ref{slingimplicitderivative} we obtain
\begin{prop}
\label{slingderivative}
Let $M$, $L_0$, $\U$, $L'_0$, $\U'$, and $v_0$ be as in Proposition \ref{slingimplicitderivative}. Suppose that for every nonzero harmonic normal vector field $u$ on $L'_0$, $\hpt T^{-1}_1 (u)$ is a nonzero harmonic normal vector field on $L_0$. Then $\SLing$, restricted to a sufficiently small open subset of $\U$, is a smooth map. Its derivative is given by mapping a normal vector field $v$ to the unique harmonic normal vector field $u$ to $L'_0$ such that
\begin{equation}
\label{eq:DSLing}
\hpt T_1^{-1} u = \hpt (T_1^{-1}(T_2 v + T_3 v) - \Delta'_v x),
\end{equation}
where $T_1$, $T_2$ and $T_3$ are as in Proposition \ref{slingimplicitderivative}, $\Delta'$ is as in Proposition \ref{harmptderivative} and $\Delta x = v_0$.
\end{prop}
\begin{rmk}
Note that if $L_0$ is special Lagrangian, so $L'_0 = L_0$ and $v_0 = 0$ (and we may take $x=0$), $T_1$ and $T_2$ are the identity and $T_3$ is zero, so \eqref{eq:DSLing} reduces to $u = \hpt v$. 
\end{rmk}
\begin{proof}
We apply the inverse function theorem to the map $F$ of \eqref{eq:SLingimplicit} at $(L_0, L'_0)$. We first show that $D_{(L_0, L'_0)}F$ is an isomorphism. We know that $D_{(L_0, L'_0)}F$ is a linear map between spaces of the same finite dimension so it suffices to prove that it is injective. That is, we suppose that $v \in T_{L_0} \U$ and $u \in T_{L'_0} \U'$ satisfy $D_{(L_0, L'_0)}F (v, u) = 0$. Applying \eqref{eq:dslingimplict}, $v=0$ automatically and
\begin{equation}
0 = \iota_{\hpt (T_1^{-1}(T_2 v + T_3 v - u) + \Delta'_v x)}\omega|_{L_0} = \iota_{-\hpt T_1^{-1} u}\omega|_{L_0}.
\end{equation}
The hypothesis then implies that $u$ must be zero, so that $(v, u) = 0$ and $D_{(L_0, L'_0)}F$ is injective. 

By the inverse function theorem, therefore, when we restrict to a small neighbourhood of $(L_0, L'_0)$ in $\U \times \U'$, $F$ becomes a diffeomorphism. $\SLing(L_s)$ is precisely given by the second component of $F^{-1}(L_s, 0)$. Consequently, the derivative $(D\SLing)(v)$ is given by the harmonic normal vector field $u$ to $L'_0$ so that $DF_{(L_0, L'_0)}(v, u) = 0$. Rearranging to obtain \eqref{eq:DSLing} is straightforward. 
\end{proof}

\section{The derivative of approximate gluing}
\label{sec:dpatching}
We now prove the results corresponding to the previous section for the approximate gluing map of Definition \ref{defin:newsubmanifoldapproxgluing}. As this map does not involve the Laplacian, the results corresponding to subsection \ref{ssec:laplacianonnvfs} are not required; however, the analysis corresponding to subsections \ref{ssec:riemanniannormalvectorfields} and \ref{ssec:normalvectorfieldsspeclag} is more involved, as we need to extend these results to asymptotically cylindrical submanifolds and asymptotically translation invariant normal vector fields. We do this extension in subsection \ref{ssec:dpatchingbasics}, and then pass to approximate gluing in subsection \ref{ssec:dpatchingproper}. 

\subsection{Asymptotically translation invariant normal vector fields}
\label{ssec:dpatchingbasics}
In this subsection we give a general definition of asymptotically translation invariant normal vector fields, show that it behaves well with respect to the maps $T_i$ from subsection \ref{ssec:riemanniannormalvectorfields}, and finally show that this definition is, for an asymptotically cylindrical close to special Lagrangian submanifold of an asymptotically cylindrical Calabi--Yau manifold, equivalent to the corresponding one-form being asymptotically translation invariant. 

We briefly discussed the definition of asymptotically translation invariant normal vector fields for the asymptotically cylindrical deformation theory of Theorem \ref{acyldeformationtheorem}. In that case, we explained briefly that for special Lagrangian $L$, we could take $v$ asymptotically translation invariant if and only if $\iota_v \omega|_L$ is. To define asymptotically translation invariant vector fields in general, we use that each asymptotically cylindrical submanifold has a corresponding cylindrical submanifold. 

\begin{defin}
\label{defin:newatinvf}
Let $L$ be an asymptotically cylindrical submanifold of the asymptotically cylindrical manifold $M$. Let $K \times (R, \infty)$ be the cylindrical end, so that the end of $L$ is $\exp_v(K \times (R, \infty)$ for $v$ decaying. 

Translation gives an action on $TM|_{K \times (R, \infty)}$, and consequently a notion of \emph{translation invariant vector field}. Since we always have a notion of exponentially decaying vector fields, we obtain a notion of \emph{asymptotically translation invariant vector fields} on $K \times (R, \infty)$. 

We may extend $v$ to obtain an asymptotically cylindrical diffeomorphism with limit the identity (as in Proposition \ref{everythinginducedisacyl}) of tubular neighbourhoods of $K \times (R, \infty)$ and the end of $L$. Pushforward by this diffeomorphism induces a map from vector fields along $K \times (R, \infty)$ to vector fields along the end of $L$; as $v$ and its first derivative are exponentially decaying, it follows as in Definition \ref{defin:pointwisetransfers} (that is, by using the Rauch comparison estimate Proposition \ref{rauchcomparison}) that this defines an isomorphism far enough along the end. 

We then say that a vector field along $L$ is \emph{asymptotically translation invariant} precisely if it is the image of an asymptotically translation invariant vector field under this pushforward. An asymptotically translation invariant normal vector field is simply asymptotically translation invariant and normal. We say an asymptotically translation invariant vector field's \emph{limit} is the translation invariant vector field along $K \times (R, \infty)$ given by the limit of the vector field on $K \times (R, \infty)$ of which it is a pushforward. 
\end{defin}
We can define an extended weighted norm on such asymptotically translation invariant vector fields by using the standard extended weighted norm corresponding to $\exp_v^* g$ for a suitable extension of $v$ on the normal vector field on the cylindrical submanifold with limit $K \times (R, \infty)$.

The purpose of using pushforward for this transfer and not restricting to normal vector fields in the definition is that it makes the following result trivial. 
\begin{lem}
\label{pushforwardativf}
Suppose that $L_1$ and $L_2$ are asymptotically cylindrical submanifolds with the same limit, so that there is an exponentially decaying normal vector field $w$ to $L_1$ such that $\exp_w(L_1) = L_2$. Extend $w$ to define an asymptotically cylindrical diffeomorphism of tubular neighbourhoods, with limit the identity. A vector field along $L_1$ is asymptotically translation invariant if and only if its image under the pushforward by this diffeomorphism is. Moreover, the vector field and its pushforward have the same limit. 
\end{lem}
\begin{proof}
We restrict to the ends of of $L_1$, $L_2$ and the corresponding cylindrical end $\tilde L$. As the compact parts are irrelevant for this discussion, we will just write these ends as $L_1$ and $L_2$. We have $L_1 = \exp_{v_1}(\tilde L)$, $L_2 = \exp_{v_2}(\tilde L)$; again, make some extensions $v_1$ and $v_2$ to define asymptotically cylindrical diffeomorphisms with limit the identity between tubular neighbourhoods. Suppose $u$ is an asymptotically translation invariant vector field on $L_1$; that is, it is the pushforward by $\exp_{v_1}$ of an asymptotically translation invariant vector field on $\tilde L$. We want to show that the pushforward by $\exp_w$ of $u$ is a translation invariant vector field on $L$ with the same limit; that is, we want to show that the pushforward by the composition $\exp_{w}\exp_{v_1}$ of an asymptotically translation invariant vector field is the pushforward by $\exp_{v_2}$ of an asymptotically translation invariant vector field with the same limit. Equivalently, it suffices to show that $\exp_{v_2}^{-1}\exp_{w}\exp_{v_1}$ (which is essentially a diffeomorphism of the end $\tilde L$) preserves asymptotically translation invariant vector fields and their limits. But this is an asymptotically cylindrical diffeomorphism, and so pullback induces a corresponding asymptotically cylindrical metric on the tubular neighbourhood of $\tilde L$, and the question just becomes the independence of asymptotic translation invariance on metric. This is immediate as usual. As the diffeomorphism has limit the identity, it preserves the limits of the vector fields. 

The reverse implication is equally obvious.
\end{proof}
This implies in particular that Definition \ref{defin:newatinvf} is independent of the extension of $v$ used. We will now explain some consequences of Definition \ref{defin:newatinvf}, particularly with respect to the transfer maps $T_1$, $T_2$ and $T_3$.
\begin{prop}
\label{revisedtransferofatis}
Let $L$ be an asymptotically cylindrical submanifold of $M$; let $L'$ be another asymptotically cylindrical submanifold with the same limit. Let $u$ be an asymptotically translation invariant normal vector field on $L$. Then
\begin{enumerate}[i)]
\item $\|u\|_{C^k}$ is finite for every $k$. 
\item $T_1 u$ and $T_2 u$ are asymptotically translation invariant normal vector fields on $L'$ with the same limit as $u$. Similarly, if $u'$ is an asymptotically translation invariant normal vector field to $L'$, $T^{-1}_1 u'$, $T^{-1}_2 u'$ are asymptotically translation invariant normal vector fields on $L$ with the same limits as $u'$. Finally, $T_3 u$ is an exponentially decaying normal vector field on $L'$. 
\item If the ambient metric is cylindrical as in subsection \ref{ssec:speclagdeformation}, asymptotically translation invariant normal vector fields correspond to nearby asymptotically cylindrical submanifolds, with the limits also corresponding. 
\end{enumerate}
\end{prop}
\begin{proof}
We begin with (i). By definition, away from a compact part, $u$ is the image under pushforward of an asymptotically translation invariant vector field $w$ on $K \times (R, \infty)$. $w$ is clearly bounded in $C^k$ for every $k$, at least for $R$ large enough. We have to show that pushforward preserves this, and this follows easily from the smoothness and continuity result obtained as part of Proposition \ref{basicgluingsmoothness}. For instance, $|u_{\exp_v(p)}|$ is precisely given by applying the pointwise pushforward map with $w_p$ and $(v_p, (\nabla v)_p)$; this depends smoothly on $p$ and converges to a limit as we approach the end, using the continuity in the metric. This proves the $C^0$ bound; the $C^k$ bound is similar by using the derivatives. 

The major part of this proposition is (ii). We shall prove the result for $T_1$ and $T_2$ and their inverses; $T_3$ is similar. As usual we shall argue using the ideas of Proposition \ref{basictransfersmoothness}. It follows from Lemma \ref{pushforwardativf} that the result is true just for applying pushforward, so it suffices to check two things. Firstly, we shall show that these four maps are close to pushforward in the sense that their differences decay exponentially. Secondly, we shall show that exponentially decaying vector fields are asymptotically translation invariant with zero limit. By the linearity of Definition \ref{defin:newatinvf}, it follows that the image under these maps is asymptotically translation invariant, and that the limits are preserved. 

We begin by working with $T$ which is either $T_1$ or $T_2$. This is defined pointwise by a map on $TM \oplus J^1 (TM) \oplus \Gr_n(TM)$. By Proposition \ref{basictransfersmoothness}, this map is smooth and depends continuously on the metric in the sense of that proposition. Similarly, pushforward, which we shall denote $P$, is defined pointwise by a map on $TM \oplus J^1(TM)$, which by the appropriate part of Proposition \ref{basicgluingsmoothness} is smooth and depends continuously on the metric. We shall restrict to the submanifold $N$ of $TM \oplus J^1(TM) \oplus \Gr_n(TM)$ defined in Definition \ref{defin:newbigsubbundles}; that is, $(u, (v, \nabla v), \ell)$ with $u$ normal to $\ell$. For any such $u$ and $\ell$, we immediately have
\begin{equation}
|T(u, 0, \ell) - P(u, 0)| = 0.
\end{equation}
Since these maps are smooth, it follows that for each $p$ we have a constant $C$ with
\begin{equation}
|T(u, (v, \nabla v), \ell) - P(u, (v, \nabla v))| \leq C |(v, \nabla v)||u|,
\end{equation}
as a pointwise estimate. As this constant may be chosen smoothly, and we may differentiate and obtain the same results for jets, we have a local estimate
\begin{equation}
\|T(u, (v, \nabla v), \ell) - P(u, (v, \nabla v))\|_{C^k} \leq C_k \|v\|_{C^{k+1}}\|u\|_{C^k}.
\end{equation}
Just as in Proposition \ref{speclaggluingremainderestimate}, we now only have to show that $C_k$ may be chosen uniformly as $v$ is exponentially decaying and $u$ is bounded by (i). But, using that we have \eqref{eq:oldcomplicatedsctyest} and its analogue for pushforward, the derivatives are continuous in the metric. $L$ itself can be regarded as a finite-dimensional parameter space of metrics, and as $p$ heads to the end of $L$, the metric converges to a limit $\tilde g$. Hence, the whole map depends continuously on the point of $L$, and we may choose a uniform bound.

This shows that the image of a normal vector field by $T$ and by pushforward differ by an exponentially decaying vector field. 

As for the inverses, we know from the inverse part of Proposition \ref{basictransfersmoothness} and the analogous result, proved identically, for the inverse of pushforward that these satisfy the same properties on $N'$ from Definition \ref{defin:newbigsubbundles}: the result is then proved in exactly the same way. 

It now only remains to show that exponentially decaying vector fields are asymptotically translation invariant with zero limit, that is they are the pushforwards of exponentially decaying vector fields on $\tilde L$. To do this, we again apply the inverse of pushforward. Just as above, we know that we can take an appropriate smooth submanifold of $TM \times J^1(TM)$, and define a map that is smooth and continuous in the metric. Note that as here we need to ensure the vector field to $L$ is normal, this submanifold is not $N'$. It follows in the same way that the image of a pair of exponentially decaying objects is exponentially decaying, since the pushforward of zero is always zero and we converge to a constant metric. The converse result can be proved the same way: an asymptotically translation invariant vector field with zero limit is exponentially decaying. 

As for (iii), the relationship between asymptotically translation invariant normal vector fields and asymptotically cylindrical manifolds holds just as sketched in subsection \ref{ssec:speclagdeformation}; passing to pointwise operators on appropriate bundles as in this section enables us to formalise that argument. Note that the condition that the limit was zero for (ii) meant we did not have to assume cylindricality of the ambient metric, whereas that will be necessary in this case.  
\end{proof}

We now prove that asymptotically translation invariant one-forms correspond to asymptotically translation invariant normal vector fields. That is, Definition \ref{defin:newatinvf} is equivalent to the one-forms we used in subsection \ref{ssec:speclagdeformation}. 
\begin{prop}
\label{aticonsistency}
Let $M$ be an asymptotically cylindrical Calabi--Yau manifold, and let $L$ be an asymptotically cylindrical submanifold so that $\omega|_L$ is small with all derivatives, so that the estimate \eqref{eq:Jbehavesbounds} of Lemma \ref{Jbehaves} applies in every $C^k$ space. Then a normal vector field $v$ on $L$ is asymptotically translation invariant in the sense of Definition \ref{defin:newatinvf} if and only if the one-form $\iota_v \omega|_L$ is asymptotically translation invariant. 
\end{prop}
\begin{proof}
As usual, we are only interested in the end of $L$, and this is asymptotic to a cylindrical submanifold $\tilde L$. There is an exponentially decaying normal vector field $v$ on $L$ whose image under the exponential map is $\tilde L$; this extends to a diffeomorphism $\exp_v$ between tubular neighbourhoods decaying to the identity, and Definition \ref{defin:newatinvf} says a vector field $u$ along $L$ is asymptotically translation invariant if and only if it is $(\exp_v)_* w$ for some asymptotically translation invariant vector field $w$ along $\tilde L$. 

It suffices to suppose that $L$ is cylindrical, as follows. With the notation as above, 
\begin{equation}
\iota_u \omega|_L = \iota_{(\exp_v)_* w} \omega|_{\exp_v(\tilde L)} = \iota_w \exp_v^* \omega|_{\tilde L}.
\end{equation}
Moreover, if $u$ is normal, $w$ is normal with respect to the metric $\exp_v^*g$ and vice versa. Hence, the result is true for $L$ if and only if it is true for $\tilde L$ with the metric $\exp_v^*g$ and the $2$-form $\exp_v^* \omega$; these correspond to the asymptotically cylindrical Calabi--Yau structure $(\exp_v^* \Omega, \exp_v^* \omega)$, and as $v$ is exponentially decaying and $\omega$ asymptotically translation invariant $\exp_v^* \omega$ also restricts to a small form (essentially as in Proposition \ref{revisedtransferofatis}), and so this reduces the problem to $\tilde L$. 

If $L$ is cylindrical, we just note that $\omega$ determines an asymptotically translation invariant section of $\bigwedge^2 T^*M|_L$. Composing with the projection map $T^*M|_L \to T^*L$, we obtain an asymptotically translation invariant section of $T^*L \otimes (T^*M|_L)$. Hence, if $v$ is an asymptotically translation invariant normal vector field, the corresponding one-form is asymptotically translation invariant. Moreover, we know that if we restrict this section to $\nu_L$, it defines an asymptotically translation invariant section of $T^*L \otimes (\nu_L)^*$. By Lemma \ref{Jbehaves}, this section consists of isomorphisms, so we can invert it; it is easy to see that the inverse section is again asymptotically translation invariant. This completes the cylindrical case, and hence the proof. 
\end{proof}

\subsection{The derivative of approximate gluing}
\label{ssec:dpatchingproper}
In this subsection, we will deal with the approximate gluing map of Definition \ref{defin:newsubmanifoldapproxgluing}. We will identify its derivative, and use this to define an approximate gluing map of normal vector fields. We will also show that this gluing map, when we pass to special Lagrangians and use the identifications between normal vector fields and one-forms, is close to the approximate gluing map of one-forms given in Definition \ref{defin:newformgluing}. In particular, this shows that for special Lagrangian submanifolds the approximate gluing map of submanifolds is smooth, and its derivative is close to the approximate gluing map of one-forms. 

We make the following preliminary observation. To discuss smoothness of maps of submanifolds, we treat submanifolds as equivalent to normal vector fields using the Riemannian exponential map. It is not hard to see, essentially by the argument of Hamilton \cite[Example 4.4.7]{hamiltonnashmoser}, that smoothness of maps is independent of the metric used, \ie that the identity map on submanifolds defines a smooth map of vector fields normal with respect to different metrics, and that the derivative of this map is just given by taking the normal part of the vector field with respect to the new metric.  Compare Proposition \ref{basetransferfundamental}. 

This observation is relevant to this analysis for various reasons, but most obviously because we have to change the metric in gluing the ambient spaces $M_1$ and $M_2$. Recall from Definition \ref{defin:newsubmanifoldapproxgluing} that to glue asymptotically cylindrical submanifolds $L_1$ and $L_2$, we cut them off to form $\hat L_1$ and $\hat L_2$ and then identify. The observation shows that, regardless of metric, the identification part of this is smooth. Its derivative is given by taking the normal parts with respect to the cutoff metrics and then identifying; equivalently, it is given by identification and taking the normal part with respect to the final metric on $M^T$. 

We thus need to understand the cutoff map $L_i \mapsto \hat L_i$, and may work with only one asymptotically cylindrical submanifold $L$, with cutoff $\hat L$, in a fixed asymptotically cylindrical $M$. Of course, since it depends on cutting off a normal vector field, this map depends on some fixed metric on $M$, but we may fix one once for all. Since we need to find a derivative, we in fact work with a family $L_s$ of asymptotically cylindrical submanifolds with cutoffs $\hat L_s$ and cross-sections $K_s$. We assume that this family decays at a fixed uniform rate, so that we can fix a decay rate for normal vector fields; note that if $L_s$ is a deformation family of special Lagrangians this is immediate by the argument at the end of subsection \ref{ssec:speclagdeformation}. 

We now introduce families of normal vector fields giving $L_s$ and $\hat L_s$, and prove that $\hat L_s$ is a smooth family. First of all, we have normal vector fields $u_s$ and $\hat u_s$ to $L_0$ and $\hat L_0$ so that $\exp_{u_s}(L_0) = L_s$ and $\exp_{\hat u_s}(\hat L_0) = \hat L_s$. These normal vector fields are not well-adapted to the cutoff, and so we define some further normal vector fields. There is a normal vector field $v_s$ to $K_s \times (R, \infty)$ giving the end of $L_s$; it depends smoothly on $s$ by the remark after Proposition \ref{transfersmoothness}. Moreover, there is a normal vector field $w_s$ to $K_0 \times (R, \infty)$ giving $K_s \times (R, \infty)$; this also depends smoothly on $s$. The cutoff function needed in Definition \ref{defin:newsubmanifoldapproxgluing} is then $\varphi_s = \varphi_{T}\circ\exp_{w_s}\circ\iota$ where $\iota$ is the inclusion of $K_0 \times (R, \infty)$.  Thus the resulting cutoff normal vector field depends smoothly on $s$, and so so does $\hat L_s$. This shows that the map we consider is smooth. 

By using the transfer map $T_{2, s}$ as in Proposition \ref{fieldtransferfundamental} to consider the family $v_s$ of normal vector fields in the previous paragraph on the same space, we can summarise this as
\begin{align}
L_s &= \exp_{T_{2, s}(v_s)}(K_s \times (R, \infty)) = \exp_{T_{2, s}(v_s)}\exp_{w_s}(K_0 \times (R, \infty)), \\ \hat L_s &= \exp_{ T_{2, s}(\varphi_s v_s)} \exp_{w_s}(K_0 \times (R, \infty)),
\end{align}
where we make a change to the definition of $v_s$.

Now this expression enables us to find the derivative. Let $T_1$ and $T_2$ be the transfer maps of Definition \ref{defin:sectiontransfers} with $L=K_i \times (R_i, \infty)$ and $L' = \exp_{v_i}((R_i, \infty) \times K_i)$; let $\hat T_1$ and $\hat T_2$ be the corresponding transfer maps with $L= K_i \times (R_i, \infty)$ and $L' = \exp_{\varphi v_i}((R_i, \infty) \times K_i)$. Define $T_3$ and $\hat T_3$ similarly  using Definition \ref{defin:t3}. 

It follows from Proposition \ref{combinedtransfer} that the normal vector field to $L_0$ giving the tangent to the curve $L_s$ is $u' = T_2 w' + T_1 v' + T_3w'$; similarly, the normal vector field to $\hat L_0$ giving the tangent to the curve $\hat L_s$ is $\hat T_2 w' + \hat T_1 (\varphi_s v_s)' + \hat T_3w' = \hat T_2 u' + \hat T_1( \varphi_0 v' + \nabla_{u'} \varphi_{T}) + \hat T_3w'$. Note that $\varphi_{T}$ is defined on $M$, so the normal derivative $\nabla_{u'} \varphi_T$ makes sense. 

The derivative is consequently the map
\begin{equation}
\label{eq:nvfgluingderiv}
u' = T_2 w' + T_1 v' + T_3w' \mapsto \hat T_2 w' + \hat T_1 (\varphi_0 v' + \nabla_{u'} \varphi v') + \hat T_3 w' = \hat u'. 
\end{equation}
In order to understand this derivative explicitly, we have to explain how to obtain $w'$ and $v'$ from $u'$.  Note that if $M$ were cylindrical, $w_s$ would be translation invariant for all $s$, and thus so would $w'$. Hence by our preliminary observation, $w'$ is the normal part of the translation invariant vector field corresponding to the behaviour of the limits, and is determined by its limit. On the other hand, by assumption $v'$ decays at a uniform rate and so $v'$ also decays. 

Since $v_0$ and $v'$ are exponentially decaying, it follows by (ii) of Proposition \ref{revisedtransferofatis} that $T_1 v'$ is too. Similarly it follows that $T_3w'$ is exponentially decaying. 

On the other hand, since $w'$ is determined by its limit, the proof of Proposition \ref{revisedtransferofatis} implies $T_2 w'$ is also determined by its limit: we may find the limit of $w'$ from the limit of $T_2 w'$, hence $w'$, and thence $T_2 w'$. The previous paragraph says the limit of $T_2 w'$ is the limit of $u'$, hence the map of \eqref{eq:nvfgluingderiv} is well-defined. That is, the limit of $u'$ gives us $w'$ as above; then $T_1 v' = u' - T_2 w' - T_3 w'$ enables us to find $v'$. 

Note that at a point with $r< T-1$, we have that $T_1 = \hat T_1$, $T_2 = \hat T_2$, $T_3 = \hat T_3$, $\varphi_{T} \equiv 1$, and $\nabla \varphi_{T} = 0$. Hence \eqref{eq:nvfgluingderiv} becomes the identity map at these points, and so it can be extended to a map from normal vector fields on $L_0$ to normal vector fields on $\hat L_0$. 

This identifies the derivative of the approximate gluing map of submanifolds from Definition \ref{defin:newsubmanifoldapproxgluing}. It also defines an approximate gluing map of matching asymptotically translation invariant normal vector fields (that is, those vector fields $v_i$ on $L_i$ whose limits $\tilde v_i$ satisfy $F_* \tilde v_1 = \tilde v_2$). To check this, we just have to show that this condition implies the constructed normal vector fields match in the identification region. Since the limits are the same, we know that the limit of $w'$ is the same; thus $\hat T_2 w'$ agrees in the identification region provided we choose the metrics to agree. 

There are simpler definitions of such a gluing map, but we will have to use the transfer maps because to define a normal vector field on $L^T$, we will need normal vector fields on $\hat L_1$ and $\hat L_2$. 

We now translate this to one-forms. If $L_1$ and $L_2$ are a matching pair of special Lagrangians, and Hypothesis \ref{hyp:ambientgluing} holds, then by Proposition \ref{ambientimplications} and Lemma \ref{Jbehaves} $v \mapsto \iota_v \omega|_{L^T}$ gives an isomorphism between normal vector fields and one-forms on $L^T$. As $L_1$ and $L_2$ are themselves special Lagrangian we also have such an isomorphism on $L_1$ and $L_2$. With respect to these isomorphisms, we have
\begin{prop}
\label{nvfandoneformgluing}
Let $M_1$, $M_2$, $L_1$, $L_2$ and $L^T$ be as in Proposition \ref{ambientimplications}. Suppose $T$ is sufficiently large that $v \mapsto \iota_v\omega|_{L^T}$ is an isomorphism. Then we can induce a gluing map of normal vector fields from the gluing map of one-forms. This map differs from the gluing map described above by a linear map decaying exponentially in $T$. This means that there exist constants $\epsilon$ and $C_k$ so that the two gluing maps $G_1$ and $G_2$ satisfy
\begin{equation}
\label{eq:opnormsmallinT}
\|G_1(v_1, v_2) - G_2(v_1, v_2)\|_{C^k} \leq C_k e^{-\epsilon T} (\|v_1\|_{C^k} + \|v_2\|_{C^k}).
\end{equation}
\end{prop}
\begin{proof}
Let $\alpha_1$ and $\alpha_2$ be matching asymptotically translation invariant one-forms with common limit $\tilde \alpha$. We shall show that the difference of the one-form given by gluing these as above and $\gamma_T(\alpha_1, \alpha_2)$ is exponentially decaying: this proves the result as it follows from Proposition \ref{ambientimplications} and Lemma \ref{Jbehaves} that the isomorphism between normal vector fields and one-forms on $L^T$ is bounded uniformly in $T$. 

Let $u_i$ be the normal vector field corresponding to $\alpha_i$ using the Calabi--Yau structure on $M_i$, and let $u$ be the gluings of $u_1$ and $u_2$ as above. 

We prove this in two stages. Firstly, we prove that on each cutoff cylindrical submanifold $\hat L_i$, with cutoff normal vector field $\hat u_i$ as in \eqref{eq:nvfgluingderiv}, and cutoff one-form $\hat \alpha_i$, $\iota_{\hat u_i} \omega_i|_{\hat L_i} - \hat \alpha_i$ is exponentially small in $T$. We then note that $\iota_{\hat u_i} \omega_i|_{\hat L_i} - \iota_{\hat u_i} \hat\omega_i|_{\hat L_i}$ is exponentially small in $T$. This proves that $\iota_u \gamma_T(\omega_1, \omega_2)|_{L^T} - \gamma_T(\alpha_1, \alpha_2)$ is exponentially small in $T$. Secondly, we use Hypothesis \ref{hyp:ambientgluing} to show that $\iota_u (\gamma_T(\omega_1, \omega_2)-\omega^T)|_{L^T}$ is exponentially small in $T$.  

To do the first of these we simply replace $T_1$, $T_2$, $T_3$, $\hat T_1$, $\hat T_2$ and $\hat T_3$ with the corresponding maps $\tau_1, \tau_2, \tau_3, \hat \tau_1, \hat \tau_2, \hat \tau_3$ of one-forms as in Definition \ref{defin:taui}: Proposition \ref{tauiareti} says that these are just the one-form versions of the $T_i$ and $\hat T_i$. 

By construction, $\hat \alpha_i$ and $\iota_{\hat u_i}\hat\omega_i|_{\hat L_i}$ are both equal to $\alpha$, for $t< T-\frac32$, say. Hence, it suffices to consider the difference where $t> T-\frac32$, and so it suffices to prove that $\tau_1$, $\tau_2$, $\hat \tau_1$ and $\hat \tau_2$ are exponentially close in $T$ to the identity, and $\tau_3$ and $\hat \tau_3$ are exponentially small in $T$ in the sense of \eqref{eq:opnormsmallinT}. This follows by a similar argument to Proposition \ref{revisedtransferofatis} using Proposition \ref{tauismoothness}: we obtain local bounds on $\tau_i \alpha - \alpha$ in terms of $v$ and $\alpha$ using smoothness, and then continuity shows that these bounds may be chosen uniformly; since we may suppose $v$ is exponentially small in $T$, as we are only interested in this behaviour far enough along the end, the result follows. Note that this is essentially independent of which metric is used. 

It now only remains to show that $\iota_u (\gamma_T(\omega_1, \omega_2)-\omega^T)|_{L^T}$ is exponentially small in $T$. Hypothesis \ref{hyp:ambientgluing} says that $\gamma_T(\omega_1, \omega_2) - \omega^T$ is exponentially small in $T$; $u$ is uniformly bounded in $T$ since $\hat u_1$ and $\hat u_2$ are, so $\iota_u (\gamma_T(\omega_1, \omega_2) - \omega_T)$ is exponentially small in $T$, and Corollary \ref{newgluedmetricsallthesame} then implies that the restriction is. 
\end{proof}
Note that the norms $\|v_1\|_{C^k}$ may readily be bounded by the extended weighted norms, so we may suppose we have extended weighted norms on the right hand side of \eqref{eq:opnormsmallinT}. 

\section{Estimates in the gluing case and proof of Theorem B}
\label{sec:opennessthm}
We will now show that the gluing map of special Lagrangians given by combining the $\SLing$ map of Definition \ref{defin:slingmap} with the approximate gluing map of Definition \ref{defin:newsubmanifoldapproxgluing} is indeed a local diffeomorphism of moduli spaces for $T$ sufficiently large. This is Theorem B, Theorem \ref{speclaggluinglocaldiffeo} below. We will show that the derivative of this map, which we identified in the previous two sections under certain hypotheses, is an isomorphism. We will first show that the hypotheses required are satisfied. This will imply in particular that the derivative of $\SLing$ is close to taking the harmonic part of the one-form on the special Lagrangian submanifold given by gluing our asymptotically cylindrical pair (Proposition \ref{dslingisnearlyharmpt}). Finally, we show in Proposition \ref{manifoldofgluablesl} that the space of matching special Lagrangian submanifolds around any pair is a manifold, so that our derivative for the approximate gluing map in subsection \ref{ssec:dpatchingproper} applies, and prove using the linear harmonic theory in Proposition \ref{harmonicgluinglowerbound} that the composition of the derivatives is an isomorphism. 

We set up notation for our gluing analysis as follows: unfortunately, this notation is rather involved. 
\begin{convention}
\label{con:gluingslanalysis}
Let $(M_1,M_2)$ be a matching pair of asymptotically cylindrical Calabi--Yau manifolds and let $(L_1, L_2)$ be a matching pair of asymptotically cylindrical special Lagrangian submanifolds as in Definition \ref{defin:newsubmanifoldapproxgluing}. Suppose that Hypothesis \ref{hyp:ambientgluing} holds so that $(\Omega^T, \omega^T)$ is a Calabi--Yau structure on $M^T$. Let $L_0(T)$ be the family of glued submanifolds of $M^T$ given by approximately gluing. Suppose $T_0$ is sufficiently large that for $T> T_0$ sufficiently large Condition \ref{slingcondition} applies with $k$ and $\mu$, and let $L'_0(T)$ be the family of special Lagrangian submanifolds for $(\Omega^T, \omega^T)$ given by perturbing $L_0(T)$ as in Theorem A (Theorem \ref{slperturbthm}). Let $v_0(T)$ be the normal vector field to $L_0(T)$ giving $L'_0(T)$, and let $x^T$ be a normal vector field on $L_0(T)$ with $\Delta x^T = v_0(T)$. 
\end{convention}
We immediately have the following estimates
\begin{lem}
\label{opennessdecaylemma}
Suppose we are in the situation of Convention \ref{con:gluingslanalysis}. We may choose $x^T$ such that there exists fixed $\epsilon > 0$ and constants $C_{k, \mu}$ such that 
\begin{equation}
\|v_0(T)\|_{C^{k+1, \mu}} + \|x^T\|_{C^{k+3, \mu}} \leq C_{k, \mu}e^{-\epsilon T}
\end{equation}
for every $k$ and $\mu$. 
\end{lem}
\begin{proof}
The estimate on $\|v_0(T)\|_{C^{k+1, \mu}}$ follows immediately from Theorem \ref{slperturbthm}. As for $x^T$, this follows essentially from the Laplacian version of Theorem \ref{laplacelowerbound}, using Corollary \ref{newgluedmetricsallthesame} to note that $L_0(T)$ essentially has the glued metric, and Lemma \ref{Jbehaves} to note that bounding one-forms and normal vector fields is essentially equivalent. 
\end{proof}
As $L_0(T)$ and $L'_0(T)$ are close to special Lagrangian and special Lagrangian, respectively, we can identify normal vector fields on them with one-forms. We then obtain from Proposition \ref{slingderivative} that $D_{L_0(T)} \SLing$, if defined, is the map from a one-form $\alpha$ on $L_0(T)$ to the unique harmonic one-form $\beta$ on $L'_0(T)$ such that 
\begin{equation}
\label{eq:oneformdsling}
\harmpt \tau_1^{-1} \beta = \harmpt (\tau_1^{-1}(\tau_2 \alpha + \tau_3 \alpha) + \Delta'_\alpha x^T),
\end{equation}
where $\Delta'_\alpha x^T$ is given by finding the normal vector field $v$ on $L_0$ corresponding to $\alpha$, constructing the transfer operators $T_{2, s}$ from $L_0$ to $\exp_{sv}(L_0)$, constructing the curve of normal vector fields $T_{2, s}^{-1} \Delta T_{2, s} x^T$, and then taking the derivative of this curve in $s$ at zero, and $\tau_1$, $\tau_2$ and $\tau_3$ are as in Definition \ref{defin:taui}. 

We first show that for $T$ large enough, $\beta$ satisfying \eqref{eq:oneformdsling} is defined, by showing that the $\tau_i$ behave well as $T$ gets large. 
\begin{prop}
\label{transfermapsoneforms}
Let $M_1$, $M_2$, $L_1$, $L_2$, $L_0(T)$, $L'_0(T)$ be as in Convention \ref{con:gluingslanalysis}, and $\epsilon$ as in Lemma \ref{opennessdecaylemma}. Consider the maps $\tau_1$, $\tau_2$ and $\tau_3$ defined in Definition \ref{defin:taui} from one-forms on $L_0(T)$ to one-forms on $L'_0(T)$; these maps also depend on $T$. There exists a sequence of constants $C_k$ such that for every $k$
\begin{equation}
\|\alpha - \tau_1 \alpha\|_{C^k} + \|\alpha - \tau_2 \alpha\|_{C^k} + \|\tau_3 \alpha\|_{C^k} \leq C_k e^{-\epsilon T}\|\alpha\|_{C^k}.
\end{equation}
\end{prop}
Essentially, this follows by the same proof as in Proposition \ref{speclaggluingremainderestimate}. We can always find a local constant, giving a bound in terms of $v_0(T)$; by the argument in Proposition \ref{speclaggluingremainderestimate} this constant can be chosen uniform in $T$, and then the exponential decay follows from the exponential decay of $v_0(T)$. 

The other major ingredient in \eqref{eq:oneformdsling} (as well as the $\tau_i$) is the Laplacian. We have the following
\begin{prop}
\label{hptslgluinganalysis}
Let $M_1$, $M_2$, $L_1$, $L_2$, $M^T$, $L_0(T)$, $L'_0(T)$, $v_0(T)$ be as in Convention \ref{con:gluingslanalysis},  and $\epsilon$ as in Lemma \ref{opennessdecaylemma}. We have two metrics $g(\Omega^T, \omega^T)|_{L_0(T)}$ and $g(\Omega^T, \omega^T)|_{L'_0(T)}$ on the submanifold $L_0(T)$ of $M^T$. There exist constants $C_k$ such that for all $k$
\begin{equation}
\label{eq:gluingslmetricdecay}
\|g(\Omega^T, \omega^T)|_{L_0(T)} - g(\Omega^T, \omega^T)|_{L'_0(T)}\|_{C^{k}} \leq C_k e^{-\epsilon T}.
\end{equation}

Secondly, these two metrics induce two Laplacians $\Delta_{L_0(T)}$ and $\Delta_{L'_0(T)}$ on forms on these submanifolds. There exist constants $C'_{k, \mu}$ such that for any integer $k$ and $\mu \in (0, 1)$ and for any form $\alpha$
\begin{equation}
\label{eq:Deltadecexp}
\|\Delta_{L_0(T)} \alpha - \Delta_{L'_0(T)} \alpha\|_{C^{k, \mu}} \leq C'_{k, \mu} e^{-\epsilon T} \|\alpha\|_{C^{k+2, \mu}},
\end{equation}
and the same estimate for $d^*$. Thirdly, possibly increasing $C'_{k, \mu}$ there also exists $r$ such that if $\alpha$ is orthoharmonic with respect to the metric on $L_0(T)$ or $L'0(T)$, 
\begin{equation}
\label{eq:laplacelowerbounddsling}
\| d \alpha\|_{C^{k, \mu}} + \| d^* \alpha\|_{C^{k, \mu}} \geq C'_{k, \mu} T^{r} \|\alpha\|_{C^{k+1, \mu}},
\end{equation}
where the Laplacian is that on $L_0(T)$ or $L'_0(T)$ respectively. 
\end{prop}
\begin{proof}
\eqref{eq:gluingslmetricdecay} follows by an argument similar to that of Proposition \ref{transfermapsoneforms}; \eqref{eq:Deltadecexp} then follows from smoothness of the Laplacian and Hodge star written in local coordinates as functions of the metric. As for \eqref{eq:laplacelowerbounddsling}, as in the proof of Theorem \ref{slperturbthm} this is an immediate consequence of Theorem \ref{laplacelowerbound} if the metric is close to the glued metric. For $L_0(T)$ we know this by Corollary \ref{newgluedmetricsallthesame}; for $L'_0(T)$, we use \eqref{eq:gluingslmetricdecay}. 
\end{proof}

Proposition \ref{hptslgluinganalysis} has the following two corollaries. Firstly, it implies that the harmonic parts of a form taken on $L_0(T)$ and $L'_0(T)$ are similar. 
\begin{cor}
\label{cor:harmonicestimates}
Let $M_1$, $M_2$, $L_1$, $L_2$, $M^T$, $L_0(T)$, $L'_0(T)$, $v_0(T)$ be as in Convention \ref{con:gluingslanalysis}, and $\epsilon$ as in Lemma \ref{opennessdecaylemma}. Given a form $\alpha$ on $L_0(T) = L'_0(T)$, write $\hpt_{L_0(T)}$ and $\hpt_{L'_0(T)}$ for its harmonic part with respect to the two metrics of Proposition \ref{hptslgluinganalysis}. Then for every $k$ and $\mu$ there exist constants $C_{k, \mu}$ such that for every form $\alpha$, 
\begin{equation}
\|\hpt_{L_0(T)} \alpha - \hpt_{L'_0(T)} \alpha\|_{C^{k, \mu}} \leq C_{k, \mu} e^{-\epsilon T} \|\alpha\|_{C^{k, \mu}}.
\end{equation}
Furthermore, we have for some constants $r$ and $C_{k, \mu}$ 
\begin{equation}
\label{eq:dslinghptbound}
\|\hpt_{L'_0(T)} \alpha\|_{C^{k, \mu}} \leq C_{k, \mu} T^r \|\alpha\|_{C^{k, \mu}}.
\end{equation}
\end{cor}
The first follows straightforwardly as the Laplacians converge together exponentially and can be lower bounded polynomially, so that the orthoharmonic part on $L'_0(T)$ of a form harmonic on $L_0(T)$ must decay exponentially. As for the second part, Theorem \ref{laplacelowerbound} shows that we can bound the orthoharmonic part polynomially in terms of the Laplacian, and the Laplacian can obviously be uniformly bounded, provided that the metric is not too far from the glued metric: this follows from the metric comparison results of \eqref{eq:gluingslmetricdecay} and Corollary \ref{newgluedmetricsallthesame}. 

Secondly, combining Proposition \ref{hptslgluinganalysis} with Proposition \ref{transfermapsoneforms} implies that the transfer maps induce isomorphisms between harmonic normal vector fields. 
\begin{cor}
\label{cor:hpttransferisomorphism}
Let $M_1$, $M_2$, $L_1$, $L_2$, $M^T$, $L_0(T)$, $L'_0(T)$ and $v_0(T)$ be as in Convention \ref{con:gluingslanalysis}.  Let $T_i$ be either $T_1$ or $T_2$ from Definition \ref{defin:sectiontransfers}. For $T$ sufficiently large, $\hpt_{L'_0(T)} \circ T_i$ defines an isomorphism between harmonic normal vector fields. So does $\hpt_{L_0(T)} \circ T_i^{-1}$. 
\end{cor}
\begin{proof}
That $\hpt_{L'_0(T)}\circ T_i$ is an isomorphism is equivalent to $\harmpt_{L'_0(T)} \circ \tau_i$ being an isomorphism. We shall show that $\harmpt_{L'_0(T)}\circ\tau_i$ is injective on one-forms: since $L_0(T)$ and $L'_0(T)$ have the same first Betti number, it will follow that $\hpt_{L'_0(T)} \circ \tau_i$ is an isomorphism. 

Pick some $k$ and $\mu$, and suppose that $\alpha$ is a harmonic one-form on $L_0(T)$ with $\|\alpha\|_{C^{k, \mu}} = 1$. By Proposition \ref{transfermapsoneforms}, we see that $\|\tau_i \alpha - \alpha\|_{C^{k, \mu}}$ is bounded by a constant decaying exponentially in $T$.  By the bound \eqref{eq:dslinghptbound}, we obtain that $\|\harmpt_{L'_0(T)}\tau_i \alpha - \harmpt_{L'_0(T)}\alpha\|_{C^{k, \mu}}$ is similarly bounded. By Corollary \ref{cor:harmonicestimates}, we also have that $\|\harmpt_{L'_0(T)} \alpha -  \alpha\|_{C^{k, \mu}}$ can also be bounded by a constant decaying exponentially in $T$. Hence we obtain that $\|\harmpt_{L'_0(T)}\tau_i \alpha - \alpha\|_{C^{k, \mu}}$ decays exponentially in $T$, and it follows that $\harmpt_{L'_0(T)}\tau_i \alpha$ cannot be zero. 

That $\hpt_{L_0(T)}\circ T_i^{-1}$ is an isomorphism follows in exactly the same way.
\end{proof}
In particular, for $T$ sufficiently large, combining Corollary \ref{cor:hpttransferisomorphism} with Proposition \ref{slingderivative} shows that $\SLing$ is a smooth map on a small neighbourhood $\U$ of $L_0(T)$ with derivative at $L_0(T)$ given by Proposition \ref{slingderivative}. 
\begin{rmk}
Slightly restricted versions of Propositions \ref{transfermapsoneforms} and \ref{hptslgluinganalysis} also hold in a more general setting. If $M$ is a Calabi--Yau manifold and $L_0$ is a closed nearly special Lagrangian submanifold satisfying Condition \ref{slingcondition}, then $\SLing(L_0)$ exists. If $\SLing(L_0)$ is close enough to $L_0$ in $C^k$, then working with a suitably low regularity in the propositions shows again that $\hpt_{L_0} \tau_i^{-1}$ is an isomorphism. Then, as in the gluing case, Proposition \ref{slingderivative} applies and we have that $\SLing$ is again a smooth map with derivative given by  \eqref{eq:DSLing} on a small enough space $\U$ of perturbations of $L_0$. 
\end{rmk}

We now know that $D_{L_0(T)} \SLing$ exists and in terms of one-forms is given by \eqref{eq:oneformdsling}. We now show that this map is close to taking the harmonic part $\hpt_{L'_0(T)}$. Proposition \ref{transfermapsoneforms} controls the $\tau_i$; it remains to deal with the term $\Delta'_\alpha x^T$. We have
\begin{prop}
\label{deltaprimeestimate}
Let $M_1$, $M_2$, $L_1$, $L_2$, $M^T$, $L_0(T)$, $L'_0(T)$, $v_0(T)$ and $x^T$ be as in Convention \ref{con:gluingslanalysis}, and $\epsilon$ as in Lemma \ref{opennessdecaylemma}. Consider the map of one-forms 
\begin{equation}
\alpha \mapsto \Delta'_\alpha x^T,
\end{equation}
defined after \eqref{eq:oneformdsling}. This is a linear map of $\alpha$ and for every $k$ there are $C_k$ and $C'_k$, independent of $T$, such that
\begin{equation}
\label{eq:deltaprimeestimate}
\|\Delta'_\alpha x^T\|_{C^k} \leq C_k \|\alpha\|_{C^{k+1}} \|x^T\|_{C^{k+2}} \leq C'_k e^{-\epsilon T} \|\alpha\|_{C^{k+1}}.
\end{equation}
\end{prop}
\begin{proof}
$\Delta'_\alpha x^T$ is the derivative at zero of the map $\Delta_\alpha x^T$ given by finding the normal vector field $v$ corresponding to $\alpha$, the corresponding transfer operator $T_2$, and then evaluating $T_2^{-1} \Delta T_2 x^T$. Since this is a smooth map by the argument of Proposition \ref{revised:hptsmoothinsubmanifold}, it follows immediately that the derivative is well-defined and in particular linear. 

The proof of the bound is similar to Proposition \ref{transfermapsoneforms}, though is more involved. We argue that the map $\Delta_\alpha x^T$ can be determined locally, as for the $\tau_i$, by the first jet of $\alpha$, the second jet of $x^T$, and the tangent space to $L$ at the point, and is continuous in the Calabi--Yau structure in the same sense as Proposition \ref{tauismoothness}. Hence, the derivative $\Delta'_\alpha x^T$ can also be determined by the first jet of the derivative, the second jet of $x^T$, and the tangent space and is continuous in the Calabi--Yau structure in the sense of \eqref{eq:oldcomplicatedsctyest}. But then the bound follows by exactly the same argument as in Proposition \ref{transfermapsoneforms}: the argument of Proposition \ref{bundlemapsckbounded} gives it locally, and the uniform constant follows by comparing with $(\Omega_1, \omega_1)$, $(\Omega_2, \omega_2)$ and the cylindrical Calabi--Yau structure $(\tilde \Omega, \tilde \omega)$.

The second inequality follows immediately from Lemma \ref{opennessdecaylemma}. 
\end{proof}

We can now turn to our main result identifying $D_{L_0(T)}\SLing$ with $\harmpt_{L'_0(T)}$ on one-forms. 
\begin{prop}
\label{dslingisnearlyharmpt}
Let $M_1$, $M_2$, $L_1$, $L_2$, $M^T$, $L_0(T)$, $L'_0(T)$ be as in Convention \ref{con:gluingslanalysis}, and $\epsilon$ as in Lemma \ref{opennessdecaylemma}. Given a one-form $\alpha$ on $L_0(T)$ we can construct a harmonic one-form with respect to the metric on $L'_0(T)$ by applying either the harmonic projection $\hpt_{L'_0(T)}$ or $D_{L_0(T)}\SLing$. We have constants $C_k$ such that 
\begin{equation}
\|(\hpt_{L'_0(T)} - D_{L_0(T)}\SLing) \alpha\|_{C^k} \leq C_k e^{-\epsilon T} \|\alpha\|_{C^{k+1}}.
\end{equation}
\end{prop}
\begin{proof}
We recall from \eqref{eq:oneformdsling} that $D_{L_0(T)}\SLing (\alpha)$ is the unique $L'_0(T)$-harmonic $\beta$ such that 
\begin{equation}
\harmpt_{L_0(T)} \tau_1^{-1} \beta = \harmpt_{L_0(T)} (\tau_1^{-1}(\tau_2 \alpha + \tau_3 \alpha) + \Delta'_\alpha x^T).
\end{equation}
By applying Proposition \ref{transfermapsoneforms} repeatedly and Proposition \ref{deltaprimeestimate}, we find that there exist $C_k$ such that 
\begin{equation}
\|\tau_1^{-1}(\tau_2 \alpha + \tau_3 \alpha) + \Delta'_\alpha x^T - \alpha\|_{C^k} \leq C_k e^{-\epsilon T} \|\alpha\|_{C^{k+1}}.
\end{equation}
By applying $\harmpt_{L_0(T)}$, which is bounded at most polynomially in $T$ by Proposition \ref{hptslgluinganalysis}, and increasing $C_k$ if necessary, it follows that 
\begin{equation}
\label{eq:acomparisonestimate}
\|\harmpt_{L_0(T)}\tau_1^{-1} D_{L_0(T)} \SLing (\alpha) - \harmpt_{L_0(T)} \alpha\| \leq C_k e^{-\epsilon T} \|\alpha\|_{C^{k+1}}.
\end{equation}
Applying Proposition \ref{transfermapsoneforms} to the difference $\tau_1^{-1} D_{L_0(T)}\SLing(\alpha) - D_{L_0(T)} \SLing(\alpha)$, and combining the resulting estimate with \eqref{eq:acomparisonestimate} and the polynomial bound on $\hpt_{L_0(T)}$ again we find another $C_k$ such that 
\begin{equation}
\begin{split}
&\|\harmpt_{L_0(T)} D_{L_0(T)} \SLing(\alpha) - \harmpt_{L_0(T)} \alpha\| \\\leq& C_k e^{-\epsilon T} (\|\alpha\|_{C^{k+1}} + \|D_{L_0(T)} \SLing(\alpha)\|_{C^k}).
\end{split}
\end{equation}
But then by Corollary \ref{cor:harmonicestimates} and the fact that $D_{L_0(T)} \SLing(\alpha)$ is $L'_0(T)$-harmonic, we obtain for yet another $C_k$
\begin{equation}
\label{eq:anothercomparisonestimate}
\|D_{L_0(T)} \SLing(\alpha) - \harmpt_{L'_0(T)} \alpha\|_{C^k} \leq C_k e^{-\epsilon T} (\|\alpha\|_{C^{k+1}} + \|D_{L_0(T)} \SLing(\alpha)\|_{C^k}).
\end{equation}
We can then simply use the triangle inequality
\begin{equation}
\|D_{L_0(T)} \SLing(\alpha)\|_{C^k} \leq\|\harmpt_{L'_0(T)} \alpha \|_{C^k} + \|D_{L_0(T)} \SLing(\alpha) - \harmpt_{L'_0(T)} \alpha\|_{C^k}
\end{equation}
to bound the right hand side of \eqref{eq:anothercomparisonestimate}. The $\|D_{L_0(T)} \SLing(\alpha) - \harmpt_{L'_0(T)} \alpha\|_{C^k}$ term can be absorbed on the left hand side for $T$ large enough; the $\|\harmpt_{L'_0(T)} \alpha \|_{C^k}$ term is polynomially bounded by $\|\alpha\|_{C^{k+1}}$ using Corollary \ref{cor:harmonicestimates} again. Thus we get the required estimate. 
\end{proof}

We now turn to the approximate gluing map. We know from subsection \ref{ssec:dpatchingproper} that on any manifold of matching pairs this is smooth and its derivative is exponentially close in $T$ to the approximate gluing map of one-forms under the identification. 

From the asymptotically cylindrical deformation Theorem \ref{acyldeformationtheorem}, we infer that the set of matching special Lagrangians is a finite-dimensional manifold. 
\begin{prop}
\label{manifoldofgluablesl}
Let $M_1$ and $M_2$ be a matching pair of asymptotically cylindrical Calabi--Yau manifolds, and let $L_1$ and $L_2$ be a matching pair of asymptotically cylindrical special Lagrangian submanifolds as in Definition \ref{defin:newsubmanifoldapproxgluing}. We have a set of pairs $(L'_1, L'_2)$ where $L'_1$ is a special Lagrangian deformation of $L_1$ and $L'_2$ is a special Lagrangian deformation of $L_2$. The subset of $(L'_1, L'_2)$ such that $L'_1$ and $L'_2$ also match is a manifold. The tangent space is just the space of matching normal vector fields corresponding to matching bounded harmonic forms. 
\end{prop}
\begin{proof}
The map from  $L'_i$ to its cross-section $K'_i$ is evidently smooth. By Proposition \ref{speclaglimitdeformationsubmanifold}, which follows \cite[Proposition 6.4]{salurtodd}, $K'_i$ lies in a submanifold $\mathcal{K}_i$ of possible limits. The derivative of $L'_i \mapsto K'_i$ is the map $v_i \mapsto \tilde v_i$ given by taking the limit of a harmonic normal vector field such that $\iota_v \omega|_{L_i}$ is a harmonic one-form. That is, the derivative is the limit map from harmonic normal vector fields to the limits of these, and hence is surjective. 

For notational simplicity, write $K_1 = K_2$: that is, use $F$ to identify the two cross-sections. 

Let $\partial_i: H^1(K) \to H^2_\rel(L_i)$ be the coboundary map of the exact sequence needed in Proposition \ref{speclaglimitdeformationsubmanifold} for the asymptotically cylindrical manifold $L_i$. Let $\tilde \tau_i$ be the limit of an asymptotically translation invariant form $\tau_i$ on a tubular neighbourhood of $L_i$ with $\tau_i|_{L_i} = 0$ and $d\tau_i = \omega$ on this tubular neighbourhood. Note that both $\tau_1$ and $\tau_2$ are examples of the form which we called $\tau_1$ in Proposition \ref{speclaglimitdeformationsubmanifold}, with $\tau_1$ on $L_1$ and $\tau_2$ on $L_2$. By Proposition \ref{speclaglimitdeformationsubmanifold}, if $L_i$ is connected, $\K_i$ is the kernel of the map
\begin{equation}
\label{eq:limitdeformationmap}
K'_i \mapsto \partial_i([\tilde\tau_i|_{K'_i}]).
\end{equation}
on special Lagrangian deformations $K'_i = \bigcup \{(p^j)'_i\} \times (Y^j)'_i$ of $K = \bigcup \{p^j\} \times Y^j$ satisfying 
\begin{equation}
\label{eq:volumeconstraint}
\sum ((p^j)'_i - p^j) \Vol(Y_j) = 0
\end{equation}
If $L_i$ is not connected, we simply get additional constraints of the form \eqref{eq:volumeconstraint} corresponding to each connected component. 

We may simply apply all of these linear constraints to get a linear subspace of the required normal vector fields, since we are working with the cylindrical limit. Then to prove the result, we just have to prove that the joint kernel of the two maps \eqref{eq:limitdeformationmap} is a manifold, restricted to this space. We first want to show that the two maps are as similar as possible. 

We note that $d\tilde \tau_1 -d\tilde \tau_2 = \tilde\omega - \tilde \omega =0$. Consequently, for any deformation $K'$ of $K$, we have that $[(\tilde \tau_1 - \tilde\tau_2)|_{K'}] = [(\tilde \tau_1 - \tilde \tau_2)|_{K}] = 0$. Since we shall only need these restricted cohomology classes, we shall write $\tilde \tau = \tilde \tau_i$. 

We know, as a consequence of the calculation of its linearisation in Proposition \ref{speclaglimitdeformationsubmanifold}, that the map from $K'$ to $[\tilde \tau|_{K'}]$ is a submersion into $H^1(K)$, even with the further restriction above. Consequently, as $\ker \partial_1 \cap \ker \partial_2$ is a vector subspace of $H^1(K)$, we have that $\K_1 \cap \K_2$ is a submanifold of the manifold of all special Lagrangian deformations of $K$.

Consequently, the diagonal space $\{(K', K'): K' \in \K_1 \cap \K_2\}$ is a submanifold of $\K_1 \times \K_2$. The desired submanifold of matching deformations of $(L_1, L_2)$ is precisely the inverse image of this submanifold under the submersion $(L'_1, L'_2) \mapsto (K'_1, K'_2)$. 

As for the tangent space, the last part shows that it is the inverse image of the tangent space of  $\{(K', K'): K' \in \K_1 \cap \K_2\}$ consisting of pairs of harmonic normal vector fields on $K$, under the projection map; that is, it is all pairs of matching harmonic normal vector fields whose limit lies in the tangent space of $\K_1 \cap \K_2$. This is the intersection of the tangent spaces of $\K_1$ and $\K_2$, that is the harmonic normal vector fields on $K$ which arise as limits of harmonic normal vector fields on both $L_1$ and $L_2$. Thus we have all matching pairs of harmonic normal vector fields, as required. 
\end{proof}

Consequently, we have
\begin{prop}
If we restrict the gluing map of submanifolds given by Definition \ref{defin:newsubmanifoldapproxgluing} to pairs of matching asymptotically cylindrical special Lagrangian submanifolds, for $T$ large enough, the approximate gluing map is a smooth immersion and so maps to a finite-dimensional submanifold of the deformations of $L^T$. 
\end{prop}
\begin{proof}
Subsection \ref{ssec:dpatchingproper} shows that the derivative of this gluing map is close to the approximate gluing map of one-forms for $T$ large enough. By the argument of Proposition \ref{harmonicgluinglowerbound} the approximate gluing map of one-forms is injective and bounded below independently of $T$ when restricted to the matching harmonic forms, and so this derivative must also be injective. 
\end{proof}

We may now combine Proposition \ref{dslingisnearlyharmpt} with Proposition \ref{harmonicgluinglowerbound} (on the harmonic gluing map) to prove that the gluing map of special Lagrangians is a local diffeomorphism of moduli spaces. 
\begin{thm}[Theorem B]
\label{speclaggluinglocaldiffeo}
Let $M_1$ and $M_2$ be a matching pair of asymptotically cylindrical Calabi--Yau manifolds, and suppose that Hypothesis \ref{hyp:ambientgluing} holds so $M_1$ and $M_2$ can be glued to give a Calabi--Yau manifold $M^T$. Let $L_1$ and $L_2$ be a matching pair of asymptotically cylindrical special Lagrangians in $M_1$ and $M_2$. By Theorem A (Theorem \ref{slperturbthm}), there exists $T_0>0$ such that $L_1$ and $L_2$ can be glued to a form a special Lagrangian in $M^T$ for all $T>T_0$. Moreover, this applies for any sufficiently small deformation of $L_1$ and $L_2$ as a matching pair, and hence we obtain a gluing map from the deformation space of matching pairs of submanifolds in Proposition \ref{manifoldofgluablesl} to the space of deformations of the gluing of $L_1$ and $L_2$. This map is a local diffeomorphism for $T$ sufficiently large.
\end{thm}
\begin{proof}
This gluing map is the composition of the approximate gluing map of Definition \ref{defin:newsubmanifoldapproxgluing} with the SLing map of Definition \ref{defin:slingmap}. We shall show that the derivative of this gluing map, regarded as a map of one-forms, is exponentially close to the gluing map $\Gamma_T$ of harmonic one-forms described in Proposition \ref{harmonicgluinglowerbound}; this derivative is the composition of the derivatives of the approximate gluing and $\SLing$. For the proof we use the notation of Convention \ref{con:gluingslanalysis}. 

As in the previous proposition, by subsection \ref{ssec:dpatchingproper}, the derivative of the approximate gluing map is exponentially close to the approximate gluing map of one-forms of Definition \ref{defin:newformgluing}.

By Proposition \ref{dslingisnearlyharmpt},  $D_{L_0(T)}\SLing$ is exponentially close to $\harmpt_{\SLing(L_0(T))}$. In particular, it is bounded polynomially in $T$, by Corollary \ref{cor:harmonicestimates}. 

By composition (using that the harmonic part map is polynomially bounded, so that this exponential decay is preserved) we find that the difference between the derivative of the gluing map of special Lagrangians and the gluing map $\Gamma_T$ of harmonic forms with respect to the metric $g(\Omega^T, \omega^T)|_{L'_0(T)}$ is exponentially small in $T$. Combining Corollary \ref{newgluedmetricsallthesame} with Proposition \ref{hptslgluinganalysis}, this metric is exponentially close in $T$ to that given by Definition \ref{defin:newgluingmanifolds}, so, since by Proposition \ref{harmonicgluinglowerbound} $\Gamma_T$ is bounded below uniformly in $T$ and is an isomorphism, it follows that the derivative of the gluing map of special Lagrangians is an isomorphism for $T$ sufficiently large. By the inverse function theorem, the result follows. 
\end{proof}
\bibliographystyle{abbrv}
\bibliography{bibfile.bib}
\end{document}